\RequirePackage[l2tabu, orthodox]{nag}

\documentclass[11pt]{amsart}
\usepackage{stmaryrd}
\usepackage{amsmath}
\usepackage{amsfonts}
\usepackage{amsthm}
\usepackage{mathtools}
\usepackage{txfonts}
\usepackage{hyperref}
\usepackage{graphicx,amssymb,amsfonts,amsmath,amscd}
\usepackage{textcomp}
\usepackage[all,ps,cmtip]{xy}
\usepackage{amscd}
\usepackage{xspace}
\usepackage{mathrsfs}
\usepackage[varg]{pxfonts}
\usepackage[usenames,dvipsnames]{xcolor}
\hypersetup{colorlinks=true,citecolor=NavyBlue,linkcolor=Brown,urlcolor=Orange}

\usepackage{dsfont}
\usepackage{xcolor}
\usepackage{tikz}\usetikzlibrary{decorations.markings,matrix,arrows}

\newtheorem{thm}{Theorem}[section]
\newtheorem{prop}[thm]{Proposition}
\newtheorem{lemma}[thm]{Lemma}
\newtheorem{cor}[thm]{Corollary}
\newtheorem{conj}[thm]{Conjecture}
\newtheorem{meta-conj}[thm]{Meta-conjecture}

\newtheorem{def-prop}[thm]{Definition-Proposition}
\newtheorem{prop-def}[thm]{Proposition-Definition}

\theoremstyle{definition} 
\newtheorem{defi}[thm]{Definition}
\newtheorem{rmk}[thm]{Remark}
\newtheorem{rmks}[thm]{Remarks}
\newtheorem{expl}[thm]{Example}
\newtheorem{expls}[thm]{Examples}

\newcommand{\C}{{\textbf C}}

\newcommand{\Z}{{\textbf Z}}
\newcommand{\N}{{\textbf N}}
\newcommand{\Q}{{\textbf Q}}

\newcommand{\sA}{{\mathscr A}}

\newcommand{\sO}{{\mathscr O}}
\newcommand{\sP}{{\mathscr P}}

\newcommand{\rM}{{\mathcal M}}

\newcommand{\rX}{{\mathcal X}}

\newcommand{\ie}{\textit {i.e.}~}
\newcommand{\cf}{\textit {cf.}~}
\newcommand{\vs}{\textit {vs.}~}
\newcommand{\loccit}{\textit {loc.cit.}~}

\newcommand{\apriori}{\textit{a priori} }

\newcommand{\etc}{\textit {etc.}}
\newcommand{\resp}{\textit {resp.}~}

\newcommand{\CH}{\mathop{\rm CH}\nolimits} 
\newcommand{\ch}{\mathop{\rm ch}\nolimits} 
\newcommand{\cl}{\mathop{\rm cl}\nolimits} 
\newcommand{\codim}{\mathop{\rm codim}\nolimits} 
\newcommand{\Coker}{\mathop{\rm Coker}\nolimits} 
\newcommand{\dual}{\mathop{^\vee}\nolimits} 
\newcommand{\End}{\mathop{\rm End}\nolimits} 
\newcommand{\Hilb}{\mathop{\rm Hilb}\nolimits}
\newcommand{\Hom}{\mathop{\rm Hom}\nolimits}
\newcommand{\id}{\mathop{\rm id}\nolimits} 
\newcommand{\im}{\mathop{\rm Im}\nolimits} 
\newcommand{\Ker}{\mathop{\rm Ker}\nolimits} 
\newcommand{\Obj}{\mathop{\rm Obj} \nolimits} 
\newcommand{\pr}{\mathop{\rm pr}\nolimits} 

\newcommand{\rk}{\mathop{\rm rk}\nolimits}
\newcommand{\Spec}{\mathop{\rm Spec}\nolimits}
\newcommand{\Sym}{\mathop{\rm Sym}\nolimits} 
\newcommand{\Tor}{\mathop{\rm Tor}\nolimits} 

\newcommand{\GL}{\mathop{\rm GL}\nolimits}

\newcommand{\SL}{\mathop{\rm SL}\nolimits}

\renewcommand{\bar}{\overline}
\newcommand{\inj}{\hookrightarrow}
\newcommand{\surj}{\twoheadrightarrow}

\newcommand{\lra}{\xrightarrow}

\newcommand{\cart}{\ar@{}[dr]|\square} 
\renewcommand{\dual}{^{\vee}} 
\newcommand{\isom}{\simeq} 

\renewcommand{\tilde}{\widetilde}

\newcommand{\deff}{\mathrel{:=}}

\newcommand{\age }{\mathop{\rm {age}}\nolimits}
\newcommand{\sAV }{\mathop{\mathscr {AV}}\nolimits}
\newcommand{\h}{\mathop{\mathfrak {h}}\nolimits} 
\newcommand{\CHM }{\mathop{\rm {CHM}}\nolimits}
\renewcommand{\1}{\mathop{\mathds{1}}\nolimits} 
\newcommand{\p}{\mathop{\mathfrak {p}}\nolimits} 

\newcommand{\Smproj }{\mathop{\rm {SmProj}}\nolimits}

\newcommand{\gS }{\mathop{\mathfrak{S}}\nolimits} 
\newcommand{\atts}{\textit {a.t.t.s.}~}

\begin{document}
	
	\title[Motivic hyperK\"ahler resolution conjecture for generalized Kummer
	varieties]{Motivic hyperK\"ahler resolution conjecture\,: I. \\
		Generalized Kummer varieties}
	\author{Lie Fu}
	\address{Universit\'e Claude Bernard Lyon 1, France}
	\email{fu@math.univ-lyon1.fr}
	
	\author{Zhiyu Tian}
	\address{Beijing International Center for Mathematical Research, Peking University, Beijing, 100871, China}
	\email{zhiyutian@bicmr.pku.edu.cn}
	
	\author{Charles Vial}
	\address{Universit\"at Bielefeld, Germany}
	\email{vial@math.uni-bielefeld.de}
	
	\thanks{Lie Fu is supported by the Agence Nationale de la Recherche (ANR)
		through ECOVA (ANR-15-CE40-0002) and LABEX MILYON (ANR-10-LABEX-0070). Lie Fu
		and Zhiyu Tian are supported by ANR HodgeFun (ANR-16-CE40-0011) and \emph{Projet
			Exploratoire Premier Soutien} (PEPS) Jeunes chercheur-e-s 2016 operated by Insmi
		and  \emph{Projet Inter-Laboratoire} 2016, 2017 and 2018 by F\'ed\'eration de
		Recherche en Math\'ematiques Rh\^one-Alpes/Auvergne CNRS 3490. Charles Vial was
		supported by  EPSRC Early Career Fellowship EP/K005545/1.}

	\begin{abstract}
		Given a smooth projective variety $M$ endowed with a faithful action of a finite
		group $G$, following Jarvis--Kaufmann--Kimura \cite{MR2285746} and
		Fantechi--G\"ottsche \cite{MR1971293}, we define the orbifold motive (or
		Chen--Ruan motive) of the quotient stack $[M/G]$ as an algebra object in the
		category of Chow motives. Inspired by Ruan \cite{MR1941583}, one can formulate a
		motivic version of his  Cohomological HyperK\"ahler Resolution Conjecture
		(CHRC). We prove this motivic version, as well as its K-theoretic analogue
		conjectured in \cite{MR2285746}, in two situations related to an abelian surface
		$A$ and a positive integer $n$. Case (A) concerns Hilbert schemes of points of
		$A$\,: the Chow motive of $A^{[n]}$ is isomorphic as algebra objects, up to a
		suitable sign change, to the orbifold motive of the quotient stack
		$[A^{n}/\mathfrak{S}_{n}]$. Case (B) for generalized Kummer varieties\,: the
		Chow motive of the generalized Kummer variety $K_n(A)$ is isomorphic as algebra
		objects, up to a suitable sign change, to the orbifold motive of the quotient
		stack $[A_{0}^{n+1}/\mathfrak {S}_{n+1}]$, where $A_{0}^{n+1}$ is the kernel
		abelian variety of the summation map $A^{n+1}\to A$. 
		As a byproduct, we prove the original Cohomological HyperK\"ahler Resolution
		Conjecture for generalized Kummer varieties. As an application, we provide
		multiplicative Chow--K\"unneth decompositions for Hilbert schemes of abelian
		surfaces and for generalized Kummer varieties. In particular, we have a
		multiplicative direct sum decomposition of their Chow rings with rational
		coefficients, which is expected to be the splitting of the conjectural
		Bloch--Beilinson--Murre filtration. The existence of such a splitting for
		holomorphic symplectic varieties is conjectured by Beauville \cite{MR2187148}.
		Finally, as another application, we prove that over a non-empty Zariski open
		subset of the base, there exists a decomposition isomorphism $R\pi_{*}\Q\isom
		\oplus R^{i}\pi_{*}\Q[-i]$ in $D^{b}_{c}(B)$ which is compatible with the
		cup-products on both sides, where $\pi: \mathcal K_{n}(\mathcal A)\to B$ is the
		relative generalized Kummer variety associated to a (smooth) family of abelian
		surfaces $\mathcal A\to B$.
	\end{abstract}
	
	\maketitle
	\setcounter{tocdepth}{1}
	\vspace{-33pt}
	\tableofcontents

	\section{Introduction}\label{sect:intro}
	\subsection{Motivation 1\,: Ruan's hyperK\"ahler resolution conjectures}
	In \cite{MR2104605}, Chen and Ruan construct the orbifold cohomology ring
	$H_{orb}^*(\rX)$ for any complex orbifold $\rX$. As a $\Q$-vector space, it is
	defined to be the cohomology of its inertia variety $H^*(I\rX)$ (with degrees
	shifted by some rational numbers called \emph{age}), but is endowed with a
	highly non-trivial ring structure coming from moduli spaces of curves mapping to
	$X$. An algebro-geometric treatment is contained in
	Abramovich--Graber--Vistoli's work \cite{MR2450211}, based on the construction
	of moduli stack of twisted stable maps in \cite{MR1862797}. In the global
	quotient case\footnote{In this paper, by `global quotient', we always mean the
		quotient of a smooth projective variety by a \emph{finite} group.}, some
	equivalent definitions are available\,: see for example \cite{MR1971293},
	\cite{MR2285746}, \cite{MR2428359} and \S \ref{sec:OrbDef}.
	
	Originating from the topological string theory of orbifolds in \cite{MR818423},
	\cite{MR851703}, one observes that the stringy topological invariants of an
	orbifold, \emph{e.g.} the orbifold Euler number and the orbifold Hodge numbers,
	should be related to the corresponding invariants of a crepant resolution
	(\cite{MR1672108}, \cite{MR1404917}, \cite{MR2027195}, \cite{MR2069013}). A much
	deeper relation was brought forward by Ruan, who made, among others, the
	following Cohomological HyperK\"ahler Resolution Conjecture (CHRC) in
	\cite{MR1941583}. For more general and sophisticated versions of this
	conjecture, see \cite{MR2234886}, \cite{MR2483931}, \cite{MR3112518}.
	\begin{conj}[Ruan's CHRC]\label{conj:CHRC}
		Let $\rX$ be a compact complex orbifold with underlying variety $X$ being
		Gorenstein. If there is a crepant resolution $Y\to X$ with $Y$ being
		hyperK\"ahler, then we have an isomorphism of graded commutative
		$\C$-algebras\,: $H^*(Y, \C)\isom H^*_{orb}(\rX,\C)$.
	\end{conj}
	
	As the construction of the orbifold product can be expressed using algebraic
	correspondences (\cf \cite{MR2450211} and \S \ref{sec:OrbDef}), one has the
	analogous definition of the orbifold Chow ring $\CH_{orb}(\rX)$ (see Definition
	\ref{def:OrbChow} for the global quotient case)
	of a smooth proper Deligne--Mumford stack $\rX$. Motivated by the study of
	algebraic cycles on hyperK\"ahler varieties, we propose to investigate the
	Chow-theoretic analogue of Conjecture \ref{conj:CHRC}. For reasons which will
	become clear shortly, it is more powerful and fundamental to consider the
	following \emph{motivic} version of Conjecture \ref{conj:CHRC}. Let $\CHM_{\C}$
	be the category of Chow motives with complex coefficients and $\h$ be the
	(contravariant) functor that associates to a smooth projective variety its Chow
	motive. 
	
	\begin{meta-conj}[MHRC]\label{conj:MHRC}
		Let $\rX$ be a smooth proper complex Deligne--Mumford stack with underlying
		coarse moduli space $X$ being a (singular) symplectic variety. If there is a
		symplectic resolution $Y\to X$, then we have an isomorphism $\h(Y)\isom
		\h_{orb}(\rX)$ of commutative algebra objects in $\CHM_{\C}$, hence in
		particular an isomorphism of graded $\C$-algebras\,: $\CH^{*}(Y)_{\C}\isom
		\CH_{orb}^{*}(\rX)_{\C}$.
	\end{meta-conj}
	See Definition \ref{def:SympSing} for generalities on symplectic singularities
	and symplectic resolutions. The reason why it is only a \emph{meta-}conjecture
	is that the definition of orbifold Chow motive for a smooth proper
	Deligne--Mumford stack in general is not available in the literature and we will
	not develop the theory in this generality in this paper (see however Remark
	\ref{rmk:OrbMotDM}). From now on, let us restrict ourselves to the case where
	the Deligne--Mumford stack in question is of the form of a global quotient
	$\rX=[M/G]$, where $M$ is a smooth projective variety with a faithful action of
	a finite group $G$, in which case we will define the \emph{orbifold Chow motive}
	$\h_{orb}(\rX)$ in a very explicit way in Definition \ref{def:OrbMot}.
	
	The \emph{Motivic HyperK\"ahler Resolution Conjecture} that we are interested in
	is the following more precise statement, which would contain all situations
	considered in this paper and its sequel.
	\begin{conj}[MHRC\,: global quotient case]\label{conj:MHRCbis}
		Let $M$ be a smooth projective holomorphic symplectic variety equipped with a faithful action of a finite group $G$ by symplectic automorphisms of $M$. If $Y$ is a
		symplectic resolution of the quotient variety $M/G$, then we have an isomorphism
		of (commutative) algebra objects in the category  of Chow motives with
		\emph{complex} coefficients\,: $$\h(Y)\isom \h_{orb}\left([M/G]\right) \text{ in
		} \CHM_{\C}.$$ In particular, we have an isomorphism of graded $\C$-algebras
		$$\CH^{*}(Y)_{\C}\isom \CH_{orb}^{*}([M/G])_{\C}.$$
	\end{conj}
	
	The definition of the \emph{orbifold motive} of $[M/G]$ as a \emph{(commutative)
		algebra object} in the category of Chow motives with rational
	coefficients\footnote{Strictly speaking, the orbifold Chow motive of $[M/G]$ in
		general lives in the larger category of Chow motives with \emph{fractional} Tate
		twists. However, in our cases of interest, namely when there exists a crepant
		resolution, for the word `crepant resolution' to make sense we understand that
		the underlying singular variety $M/G$ is at least Gorenstein, in which case all
		age shiftings are integers and we stay in the usual category of Chow motives.
		See Definitions \ref{def:Mot} and \ref{def:OrbMot} for the general notions.} is
	particularly down-to-earth\,; it is the $G$-invariant sub-algebra object of some
	explicit algebra object\,:
	$$\h_{orb}\left([M/G]\right) \deff \left(\bigoplus_{g\in G}
	\h(M^{g})\left(-\age(g)\right), \star_{orb}\right)^{G},$$ 
	where for each $g\in G$, $M^{g}$ is the subvariety of fixed points of $g$ and the orbifold
	product $\star_{orb}$ is defined by using natural inclusions and Chern classes
	of normal bundles of various fixed loci\,; see Definition \ref{def:OrbMot} (or
	(\ref{eqn:OrbChowProdIntro}) below) for the precise formula of $\star_{orb}$ as
	well as the Tate twists by age (Definition~\ref{def:age}) and the $G$-action.
	The orbifold Chow ring\footnote{The definition of the orbifold Chow ring has
		already appeared in Page 211 of Fantechi--G\"ottsche \cite{MR1971293} and proved
		to be equivalent to Abramovich--Grabber--Vistoli's construction in
		\cite{MR2450211} by Jarvis--Kaufmann--Kimura in \cite{MR2285746} .} is then
	defined as the following commutative algebra
	$$\CH^{*}_{orb}([M/G]):=\bigoplus_{i} \Hom_{\CHM}(\1(-i), \h_{orb}([M/G])),$$
	or, equivalently and more explicitly,
	\begin{equation}\label{eqn:OrbChowDefi}
	\CH^{*}_{orb}\left([M/G]\right) \deff \left(\bigoplus_{g\in G}
	\CH^{*-\age(g)}(M^{g}), \star_{orb}\right)^{G},
	\end{equation}
	where $\star_{orb}$ is as follows\,: for two elements $g, h \in G$ and
	$\alpha\in \CH^{i-\age(g)}(M^{g})$, $\beta\in \CH^{j-\age(h)}(M^{h})$, their
	orbifold product is the following element in $\CH^{i+j-\age(gh)}(M^{gh})$\,:
	\begin{equation}\label{eqn:OrbChowProdIntro}
	\alpha\star_{orb}\beta\deff
	\iota_{*}\left(\left.\alpha\right\vert_{M^{<g,h>}}\cdot \beta|_{M^{<g,h>}}\cdot
	c_{top}(F_{g,h})\right),
	\end{equation}
	where $M^{<g,h>} := M^g \cap M^h$, $\iota: M^{<g,h>}\inj M^{gh}$ is the natural
	inclusion and $F_{g,h}$ is the obstruction bundle.
	This construction is completely parallel to the construction of orbifold
	cohomology due to Fantechi--G\"ottsche \cite{MR1971293} which is further
	simplified in Jarvis--Kaufmann--Kimura \cite{MR2285746}. 
	
	With the orbifold Chow theory briefly reviewed above, we see that in Conjecture
	\ref{conj:MHRCbis}, the fancy side of $[M/G]$ is actually the easier side which
	can be used to study the motive and cycles of the hyperK\"ahler variety $Y$.
	Let us turn this idea into the following working principle, which will be
	illustrated repeatedly in examples in the rest of the introduction.
	
	\noindent\textbf{Slogan :} \emph{The cohomology theories\footnote{In the large
			sense\,: Weil cohomology, Chow rings, K-theory, motivic cohomology, \etc {} and
			finally, motives.} of a holomorphic symplectic variety can be understood via the
		hidden stack structure of its singular symplectic models.}

	Interesting examples of symplectic resolutions appear when considering the
	Hilbert--Chow morphism of a smooth projective surface. More precisely, in his
	fundamental paper \cite{MR730926}, Beauville provides such examples\,: \\
	
	\noindent\emph{Example 1.}\\ 
	Let $S$ be a complex projective K3 surface or an abelian surface. Its Hilbert
	scheme of length-$n$ subschemes, denoted by $S^{[n]}$, is a symplectic crepant
	resolution of the symmetric product $S^{(n)}$ \emph{via} the Hilbert--Chow morphism.
	The corresponding Cohomological HyperK\"ahler Resolution Conjecture was proved
	independently by Fantechi--G\"ottsche in \cite{MR1971293} and Uribe in
	\cite{MR2154668} making use of Lehn--Sorger's work \cite{MR1974889} computing
	the ring structure of $H^{*}(S^{[n]})$. The Motivic HyperK\"ahler Resolution
	Conjecture \ref{conj:MHRCbis} in the case of K3 surfaces will be proved in
	\cite{MHRCK3} and the case of abelian surfaces is our first main result\,:
	
	\begin{thm}[MHRC for $A^{[n]}$]\label{thm:mainAb}
		Let $A$ be an abelian surface and $A^{[n]}$ be its Hilbert scheme as before.
		Then we have an isomorphism of commutative algebra objects in the category
		$\CHM$ of Chow motives with rational coefficients\,:
		$$\h\left(A^{[n]}\right)\isom \h_{orb,dt}\left(\left[A^{n}/\gS_{n}\right]\right),$$  where on the left
		hand side, the product structure is given by the small diagonal of
		$A^{[n]}\times A^{[n]}\times A^{[n]}$ while on the right hand side, the product
		structure is given by the orbifold product $\star_{orb}$ with a suitable sign
		change, called \emph{discrete torsion} in Definition\ref{def:dt}. In particular, we have
		an isomorphism of commutative graded $\Q$-algebras\,:
		\begin{equation}\label{eqn:mainAb}
		\CH^{*}\left(A^{[n]}\right)_{\Q}\isom \CH^{*}_{orb,dt}([A^{n}/\gS_{n}]).
		\end{equation}
	\end{thm}
	
	\vspace{1cm}
	\noindent\emph{Example 2.}\\
	Let $A$ be a complex abelian surface. The composition of the Hilbert--Chow
	morphism followed by the summation map $A^{[n+1]}\to A^{(n+1)}\to A$ is an
	isotrivial fibration. The \emph{generalized Kummer variety} $K_{n}(A)$ is by
	definition the fiber of this morphism over the origin of $A$. It is a
	hyperK\"ahler resolution of the quotient $A^{n+1}_{0}/\gS_{n+1}$, where
	$A_{0}^{n+1}$ is the kernel abelian variety of the summation map $A^{n+1}\to A$.
	The second main result of the paper is the following theorem confirming the Motivic
	HyperK\"ahler Resolution Conjecture \ref{conj:MHRCbis} in this situation.
	
	\begin{thm}[MHRC for $K_{n}(A)$]\label{thm:mainKummer}
		Let $K_{n}(A)$ be the $2n$-dimensional generalized Kummer variety associated to
		an abelian surface $A$. Let $A^{n+1}_{0}\deff \Ker\left(+: A^{n+1}\to A\right)$
		endowed with the natural $\gS_{n+1}$-action. Then we have an isomorphism of
		commutative algebra objects in the category $\CHM$ of Chow motives with rational
		coefficients\,:
		$$\h\left(K_{n}(A)\right)\isom \h_{orb,dt}\left(\left[A^{n+1}_{0}/\gS_{n+1}\right]\right),$$  where on
		the left hand side, the product structure is given by the small diagonal while
		on the right hand side, the product structure is given by the orbifold product
		$\star_{orb}$ with the sign change given by discrete torsion in \ref{def:dt}. In
		particular, we have an isomorphism of commutative graded $\Q$-algebras\,:
		\begin{equation}\label{eqn:mainKummer}
		\CH^{*}\left(K_{n}(A)\right)_{\Q}\isom
		\CH^{*}_{orb,dt}\left(\left[A^{n+1}_{0}/\gS_{n+1}\right]\right).
		\end{equation}
	\end{thm}
	
	\subsection{Consequences}
	We get some by-products of our main results.
	
	Taking the Betti cohomological realization, we confirm Ruan's original
	Cohomological HyperK\"ahler Resolution Conjecture \ref{conj:CHRC}  in the 
	case of generalized Kummer varieties\,:

	\begin{thm}[CHRC for $K_{n}(A)$]\label{cor:CHRCKummer}
		Let the notation be as in Theorem \ref{thm:mainKummer}. We have an isomorphism of
		graded commutative $\Q$-algebras\,:
		$$H^{*}\left(K_{n}(A)\right)_{\Q}\isom
		H^{*}_{orb,dt}\left(\left[A^{n+1}_{0}/\gS_{n+1}\right]\right).$$
	\end{thm}
	The CHRC has never been proved in the case of generalized Kummer varieties in
	the literature. Related work on the CHRC in this case are Nieper--Wi\ss
	kirchen's description of the cohomology ring $H^{*}(K_{n}(A),\C)$ in
	\cite{MR2578804}, which plays an important r\^ole in our proof\,; and Britze's
	thesis \cite{Britze} comparing $H^{*}(A\times K_{n}(A),\C)$ and the computation
	of the orbifold cohomology ring of $[A\times A^{n+1}_{0}/\gS_{n+1}]$ in
	Fantechi--G\"ottsche \cite{MR1971293}. See however Remark \ref{rmk:CHRCKummer}.
	
	From the K-theoretic point of view, we also have the following closely related
	conjecture (KHRC) in \cite[Conjecture 1.2]{MR2285746}, where the \emph{orbifold
		K-theory} is defined in a similar way with top Chern class in
	(\ref{eqn:OrbChowProdIntro}) replaced by the K-theoretic Euler class\,; see
	Definition \ref{def:OrbK} for details.
	\begin{conj}[K-theoretic HyperK\"ahler Resolution Conjecture
		\cite{MR2285746}]\label{conj:KHRC}
		In the same situation as in Conjecture \ref{conj:MHRC}, we have isomorphisms of
		$\C$-algebras\,:
		\begin{eqnarray*}
			K_{0}(Y)_{\C} \isom K_{orb}(\rX)_{\C}\,;\\
			K^{top}(Y)_{\C} \isom K^{top}_{orb}(\rX)_{\C}.\\
		\end{eqnarray*}
	\end{conj}
	
	Using Theorems \ref{thm:mainAb} and \ref{thm:mainKummer}, we can confirm
	Conjecture \ref{conj:KHRC} in the two cases considered here\,:
	\begin{thm}[KHRC for $A^{[n]}$ and $K_{n}(A)$]\label{thm:KHRC}
		Let $A$ be an abelian surface and $n$ be a natural number. There are isomorphisms
		of commutative $\C$-algebras\,:
		\begin{eqnarray*}
			K_{0}\left(A^{[n]}\right)_{\C}&\isom& K_{orb}\left(\left[A^{n}/\gS_{n}\right]\right)_{\C}\,;\\
			K^{top}\left(A^{[n]}\right)_{\C}&\isom& K^{top}_{orb}\left(\left[A^{n}/\gS_{n}\right]\right)_{\C}\,;\\
			K_{0}\left(K_{n}(A)\right)_{\C}&\isom& K_{orb}\left(\left[A_{0}^{n+1}/\gS_{n+1}\right]\right)_{\C}\,;\\
			K^{top}\left(K_{n}(A)\right)_{\C}&\isom&
			K^{top}_{orb}\left(\left[A_{0}^{n+1}/\gS_{n+1}\right]\right)_{\C}.
		\end{eqnarray*}
	\end{thm}

	\subsection{On explicit descriptions of the Chow rings}
	Let us make some remarks on the way we understand Theorem \ref{thm:mainAb} and
	Theorem \ref{thm:mainKummer}. For each of them, the seemingly fancy right hand
	side of (\ref{eqn:mainAb}) and (\ref{eqn:mainKummer}) given by orbifold Chow
	ring is actually very concrete (see (\ref{eqn:OrbChowDefi}))\,: as groups, since
	all fixed loci are just various diagonals, they are direct sums of Chow groups
	of products of the abelian surface $A$, which can be handled by Beauville's
	decomposition of Chow rings of abelian varieties \cite{MR826463}\,; while the
	ring structures are given by the orbifold product which is extremely simplified
	in our cases (see (\ref{eqn:OrbChowProdIntro}))\,: all obstruction bundles
	$F_{g,h}$ are trivial and hence the orbifold products are either the
	intersection product pushed forward by inclusions or simply zero. 
	
	In short, given an abelian surface $A$, Theorem \ref{thm:mainAb} and Theorem
	\ref{thm:mainKummer} provide an explicit description of the Chow rings of
	$A^{[n]}$ and of $K_{n}(A)$ in terms of Chow rings of products of $A$ (together
	with some combinatoric rules specified by the orbifold product). To illustrate
	how explicit it is, we work out two simple examples in \S\ref{subsec:ToyEx}\,:
	the Chow ring of the Hilbert square of a K3 surface or an abelian surface and
	the Chow ring of the Kummer K3 surface associated to an abelian surface.

	\subsection{Motivation 2\,: Beauville's splitting property}
	The original motivation for the authors to study the Motivic HyperK\"ahler
	Resolution Conjecture \ref{conj:MHRC} was to understand the (rational) Chow
	rings, or more generally the Chow motives, of smooth projective holomorphic
	symplectic varieties, that is, of even-dimensional projective manifolds carrying
	a holomorphic 2-form which is symplectic (\ie non-degenerate at each point). As
	an attempt to unify his work on algebraic cycles on abelian varieties
	\cite{MR826463} and his result with Voisin on Chow rings of K3 surfaces
	\cite{MR2047674}, Beauville conjectured in \cite{MR2187148}, under the name of
	\emph{the splitting property}, that for a smooth projective holomorphic
	symplectic variety $X$, there exists a canonical multiplicative splitting of the
	conjectural Bloch--Beilinson--Murre filtration of the rational Chow ring (see
	Conjecture \ref{conj:SplittingPrincipleCh} for the precise statement). In this
	paper, we will understand the splitting property as in the following motivic
	version (see Definition \ref{def:MCK} and Conjecture
	\ref{conj:SplittingPrincipleMot})\,:
	
	\begin{conj}[Beauville's Splitting Property\,: motives]\label{conj:MSPIntro}
		Let $X$ be a smooth projective holomorphic symplectic variety of dimension $2n$.
		Then we have a canonical multiplicative Chow--K\"unneth decomposition of $\h(X)$
		of Bloch--Beilinson type, that is, a direct sum decomposition in the category of
		rational Chow motives\,:
		\begin{equation}\label{eqn:IntroDecomp}
		\h(X)=\bigoplus_{i=0}^{4n}\h^{i}(X)
		\end{equation}
		satisfying the following properties\,:
		\begin{enumerate}
			\item[(i)] (Chow--K\"unneth) The cohomology realization of the decomposition
			gives the K\"unneth decomposition\,: for each $0\leq i\leq 4n$,
			$H^{*}(\h^{i}(X))=H^{i}(X)$.
			\item[(ii)] (Multiplicativity) The product $\mu: \h(X)\otimes \h(X)\to \h(X)$
			given by the small diagonal $\delta_{X}\subset X\times X\times X$ respects the
			decomposition\,: the restriction of $\mu$ on the summand
			$\h^{i}(X)\otimes\h^{j}(X)$ factorizes through $\h^{i+j}(X)$.
			\item[(iii)] (Bloch--Beilinson--Murre) For any $i, j\in \N$,
			\begin{enumerate}
				\item[-] $\CH^{i}(\h^{j}(X))=0$ if $j<i$\,;
				\item[-] $\CH^{i}(\h^{j}(X))=0$ if $j>2i$\,;
				\item[-] the realization induces an injective map $\Hom_{\CHM}\left(\1(-i),
				\h^{2i}(X)\right)\to \Hom_{\Q-HS}\left(\Q(-i), H^{2i}(X)\right)$. 
			\end{enumerate}
		\end{enumerate}
	\end{conj}
	Such a decomposition naturally induces a (multiplicative) bigrading on the Chow
	ring $\CH^{*}(X)=\oplus_{i,s}\CH^{i}(X)_{s}$ by setting\,:
	\begin{equation}\label{eq:bigradingCH}
	\CH^{i}(X)_{s}:= \Hom_{\CHM}\left(\1(-i), \h^{2i-s}(X)\right),
	\end{equation}
	which is the original splitting that Beauville envisaged.
	
	Our main results Theorem \ref{thm:mainAb} and Theorem \ref{thm:mainKummer} allow
	us, for $X$ being a Hilbert scheme of an abelian surface or a generalized Kummer
	variety, to achieve in Theorem \ref{thm:MCKIntro} below partially the goal
	Conjecture \ref{conj:MSPIntro}\,: we construct the candidate direct sum
	decomposition (\ref{eqn:IntroDecomp}) satisfying the first two conditions (i)
	and (ii) in Conjecture \ref{conj:MSPIntro}, namely a \emph{self-dual
		multiplicative Chow--K\"unneth decomposition} (see Definition \ref{def:MCK}, \cf
	\cite{MR3460114}). The remaining Condition (iii) on Bloch--Beilinson--Murre
	properties is very much related to Beauville's \emph{Weak Splitting Property},
	which has already been proved in \cite{MR3356741} for the case of generalized
	Kummer varieties\,; see \cite{MR2187148}, \cite{MR2435839}, \cite{MR3351754},
	\cite{MR3579961} for the complete story and more details. 
	
	\begin{thm}[=Theorem \ref{thm:MCK} + Proposition
		\ref{prop:ChernClass}]\label{thm:MCKIntro}
		Let $A$ be an abelian surface and $n$ be a positive integer. Let $X$ be the
		corresponding $2n$-dimensional Hilbert scheme $A^{[n]}$ or generalized Kummer
		variety $K_{n}(A)$. Then $X$ has a canonical self-dual multiplicative
		Chow--K\"unneth decomposition induced by the isomorphisms of
		Theorems~\ref{thm:mainAb} and \ref{thm:mainKummer}, respectively.
		Moreover, via the induced canonical multiplicative bigrading on the (rational)
		Chow ring given in \eqref{eq:bigradingCH}, the $i$-th Chern class of $X$ lies in
		$\CH^{i}(X)_{0}$ for any $i$.
	\end{thm}
	The associated filtration $F^{j}\CH^{i}(X):=\oplus_{s\geq j}\CH^{i}(X)_{s}$ is
	supposed to satisfy the Bloch--Beilinson--Murre conjecture (see Conjecture
	\ref{conj:BB}). We point out in Remark \ref{rmk:MurreKummer} that Beauville's
	Conjecture \ref{conj:BeauvilleAV} on abelian varieties implies for $X$ in our
	two cases some Bloch--Beilinson--Murre properties\,: $\CH^{*}(X)_{s}=0$ for
	$s<0$ and the cycle class map restricted to $\CH^{*}(X)_{0}$ is injective.
	
	See Remark \ref{rmk:related} for previous related results.
	
	\subsection{Cup products \vs decomposition theorem}
	For a smooth projective morphism $\pi: \rX\to B$ Deligne shows in
	\cite{MR0244265} that one has an isomorphism 
	$$R\pi_*\Q \cong \bigoplus_i R^i\pi_*\Q[-i],$$ in the derived category of
	sheaves of $\Q-$vector spaces on $B$.
	Voisin \cite{MR2916291} shows that, although this isomorphism cannot in general
	be made compatible with the product structures on both sides, not even after
	shrinking $B$ to a Zariski open subset, it can be made so if $\pi$ is a smooth
	family of projective K3 surfaces. Her result is extended in \cite{VialHK} to
	relative Hilbert schemes of finite lengths of a smooth family of projective K3
	surfaces or abelian surfaces. As a by-product of our main result in this paper,
	we can similarly prove the case of generalized Kummer varieties.
	\begin{thm}[=Corollary \ref{cor:decompKummer}]  Let $\mathcal{A} \rightarrow B$
		be
		an abelian surface over $B$. Consider $\pi: K_n(\mathcal{A}) \rightarrow B$ the
		relative generalized Kummer variety. Then there exist a decomposition
		isomorphism
		\begin{equation}
		R\pi_*\Q \cong \bigoplus_i R^i\pi_*\Q[-i],
		\end{equation} 
		and a nonempty Zariski open subset $U$ of $B$, such that
		this decomposition becomes multiplicative for the restricted family
		over $U$.
	\end{thm}

	\noindent\textbf{Convention and notation.} Throughout the paper, all varieties
	are defined over the field of complex numbers. 
	\begin{itemize}
		\item The notation $\CH$ (\resp $\CH_{\C}$) means Chow groups with
		\emph{rational} (\resp \emph{complex}) coefficients. $\CHM$ is the category of
		Chow motives over the complex numbers with \emph{rational} coefficients.
		\item For a variety $X$, its small diagonal, always denoted by $\delta_{X}$, is
		$\left\{(x,x,x)~\middle\vert~ x\in X\right\}\subset X\times X\times X$.
		\item For a smooth surface $X$, its Hilbert scheme of length-$n$ subschemes is
		always denoted by $X^{[n]}$. It is smooth of dimension $2n$ by \cite{MR0237496}.
		\item An (even) dimensional smooth projective variety is \emph{holomorphic
			symplectic} if it has a holomorphic symplectic (\ie non-degenerate at each
		point) 2-form. When talking about resolutions, we tend to use the word
		\emph{hyperK\"ahler} as its synonym, which usually (but not in this paper)
		requires also the `irreducibility', that is, the simple connectedness of the
		variety and the uniqueness up to scalars of the holomorphic symplectic 2-form.
		In particular, punctual Hilbert schemes of abelian surfaces are examples of
		holomorphic symplectic varieties. 
		\item An abelian variety is always supposed to be connected. Its non-connected
		generalization causes extra difficulty and is dealt with in \S\ref{subsec:SD-B}.
		\item When working with 0-cycles on an abelian variety $A$, to avoid confusion, 
		for a collection of points $x_{1}, \ldots, x_{m}\in A$, we will write
		$[x_{1}]+\cdots+[x_{m}]$ for the 0-cycle of degree $m$ (or equivalently, a point
		in $A^{(m)}$, the $m$-th symmetric product of $A$) and $x_{1}+\cdots +x_{m}$
		will stand for the usual sum using the group law of $A$, which is therefore a
		point in $A$.
	\end{itemize}
	
	\noindent\textbf{Acknowledgements.} We would like to thank Samuel Boissi\`ere,
	Alessandro Chiodo, Julien Grivaux, Bruno Kahn, Manfred Lehn,  Marc Nieper-Wi\ss
	kirchen, Yongbin Ruan,  Claire Voisin and Qizheng Yin for helpful discussions
	and email correspondences. We are especially thankful to the referee for a
	thorough reading of our manuscript, and for many useful remarks. The project was
	initiated when we were members of the I.A.S. for the special year \emph{Topology
		and Algebraic Geometry}  in 2014--15 (L.F. and Z.T. were funded by the NSF and
	C.V. by the Fund for Mathematics). We thank the Institute for the exceptional
	working conditions.

	\section[Orbifold Chow rings]{Orbifold motives and orbifold Chow
		rings}\label{sec:OrbDef}
	
	To fix the notation, we start by a brief reminder of the construction of pure
	motives (\cf \cite{MR2115000}). In order to work with Tate twists by age
	functions (Definition \ref{def:age}), we have to extend slightly the usual notion of pure
	motives by allowing twists by a rational number. 
	\begin{defi}[Chow motives with fractional Tate twists]\label{def:Mot}
		The category of \emph{Chow motives with fractional Tate twists} with rational
		coefficients, denoted by $\tilde\CHM$, has as objects finite direct sums of
		triples of the form $(X, p, n)$ with $X$ a connected smooth projective variety,
		$p\in \CH^{\dim X}(X\times X)$ a projector and $n\in \Q$ a \emph{rational}
		number. Given two objects $(X,p,n)$ and $(Y,q,m)$, the morphism space between
		them consists of correspondences\,:
		$$\Hom_{\tilde\CHM}\left((X,p,n), (Y,q,m)\right)\deff q\circ \CH^{\dim
			X+m-n}(X\times Y)\circ p,$$
		where we simply impose that all Chow groups of a variety with non-integer
		codimension are zero. The composition law of correspondences is the usual one.
		Identifying $(X,p,n)\oplus (Y,q,n)$ with $(X\coprod Y, p\coprod q, n)$ makes
		$\CHM$ a $\Q$-linear category. Moreover, $\CHM$ is a rigid symmetric mono\"idal
		pseudo-abelian category with unit $\1:=(\Spec\C, \Spec\C, 0)$, tensor product
		defined by $(X,p,n)\otimes (Y,q,m):=(X\times Y,p\times q, n+m)$ and duality
		given by $(X,p,n)^{\dual}\deff \left(X,{}^{t}p, \dim X-n\right)$. There is a
		natural contravariant functor $\h:\Smproj^{op}\to \CHM$ sending a smooth
		projective variety $X$ to its Chow motive $\h(X)=(X,\Delta_{X}, 0)$ and a
		morphism $f:X\to Y$ to its  transposed graph ${}^t\Gamma_{f}\in \CH^{\dim
			Y}(Y\times X)=\Hom_{\CHM}(\h(Y),\h(X))$.
	\end{defi}
	
	\begin{rmks}\label{rmks:MotiveGeneral}
		Some general remarks are in order.
		\begin{enumerate}
			\item[(i)] The category $\tilde\CHM_{\C}$ of \emph{Chow motives with fractional
				Tate twists} with complex coefficients is defined similarly by replacing all
			Chow groups with rational coefficients $\CH$ by Chow groups with complex
			coefficients $\CH_{\C}$ in the above definition.
			\item[(ii)] The usual category of Chow motives with rational (\resp complex)
			coefficients $\CHM$ (\resp $\CHM_{\C}$, \cf \cite{MR2115000}) is identified with
			the full subcategory of $\tilde\CHM$ (\resp $\tilde\CHM_{\C}$) consisting of
			objects $(X,p,n)$ with $n\in\Z$. 
			\item[(iii)] Thanks to the extension of the intersection theory (with rational
			coefficients) of Fulton \cite{MR1644323} to the so-called \emph{$\Q$-varieties}
			by Mumford \cite{MR717614}, the motive functor $\h$ defined above can actually
			be extended to the larger category of finite group quotients of smooth
			projective varieties, or more generally to $\Q$-varieties with global
			Cohen-Macaulay cover, see for example \cite[\S\S 2.2-2.3]{MR2067464}. Indeed,
			for global quotients one defines $\h(M/G):=(M, \frac{1}{|G|}\sum_{g\in
				G}{}^{t}\Gamma_{g}, 0) := \h(M)^G$. (Note that it is essential to work with
			rational coefficients.) Denoting $\pi : M \to M/G$ the quotient morphism and
			letting $X$ be an auxiliary variety, a morphism from $\h(X)$ to $\h(M/G)$ is a
			correspondence in $\CH^{\dim X}(X\times M/G)$, which under the above
			identification $\h(M/G)=\h(M)^G$, is regarded as a $G$-invariant element of $\CH^{\dim
				X}(X\times M)$ \emph{via} the pull-back $\id_X \times \pi^*$, where $\pi^*$ is
			defined in \cite[Example~1.7.6]{MR1644323}. The latter has the property that
			$\pi_*\pi^* = |G|\cdot\id$ while $\pi^*\pi_* = \sum_{g\in G}{}^{t}\Gamma_{g}$.
			It is useful to observe that if we replace $G$ by $G\times H$, where $H$ acts
			trivially on $M$, the pull-back $\pi^*$ changes by the factor $|H|$. We will
			avoid this kind of confusion by only considering \emph{faithful} quotients when
			dealing with Chow groups of quotient varieties. 
		\end{enumerate}
	\end{rmks}

	Let $M$ be an $m$-dimensional smooth projective complex variety equipped with a faithful
	action of a finite group $G$. We adapt the constructions in \cite{MR1971293} and
	\cite{MR2285746} to define the orbifold motive of the smooth proper
	Deligne--Mumford stack $[M/G]$. For any $g\in G$, $M^{g}:=\left\{x\in M ~|~
	gx=x\right\}$ is the fixed locus of the automorphism $g$, which is a smooth
	subvariety of $M$. The following notion is due to Reid (see \cite{MR1886756}).
	\begin{defi}[Age]\label{def:age}
		Given an element $g\in G$, let $r\in \N$ be its order. The \emph{age} of $g$,
		denoted by $\age(g)$, is the locally constant $\Q_{\geq 0}$-valued function on
		$M^{g}$ defined as follows. Let $Z$ be a connected component of $M^{g}$.
		Choosing any point $x\in Z$, we have the induced automorphism $g_{*}\in
		\GL(T_{x}M)$, whose eigenvalues, repeated according to multiplicities, are
		$$\left\{e^{2\pi\sqrt{-1}\frac{\alpha_{1}}{r}}, \cdots,
		e^{2\pi\sqrt{-1}\frac{\alpha_{m}}{r}}\right\},$$ with $0\leq \alpha_{i}\leq
		r-1$. One defines
		$$\age(g)|_{Z}\deff\frac{1}{r}\sum_{i=1}^{m}\alpha_{i}.$$
		It is obvious that the value of $\age(g)$ on $Z$ is independent of the choice of
		$x\in Z$ and it takes values in~$\N$ if $g_{*}\in \SL(T_{x}M)$. Also immediate
		from the definition, we have $\age(g)+\age(g^{-1})=\codim(M^{g}\subset M)$ as
		locally constant functions.
		Thanks to the natural isomorphism $h: M^{g}\to M^{hgh^{-1}}$ sending $x$ to
		$h.x$, for any $g, h\in G$, the age function is invariant under conjugation.
	\end{defi}
	
	\begin{expl}\label{expl:age}
		Let $S$ be a smooth projective variety of dimension $d$ and $n$ a positive
		integer. The symmetric group $\gS_{n}$ acts by permutation on $M=S^{n}$. For
		each $g\in \gS_{n}$, a straightforward computation (see \eqref{sec:agesym})
		shows that $\age(g)$ is the constant function $\frac{d}{2}(n-|O(g)|)$, where
		$O(g)$ is the set of orbits of $g$ as a permutation of $\{1,\ldots, n\}$. For
		example, when $S$ is a surface (\ie, $d=2$), the age in this case is always a
		non-negative integer and we have $\age(\id)=0$, $\age(12\cdots r)=r-1$,
		$\age(12)(345)=3$ \etc.
	\end{expl}

	Recall that an \emph{algebra object} in a symmetric mono\"idal category $(\rM,
	\otimes, \1)$ (for example, $\CHM$, $\tilde\CHM$ \etc ) is an object $A\in
	\Obj\rM$ together with a morphism $\mu: A\otimes A\to A$ in $\rM$, called the
	\emph{multiplication} or \emph{product structure}, satisfying the associativity
	axiom $\mu\circ (\mu\otimes \id)=\mu\circ (\id\otimes \mu)$.  An algebra object
	$A$ in $\rM$ is called \emph{commutative} if $\mu\circ \iota=\mu$, where $\iota:
	A\otimes A\to A\otimes A$ is the structural symmetry isomorphism of $\rM$. For
	each smooth projective variety $X$, its Chow motive $\h(X)$ is naturally a
	commutative algebra object in $\CHM$ (hence in $\tilde\CHM$, $\tilde\CHM_{\C}$,
	\etc) whose multiplication is given by the small diagonal $\delta_{X}\in
	\CH^{2\dim X}(X\times X\times X)=\Hom_{\CHM}\left(\h(X)\otimes \h(X),
	\h(X)\right)$.
	
	\begin{defi}[Orbifold Chow motive]\label{def:OrbMot}
		We define first of all an auxiliary (in general non-commutative) algebra object
		$\h(M,G)$ of $\tilde\CHM$ in several steps\,:
		\begin{enumerate}
			\item[(i)] As a Chow motive with fractional twists, $\h(M,G)$ is defined to be
			the direct sum over $G$, of the motives of fixed loci twisted \`a la Tate by
			$-\age$\,:
			$$\h(M,G)\deff \bigoplus_{g\in G} \h(M^{g})\left(-\age(g)\right).$$
			\item[(ii)] $\h(M, G)$ is equipped with a natural $G$-action\,: each element
			$h\in G$ induces for each $g\in G$ an isomorphism $h: M^{g}\to M^{hgh^{-1}}$ by
			sending $x$ to $h.x$, hence an isomorphism between the direct summands
			$\h(M^{g})(-\age(g))$ and $\h(M^{hgh^{-1}})(-\age(hgh^{-1}))$ by the conjugation
			invariance of the age function. 
			\item[(iii)] For any $g\in G$, let $r$ be its order. We have a natural
			automorphism $g_{*}$ of the vector bundle $TM|_{M^{g}}$. Consider its
			eigen-subbundle decomposition\,: $$TM|_{M^{g}}=\bigoplus_{j=0}^{r-1}W_{g,j},$$
			where $W_{g,j}$ is the subbundle associated to the eigenvalue
			$e^{2\pi\sqrt{-1}\frac{j}{r}}$. Define $$S_{g}\deff
			\sum_{j=0}^{r-1}\frac{j}{r}[W_{g,j}]\in K_{0}(M^{g})_{\Q}.$$ Note that the
			virtual rank of $S_{g}$ is nothing but $\age(g)$ by Definition \ref{def:age}.
			\item[(iv)] For any $g_{1}, g_{2}\in G$, let $M^{<g_{1},g_{2}>}= M^{g_{1}} \cap
			M^{g_{2}}$ and $g_{3}=g_{2}^{-1}g_{1}^{-1}$. Define the following element in
			$K_{0}(M^{<g_{1},g_{2}>})_{\Q}$\,: $$F_{g_{1},g_{2}}\deff
			\left.S_{g_{1}}\right\vert_{M^{<g_{1},g_{2}>}}+\left.S_{g_{2}}\right\vert_{M^{<g_{1},g_{2}>}}+\left.S_{g_{3}}\right\vert_{M^{<g_{1},g_{2}>}}+TM^{<g_{1},g_{2}>}-\left.TM\right\vert_{M^{<g_{1},g_{2}>}}.$$
			
			Note that its virtual rank is 
			\begin{equation}\label{eqn:VirRkF}
			\rk F_{g_{1}, g_{2}}=\age(g_{1})+\age(g_{2})-\age(g_{1}g_{2})-\codim(M^{<g_{1},
				g_{2}>}\subset M^{g_{1}g_{2}}).
			\end{equation}
			In fact, this class in the Grothendieck group is represented by a genuine
			obstruction vector bundle constructed in \cite{MR1971293} (\cf
			\cite{MR2285746}). In particular, $\age(g_{1})+\age(g_{2})-\age(g_{1}g_{2})$ is
			always an integer.
			\item[(v)] The product structure $\star_{orb}$ on $\h(M,G)$ is defined to be
			multiplicative with respect to the $G$-grading and for each $g_{1}, g_{2}\in G$,
			the \emph{orbifold product} $$\star_{orb}: \h(M^{g_{1}})(-\age(g_{1}))\otimes
			\h(M^{g_{2}})(-\age(g_{2}))\to \h(M^{g_{1}g_{2}})(-\age(g_{1}g_{2}))$$ is the
			correspondence determined by the algebraic cycle 
			$$\delta_{*}(c_{top}(F_{g_{1},g_{2}}))\in \CH^{\dim M^{g_{1}}+\dim
				M^{g_{2}}+\age(g_{1})+\age(g_{2})-\age(g_{1}g_{2})}(M^{g_{1}}\times
			M^{g_{2}}\times M^{g_{1}g_{2}}),$$
			where $\delta: M^{<g_{1}, g_{2}>}\to M^{g_{1}}\times M^{g_{2}}\times
			M^{g_{1}g_{2}}$ is the natural morphism sending $x$ to $(x,x,x)$ and $c_{top}$
			means the top Chern class of $F_{g_{1},g_{2}}$. One can check easily that the
			product structure $\star_{orb}$ is invariant under the action of $G$.
			\item[(vi)] The associativity of $\star_{orb}$ is non-trivial. The proof in
			\cite[Lemma 5.4]{MR2285746} is completely algebraic hence also works in our
			motivic case. 
			\item[(vii)] Finally,  the \emph{orbifold Chow motive} of $[M/G]$, denoted by
			$\h_{orb}([M/G])$, is the $G$-invariant subalgebra object\footnote{Here we use
				the fact that the category $\tilde\CHM$ is  $\Q-$linear and pseudo-abelian to
				define the $G$-invariant part $A^{G}$ of a $G$-object $A$ as the image of the
				projector $\frac{1}{|G|}\sum_{g\in G}g \in \End(A)$.} of $\h(M,G)$, which turns
			out to be a commutative algebra object in $\tilde\CHM$\,:
			\begin{equation}\label{eqn:OrbMot}
			\h_{orb}\left([M/G]\right) \deff \h(M, G)^{G}= \left(\bigoplus_{g\in G}
			\h(M^{g})\left(-\age(g)\right), \star_{orb}\right)^{G}
			\end{equation}
			We still use $\star_{orb}$ to denote the orbifold product on this sub-algebra
			object $\h_{orb}([M/G])$.
		\end{enumerate}
	\end{defi}
	
	\begin{rmk}\label{rmk:additiveOrbMot}
		With Definition \ref{def:OrbMot}(ii) in mind, the correspondence $$p :=
		\frac{1}{|G|}\sum_{h\in G} \Gamma_h \quad \in \bigoplus_{h\in G}\bigoplus_{g\in G} \CH^{\dim
			M^g}(M^g\times M^{hgh^{-1}})$$ defines an idempotent endomorphism of the Chow
		motive $\h(M,G) = \bigoplus_{g\in G} \h(M^{g})\left(-\age(g)\right)$.  Under
		this identification, and ignoring the algebra structure, the Chow motive
		$\h_{orb}\left([M/G]\right)$ is defined explicitly as the image of $p$ (which
		exists, since the category of Chow motives is pseudo-abelian). Composing a
		correspondence in $\Hom_{\CHM}((Y,q,m),\h(M,G))$ with $p$ amounts to
		symmetrizing. The orbifold product on $\h_{orb}\left([M/G]\right)$  is then
		given by the symmetrization of the orbifold product of $\h(M,G)$ (Definition
		\ref{def:OrbMot}(v)), that is, by $p\circ \star_{orb} \circ (p\otimes p) :
		\h(M,G)\otimes \h(M,G)\to \h(M,G)$. We note that $p$ is self-dual so that, by
		\cite[Lemma~3.3]{Vialabelian}, $p\circ \star_{orb} \circ (p\otimes p) =
		(p\otimes p\otimes p)_* \star_{orb}$ if $\star_{orb}$ is viewed as a cycle on
		$\h(M,G)\otimes \h(M,G)\otimes \h(M,G)$.
	\end{rmk}

	By replacing the rational equivalence relation by another adequate equivalence
	relation (\cf \cite{MR2115000}), the same construction gives the \emph{orbifold
		homological motives}, \emph{orbifold numerical motives}, \etc~ associated to a
	global quotient smooth proper Deligne--Mumford stack as algebra objects in the
	corresponding categories of pure motives (with fractional Tate twists).
	
	The definition of the orbifold Chow ring then follows in the standard way and
	agrees with the one in \cite{MR1971293}, \cite{MR2285746} and \cite{MR2450211}.
	
	\begin{defi}[Orbifold Chow ring]\label{def:OrbChow}
		The \emph{orbifold Chow ring} of $[M/G]$ is the commutative $\Q_{\geq 0}$-graded
		$\Q$-algebra $\CH^{*}_{orb}([M/G]):=\bigoplus_{i\in \Q_{\geq 0}}
		\CH_{orb}^{i}([M/G])$ with
		\begin{equation}
		\CH^{i}_{orb}([M/G]):=\Hom_{\tilde\CHM}(\1(-i), \h_{orb}([M/G]))
		\end{equation}
		The ring structure on $\CH^{*}_{orb}([M/G])$, called \emph{orbifold product},
		denoted again by $\star_{orb}$, is determined by the product structure
		$\star_{orb}: \h_{orb}([M/G])\otimes \h_{orb}([M/G])\to \h_{orb}([M/G])$ in
		Definition \ref{def:OrbMot}. \\
		More concretely, $\CH^{*}_{orb}([M/G])$ is the $G$-invariant sub-$\Q$-algebra of
		an auxiliary (non-commutative) finitely $\Q_{\geq 0}$-graded $\Q$-algebra
		$\CH^{*}(M, G)$, which is defined by $$\CH^{*}(M, G)\deff\left(\bigoplus_{g\in
			G}\CH^{*-\age(g)}(M^{g}), \star_{orb}\right),$$ where for two elements $g, h \in
		G$ and $\alpha\in \CH^{i-\age(g)}(M^{g})$, $\beta\in \CH^{j-\age(h)}(M^{h})$,
		their orbifold product is the following element in
		$\CH^{i+j-\age(gh)}(M^{gh})$\,:
		\begin{equation}\label{eqn:OrbProdChow}
		\alpha\star_{orb}\beta\deff
		\iota_{*}\left(\left.\alpha\right\vert_{M^{<g,h>}}\cdot \beta|_{M^{<g,h>}}\cdot
		c_{top}(F_{g,h})\right),
		\end{equation}
		where $\iota: M^{<g,h>}\inj M^{gh}$ is the natural inclusion.
	\end{defi}
	
	Similarly, the orbifold K-theory is defined as follows. Recall that for a smooth
	variety $X$ and for $F\in K_{0}(X)$, we have the \emph{Lambda operation}
	$\lambda_{t}\colon K_{0}(X)\to K_{0}(X)\llbracket t\rrbracket$, where $\lambda_t(F)$ is
	a formal power series $\sum_{i=0}^\infty t^i\lambda^i(F)$  subject to the
	multiplicativity relation $\lambda_{t}(F\oplus F') = \lambda_{t}(F) \cdot
	\lambda_{t}(F')$ for all objects $F,F'\in K(X)$, and such that, for any rank-$r$
	vector bundle $E$ over $X$, we have $\lambda_{t}([E]) = \sum_{i=0}^r
	t^i[\bigwedge^i E]$. \cf \cite[Chapter II, \S 4]{MR3076731}. Finally
	$\lambda_{-1}(F)$ is defined by evaluating at $t=-1$ in $\lambda_{t}(F)$ and is
	called the \emph{K-theoretic Euler class} of $F^{\vee}$\,; see also
	\cite[p.~34]{MR2285746}.
	\begin{defi}[Orbifold K-theory]\label{def:OrbK}
		The \emph{orbifold K-theory} of $[M/G]$, denoted by $K_{orb}([M/G])$, is the
		 sub-algebra of $G$-invariant elements of the $\Q$-algebra $K(M, G)$, which is defined by
		$$K(M, G)\deff\left(\bigoplus_{g\in G}K_{0}(M^{g}), \star_{orb}\right),$$ where for
		two elements $g, h \in G$ and $\alpha\in K_{0}(M^{g})$, $\beta\in K_{0}(M^{h})$, their
		orbifold product is the following element in $K_{0}(M^{gh})$\,:
		\begin{equation*}
		\alpha\star_{orb}\beta\deff
		\iota_{*}\left(\left.\alpha\right\vert_{M^{<g,h>}}\cdot \beta|_{M^{<g,h>}}\cdot
		\lambda_{-1}(F_{g,h}^{\dual})\right),
		\end{equation*}
		where $\iota: M^{<g,h>}\inj M^{gh}$ is the natural inclusion and
		$\lambda_{-1}(F_{g,h}^{\dual})$ is the \emph{K-theoretic Euler class} of
		$F_{g,h}$ as defined above.
	\end{defi}
	
	\begin{rmk}\label{rmk:Gorenstein}
		The main interest of the paper lies in the situation when the underlying
		singular variety of the orbifold has at worst Gorenstein singularities. Recall
		that an algebraic variety $X$ is \emph{Gorenstein} if it is Cohen--Macaulay and
		the dualizing sheaf is a line bundle, denoted $\omega_{X}$.  In the case of a
		global quotient $M/G$, being Gorenstein is implied by the local $G$-triviality of the canonical bundle $\omega_{M}$, which means that the stabilizer of each point $x\in M$ is contained in $\SL(T_{x}M)$. In this case, it is straightforward to check that the
		age function actually takes values in the integers $\Z$ and therefore the
		orbifold motive lies in the usual category of pure motives (without fractional
		twists) $\CHM$. In particular, the orbifold Chow ring and orbifold cohomology
		ring are $\Z$-graded. Example \ref{expl:age} exhibits a typical situation that
		we would like to study\,; see also Remark \ref{rmk:AgeIsInteger}.
	\end{rmk}
	
	\begin{rmk}[Non-global quotients]\label{rmk:OrbMotDM}
		In the broader setting of smooth proper Deligne--Mumford stacks which are not
		necessarily finite group global quotients, the orbifold Chow ring is still
		well-defined in \cite{MR2450211} but the down-to-earth construction as above,
		which is essential for the applications (\cf our slogan in \S\ref{sect:intro}),
		is lost (see however the equivariant treatment \cite{MR2667422}). Another
		problem is that the definition of the orbifold Chow motive in this general
		setting is neither available in the literature nor covered in this paper. In the case where the coarse moduli space is projective with Gorenstein singularities, the orbifold Chow motive is constructed in \cite[\S 2.3]{McKayFuTian} in the spirit of \cite{MR2450211}.
	\end{rmk}

	\section{Motivic HyperK\"ahler Resolution Conjecture}\label{sec:MHRC}
	
	\subsection{A motivic version of  the Cohomological HyperK\"ahler Resolution
		Conjecture}
	In \cite{MR1941583}, as part of the broader picture of \emph{stringy geometry
		and topology of orbifolds}, Yongbin Ruan proposed the Cohomological
	HyperK\"ahler Resolution Conjecture (CHRC) which says that the orbifold
	cohomology ring of a compact Gorenstein orbifold is isomorphic to the  Betti
	cohomology ring of a hyperK\"ahler crepant resolution of the underlying singular
	variety if one takes $\C$ as coefficients\,; see Conjecture \ref{conj:CHRC} in
	the introduction for the statement. As explained in Ruan \cite{MR2234886}, the
	plausibility of CHRC is justified by some considerations from theoretical
	physics as follows. Topological string theory predicts that the quantum
	cohomology theory of an orbifold should be equivalent to the quantum cohomology
	theory of a/any crepant resolution of (possibly some deformation of) the
	underlying singular variety. On the one hand, the orbifold cohomology ring
	constructed by Chen--Ruan \cite{MR2104605} is the classical part (genus zero
	with three marked points) of the quantum cohomology ring of the orbifold (see
	\cite{MR1950941})\,; on the other hand, the classical limit of the quantum
	cohomology of the resolution is the so-called \emph{quantum corrected}
	cohomology ring (\cite{MR2234886}). However, if the crepant resolution has a
	hyperK\"ahler structure, then all its Gromov--Witten invariants as well as the
	quantum corrections vanish and one expects therefore an equivalence, \ie an
	isomorphism of $\C$-algebras, between the orbifold cohomology of the orbifold
	and the usual Betti cohomology of the hyperK\"ahler crepant resolution.  
	
	Before moving on to a more algebro-geometric study, we have to recall some
	standard definitions and facts on (possibly singular) symplectic varieties (\cf
	\cite{MR1738060}, \cite{MR1863856})\,:
	
	\begin{defi}\label{def:SympSing}
		\begin{itemize}
			\item A \emph{symplectic form} on a smooth complex algebraic variety is a closed
			holomorphic 2-form that is non-degenerate at each point. A smooth variety is
			called \emph{holomorphic symplectic} or just \emph{symplectic} if it admits a
			symplectic form. Projective examples include deformations of Hilbert schemes of
			K3 surfaces and abelian surfaces and generalized Kummer varieties \etc. A
			typical non-projective example is provided by the cotangent bundle of a smooth
			variety.
			\item A (possibly singular) \emph{symplectic variety} is a normal complex
			algebraic variety such that its smooth part admits a symplectic form whose
			pull-back to a/any resolution extends to a holomorphic 2-form. A germ of such a
			variety is called a \emph{symplectic singularity}. Such singularities are
			necessarily rational Gorenstein \cite{MR1738060} and conversely,  by a result of
			Namikawa \cite{MR1863856}, a normal variety  is symplectic if and only if it has
			rational Gorenstein singularities and its smooth part admits a symplectic form.
			The main examples that we are dealing with are of the form of a quotient by a
			finite group of symplectic automorphisms of a smooth symplectic variety,
			\textit{e.g.}, the symmetric products $S^{(n)}=S^{n}/\gS_{n}$ of smooth
			algebraic surfaces $S$ with trivial canonical bundle. 
			\item Given a singular symplectic variety $X$, a \emph{symplectic resolution} or
			\emph{hyperK\"ahler resolution} is a resolution $f: Y\to X$ such that the
			pull-back of a symplectic form on the smooth part $X_{reg}$ extends to a
			symplectic form on $Y$. Note that a resolution is symplectic if and only if it
			is \emph{crepant}\,: $f^{*}\omega_{X}=\omega_{Y}$. The definition is independent
			of the choice of a symplectic form on $X_{reg}$. A symplectic resolution is
			always semi-small. The existence of symplectic resolutions and the relations
			between them form a highly attractive topic in holomorphic symplectic geometry.
			An interesting situation, which will not be touched upon in this paper however,
			is the normalization of the closure of a nilpotent orbit in a complex
			semi-simple Lie algebra, whose symplectic resolutions are extensively studied in
			the literature (see \cite{MR1943745}, \cite{MR3356756}). For examples relevant
			to this paper, see Examples \ref{exp:SympRes}.
		\end{itemize}
	\end{defi}

	Returning to the story of the HyperK\"ahler Resolution Conjecture, in order to
	study algebraic cycles and motives of holomorphic symplectic varieties,
	especially with a view towards Beauville's splitting property conjecture
	\cite{MR2187148} (see \S\ref{sec:splitting}), we would like to propose the
	motivic version of the CHRC\,; see Meta-Conjecture \ref{conj:MHRC} in the
	introduction for the general statement. As we are dealing exclusively with the
	global quotient case in this paper and its sequel, we will concentrate on this
	more restricted case and on the more precise formulation Conjecture
	\ref{conj:MHRCbis} in the introduction. 
	
	\begin{rmk}[Integral grading]\label{rmk:AgeIsInteger}
		We use the same notation as in Conjecture \ref{conj:MHRCbis}. Then, since $G$
		preserves a symplectic form (hence a canonical form) of $M$, the quotient
		variety $M/G$ has at worst Gorenstein singularities. As is pointed out in Remark
		\ref{rmk:Gorenstein}, this implies that the age functions take values in $\Z$,
		the orbifold motive $\h_{orb}\left(\left[M/G\right]\right)$ is in $\CHM$, the usual category of
		(rational) Chow motives and the orbifold Chow ring $\CH_{orb}^{*}\left(\left[M/G\right]\right)$ is
		integrally graded.
	\end{rmk}

	\begin{rmk}[K-theoretic analogue]\label{rmk:Ktheory}
		As is mentioned in the Introduction (Conjecture \ref{conj:KHRC}), we are also
		interested in the K-theoretic version of the HyperK\"ahler Resolution
		Conjecture (KHRC) proposed in \cite[Conjecture 1.2]{MR2285746}. We want to point
		out that in Conjecture \ref{conj:MHRCbis} above, the statement for Chow rings is
		more or less equivalent to KHRC\,; however, the full formulation for Chow
		motivic algebras is, on the other hand, strictly richer. In fact, in all cases
		that we are able to prove KHRC, in this paper as well as in the upcoming one
		\cite{MHRCK3}, we have to first solve MHRC on the motive level and deduce KHRC
		as a consequence. See \S \ref{sec:skeleton} for the proof of Theorem
		\ref{thm:KHRC}.
	\end{rmk}
	
	\begin{expls}\label{exp:SympRes}
		All examples studied in this paper are in the following situation\,: let $M$ and
		$G$ be as in Conjecture \ref{conj:MHRCbis} and 
		$Y$ be (the principal component of) the $G$-Hilbert scheme $G\!-\!\Hilb(M)$ of
		$G$-\emph{clusters} of $M$, that is, a 0-dimensional $G$-invariant subscheme of
		$M$ whose global functions form the regular $G$-representation (\cf
		\cite{MR1420598}, \cite{MR1838978}). In some interesting cases, $Y$ gives a
		symplectic resolution of $M/G$\,:
		\begin{itemize}
			\item  Let $S$ be a smooth algebraic surface and $G=\gS_{n}$ act on $M=S^{n}$ by
			permutation. By the result of Haiman \cite{MR1839919},
			$Y=\gS_{n}\!-\!\Hilb(S^{n})$ is isomorphic to the $n$-th punctual Hilbert scheme
			$S^{[n]}$, which is a crepant resolution, hence symplectic resolution if $S$ has
			trivial canonical bundle, of $M/G=S^{(n)}$, the $n$-th symmetric product. 
			\item Let $A$ be an abelian surface, $M$ be the kernel of the summation map $s:
			A^{n+1}\to A$ and $G=\gS_{n+1}$ acts on $M$ by permutations, then
			$Y=G\!-\!\Hilb(M)$ is isomorphic to the generalized Kummer variety $K_{n}(A)$
			and is a symplectic resolution of $M/G$.
			
		\end{itemize}
	\end{expls}
	
	Although both sides of the isomorphism in Conjecture \ref{conj:MHRCbis} are in
	the category $\CHM$ of motives with \emph{rational} coefficients, it is in
	general necessary to make use of roots of unity to realize such an isomorphism
	of algebra objects. However, in some situation, it is possible to stay in $\CHM$
	by making some sign change, which is related to the notion of \emph{discrete
		torsion} in theoretical physics\,:
	\begin{defi}[Discrete torsion]\label{def:dt}
		For any $g, h\in G$, let
		\begin{equation}\label{eqn:epsilon}
		\epsilon(g,h)\deff \frac{1}{2}\left(\age(g)+\age(h)-\age(gh)\right).
		\end{equation}
		It is easy to check that
		\begin{equation}\label{eqn:AssDT}
		\epsilon(g_{1}, g_{2})+ \epsilon(g_{1}g_{2},g_{3})=\epsilon(g_{1},
		g_{2}g_{3})+\epsilon(g_{2},g_{3}).
		\end{equation}
		In the case when $\epsilon(g,h)$ is an integer for all $g,h\in G$, we can define
		the \emph{orbifold Chow motive with discrete torsion}  of a global quotient
		stack $[M/G]$, denoted by $\h_{orb,dt}([M/G])$, by the following simple change
		of sign in Step (v) of Definition \ref{def:OrbMot}\,: the orbifold product with
		discrete torsion $$\star_{orb,dt}: \h(M^{g_{1}})(-\age(g_{1}))\otimes
		\h(M^{g_{2}})(-\age(g_{2}))\to \h(M^{g_{1}g_{2}})(-\age(g_{1}g_{2}))$$ is the
		correspondence determined by the algebraic cycle 
		$$(-1)^{\epsilon(g_{1}, g_{2})}\cdot\delta_{*}(c_{top}(F_{g_{1},g_{2}}))\in
		\CH^{\dim M^{g_{1}}+\dim
			M^{g_{2}}+\age(g_{1})+\age(g_{2})-\age(g_{1}g_{2})}(M^{g_{1}}\times
		M^{g_{2}}\times M^{g_{1}g_{2}}).$$
		Thanks to (\ref{eqn:AssDT}), $\star_{orb,dt}$ is still associative.
		Similarly, the \emph{orbifold Chow ring with discrete torsion} of $[M/G]$ is
		obtained by replacing Equation (\ref{eqn:OrbProdChow}) in Definition
		\ref{def:OrbChow} by
		\begin{equation}\label{eqn:dtChow}
		\alpha\star_{orb,dt}\beta
		\deff (-1)^{\epsilon(g,h)}\cdot i_{*}\left(\alpha|_{M^{<g,h>}}\cdot
		\beta|_{M^{<g,h>}}\cdot c_{top}(F_{g,h})\right),
		\end{equation}
		which is again associative by (\ref{eqn:AssDT}).
	\end{defi}
	
	Thanks to the notion of discrete torsions, we can have the following version of
	Motivic HyperK\"ahler Resolution Conjecture, which takes place in the category
	of rational Chow motives and involves only rational Chow groups.
	
	\begin{conj}[MHRC\,: global quotient case with discrete
		torsion]\label{conj:MHRCdt}
		In the same situation as Conjecture \ref{conj:MHRCbis}, suppose that
		$\epsilon(g,h)$ of Definition \ref{def:dt} is an integer for all $g,h\in G$. 
		Then we have an isomorphism of (commutative) algebra objects in the category  of
		Chow motives with \emph{rational} coefficients\,: $$\h(Y)\isom
		\h_{orb,dt}\left([M/G]\right) \text{ in } \CHM.$$ In particular, we have an
		isomorphism of graded $\Q$-algebras $$\CH^{*}(Y)\isom \CH_{orb,
			dt}^{*}([M/G]).$$
	\end{conj}
	
	\begin{rmk}\label{rmk:roots}
		It is easy to see that Conjecture \ref{conj:MHRCdt} implies Conjecture
		\ref{conj:MHRCbis}\,: to get rid of the discrete torsion sign change
		$(-1)^{\epsilon(g,h)}$, it suffices to multiply the isomorphism on each summand
		$\h(M^{g})(-\age(g))$, or $\CH(M^{g})$, by $\sqrt{-1}^{\age(g)}$, which involves
		of course the complex numbers (roots of unity at least).
	\end{rmk}

	\subsection{Toy examples}\label{subsec:ToyEx}
	To better illustrate the conjecture as well as the proof in the next section, we
	present in this subsection some explicit computations for two of the simplest
	nontrivial cases of MHRC. 
	\subsubsection{Hilbert squares of K3 surfaces}
	Let $S$ be a K3 surface or an abelian surface. Consider the involution $f$ on
	$S\times S$ flipping the two factors. The relevant Deligne--Mumford stack is $[S^{2}/f]$\,;
	its underlying singular symplectic variety is the second symmetric product
	$S^{(2)}$, and $S^{[2]}$ is its symplectic resolution. Let $\tilde{S^{2}}$ be
	the blowup of $S^{2}$ along its diagonal $\Delta_{S}$\,:
	\begin{displaymath}
	\xymatrix{
		E\ar[r]^{j} \ar[d]_{\pi} \cart & \tilde{S^{2}} \ar[d]^{\epsilon} \\
		\Delta_{S}\ar[r]_{\Delta} & S\times S
	}
	\end{displaymath}
	Then $f$ lifts to a natural involution on $\tilde{S^{2}}$ and the quotient is
	$$q: \tilde{S^{2}}\surj S^{[2]}.$$
	On the one hand, $\CH^{*}(S^{[2]})$ is identified, \emph{via} $q^{*}$, with the
	invariant part of $\CH^{*}(\tilde{S^{2}})$\,; on the other hand, by Definition
	\ref{def:OrbChow}, $\CH_{orb}^{*}([S^{2}/\gS_{2}])=\CH^{*}(S^{2},
	\gS_{2})^{inv}$. Therefore to check the MHRC \ref{conj:MHRCbis} or
	\ref{conj:MHRCdt} (at the level of Chow rings \emph{only}) in this case, we only
	have to show the following
	\begin{prop}\label{prop:Hilb2}
		We have an isomorphism of $\C$-algebras\,:
		$\CH^{*}(\tilde{S^{2}})_{\C}\isom \CH^{*}(S^{2}, \gS_{2})_{\C}$. In fact, taking
		into account the discrete torsion, there is an isomorphism of $\Q$-algebras
		$\CH^{*}(S^{[2]})\isom \CH^{*}_{orb,dt}([S^{2}/\gS_{2}])$.
	\end{prop}
	\begin{proof}
		A straightforward computation using (iii) and (iv) of Definition
		\ref{def:OrbMot} shows that all obstruction bundles are trivial (at least in the
		Grothendieck group). Hence by Definition \ref{def:OrbChow}, $$\CH^{*}(S^{2},
		\gS_{2})=\CH^{*}(S^{2})\oplus \CH^{*-1}(\Delta_{S})$$
		whose ring structure is explicitly given by
		\begin{itemize}
			\item For any $\alpha\in \CH^{i}(S^{2}), \beta\in\CH^{j}(S^{2})$,
			$\alpha\star_{orb} \beta=\alpha\cdot \beta \in \CH^{i+j}(S^{2})$\,;
			\item For any $\alpha\in \CH^{i}(S^{2}), \beta\in \CH^{j}(\Delta_{S})$,
			$\alpha\star_{orb}\beta=\alpha|_{\Delta}\cdot \beta\in
			\CH^{i+j}(\Delta_{S})$\,;
			\item For any $\alpha\in \CH^{i}(\Delta_{S}), \beta\in\CH^{j}(\Delta_{S})$,
			$\alpha\star_{orb}\beta=\Delta_{*}(\alpha\cdot \beta)\in \CH^{i+j+2}(S^{2})$.
		\end{itemize}
		The blow-up formula (\cf for example, \cite[Theorem 9.27]{MR1997577}) provides
		an \apriori only additive isomorphism 
		$$(\epsilon^{*}, j_{*}\pi^{*}): \CH^{*}(S^{2})\oplus
		\CH^{*-1}(\Delta_{S})\lra{\isom} \CH^{*}(\tilde{S^{2}}),$$ whose inverse is
		given by $(\epsilon_{*}, -\pi_{*}j^{*})$.\\
		With everything given explicitly as above, it is straightforward to check that
		this isomorphism respects also the multiplication up to a sign change\,:
		\begin{itemize}
			\item For any $\alpha\in \CH^{i}(S^{2}), \beta\in\CH^{j}(S^{2})$, one has
			$\epsilon^{*}(\alpha\star_{orb}
			\beta)=\epsilon^{*}(\alpha\cdot\beta)=\epsilon^{*}(\alpha)\cdot\epsilon^{*}(\beta)$\,;
			\item For any $\alpha\in \CH^{i}(S^{2}), \beta\in \CH^{j}(\Delta_{S})$, the
			projection formula yields
			$$j_{*}\pi^{*}(\alpha\star_{orb}\beta)=j_{*}\pi^{*}(\alpha|_{\Delta}\cdot
			\beta)=j_{*}\left(j^{*}\epsilon^{*}(\alpha)\cdot
			\pi^{*}\beta\right)=\epsilon^{*}(\alpha)\cdot j_{*}\pi^{*}(\beta)\,;$$
			\item For any $\alpha\in \CH^{i}(\Delta_{S}), \beta\in\CH^{j}(\Delta_{S})$, we
			make a sign change\,: $\alpha\star_{orb,dt}\beta=-\Delta_{*}(\alpha\cdot
			\beta)$ and we get $$j_{*}\pi^{*}(\alpha)\cdot
			j_{*}\pi^{*}(\beta)=j_{*}\left(j^{*}j_{*}\pi^{*}\alpha\cdot
			\pi^{*}\beta\right)=j_{*}\left(c_{1}(N_{E/\tilde{S^{2}}})\cdot\pi^{*}\alpha\cdot\pi^{*}\beta\right)=-\epsilon^{*}\Delta_{*}(\alpha\cdot
			\beta)=\epsilon^{*}(\alpha\star_{orb,dt}\beta),$$ where in  the last but one
			equality one uses the excess intersection formula for the blowup diagram
			together with the fact that $N_{E/\tilde{S^{2}}}=\sO_{\pi}(-1)$ while the excess
			normal bundle is 
			$$\pi^{*}T_{S}/\sO_{\pi}(-1)\isom
			T_{\pi}\otimes\sO_{\pi}(-1)\isom\sO_{\pi}(1),$$ where one uses the assumption
			that $K_{S}=0$ to deduce that $T_{\pi}\isom\sO_{\pi}(2)$.
		\end{itemize}
		As the sign change is exactly the one given by discrete torsion (Definition
		\ref{def:dt}), we have an isomorphism of $\Q$-algebras 
		$$\CH^{*}(S^{[2]})\isom \CH^{*}_{orb,dt}([S^{2}/\gS_{2}]).$$
		By Remark \ref{rmk:roots}, this yields, without making any sign change, an 
		isomorphism of $\C$-algebras\,:
		$$\CH^{*}(S^{[2]})_{\C}\isom \CH^{*}_{orb}([S^{2}/\gS_{2}])_{\C},$$
		which concludes the proof. \qedhere
	\end{proof}

	\subsubsection{Kummer K3 surfaces}
	Let $A$ be an abelian surface. We always identify $A_{0}^{2}:=\Ker\left(A\times
	A\lra{+} A\right)$ with $A$ by $(x, -x)\mapsto x$. Under this identification,
	the associated Kummer K3 surface $S:=K_{1}(A)$ is a hyperK\"ahler crepant
	resolution of the symplectic quotient $A/f$, where $f$ is the involution of
	multiplication by $-1$ on $A$. Consider the blow-up of $A$ along the fixed locus
	$F$ which is the set of 2-torsion points of $A$\,:
	\begin{displaymath}
	\xymatrix{
		E\ar[r]^{j} \ar[d]_{\pi} \cart & \tilde{A} \ar[d]^{\epsilon} \\
		F  \ar[r]_{i} & A.
	}
	\end{displaymath}
	Then $S$ is the quotient of $\tilde A$ by $\tilde f$, the lifting of the
	involution $f$. As in the previous toy example, the MHRC at the level of Chow
	rings \emph{only} in the present situation is reduced to the following
	
	\begin{prop}
		We have an isomorphism of $\C$-algebras\,:
		$\CH^{*}(\tilde{A})_{\C}\isom \CH^{*}(A, \gS_{2})_{\C}$.  In fact, taking into
		account the discrete torsion, there is an isomorphism of $\Q$-algebras
		$\CH^{*}(K_{1}(A))\isom \CH^{*}_{orb,dt}\left(\left[A/\gS_{2}\right]\right)$.
	\end{prop}
	\begin{proof}
		As the computation is quite similar to that of Proposition \ref{prop:Hilb2}, we
		only give a sketch. 
		By Definition \ref{def:OrbChow}, $\age(\id)=0$, $\age(\tilde f)=1$ and
		$\CH^{*}(A, \gS_{2})=\CH^{*}(A)\oplus \CH^{*-1}(F)$
		whose ring structure is given by
		\begin{itemize}
			\item For any $\alpha\in \CH^{i}(A), \beta\in\CH^{j}(A)$, $\alpha\star_{orb}
			\beta=\alpha\bullet \beta \in \CH^{i+j}(A)$\,;
			\item For any $\alpha\in \CH^{i}(A), \beta\in \CH^{0}(F)$,
			$\alpha\star_{orb}\beta=\alpha|_{F}\bullet \beta\in \CH^{i}(F)$\,;
			\item For any $\alpha\in \CH^{0}(F), \beta\in\CH^{0}(F)$,
			$\alpha\star_{orb}\beta=i_{*}(\alpha\bullet \beta)\in \CH^{2}(A)$.
		\end{itemize}
		Again by the blow-up formula, we have an isomorphism
		$$(\epsilon^{*}, j_{*}\pi^{*}): \CH^{*}(A)\oplus \CH^{*-1}(F)\lra{\isom}
		\CH^{*}(\tilde{A}),$$ whose inverse is given by $(\epsilon_{*}, -\pi_{*}j^{*})$.
		It is now straightforward to check that they are moreover ring isomorphisms with
		the left-hand side equipped with the orbifold product. The sign change comes
		from the negativity of the self-intersection of (the components of) the
		exceptional divisor.
	\end{proof}

	\section{Main results and steps of the proofs}\label{sec:skeleton}
	
	The main results of the paper are the verification of Conjecture
	\ref{conj:MHRCdt}, hence Conjecture \ref{conj:MHRCbis} by Remark
	\ref{rmk:roots}, in the following two cases (A) and (B). See Theorem
	\ref{thm:mainAb} and Theorem \ref{thm:mainKummer} in the introduction for the
	precise statements. These two theorems are proved in \S\ref{sec:proof-A} and
	\S\ref{sec:proof-B} respectively. 
	
	Let $A$ be an abelian surface and $n$ be a positive integer.
	\begin{itemize}
		\item[\textbf{Case (A)}](Hilbert schemes of abelian surfaces)\\
		$M=A^{n}$ endowed with the natural action of $G=\gS_{n}$. The symmetric product
		$A^{(n)}=M/G$ is a singular symplectic variety and the Hilbert--Chow morphism
		$$\rho: Y=A^{[n]}\to A^{(n)}$$ gives a symplectic resolution. 
		\item[\textbf{Case (B)}](Generalized Kummer varieties)\\
		$M=A^{n+1}_{0}:=\Ker\left(A^{n+1}\lra{s} A\right)$ endowed with the natural
		action of $G=\gS_{n+1}$. The quotient $A^{n+1}_{0}/\gS_{n+1}=M/G$ is a singular
		symplectic variety. Recall that the \emph{generalized Kummer variety} $K_{n}(A)$
		is the fiber over $O_{A}$ of the isotrivial fibration $A^{[n+1]}\to
		A^{(n+1)}\lra{s} A$. The restriction of the Hilbert--Chow morphism
		$$Y=K_{n}(A)\to A^{n+1}_{0}/\gS_{n+1}$$ gives a symplectic resolution. 
	\end{itemize}

	Let us deduce the KHRC \ref{conj:KHRC} in these two cases from our main
	results\footnote{Note that our proof for KHRC passes through Chow rings, thus a
		direct geometric (sheaf-theoretic) description of the isomorphism between
		$K(Y)_{\C}$ and $K_{orb}([M/G])_{\C}$ is still missing.}\,:
	\begin{proof}[Proof of Theorem \ref{thm:KHRC}]
		Let $M$ and $G$ be either as in Case (A) or Case (B) above. Without using
		discrete torsion, we have an isomorphism of $\C$-algebras
		$\CH^{*}(M)_{\C}\isom \CH^{*}_{orb}([M/G])_{\C}$ by Theorems \ref{thm:mainAb} and \ref{thm:mainKummer}. An \emph{orbifold Chern character}
		is constructed in \cite{MR2285746}, which by \cite[Main result 3]{MR2285746}
		provides an isomorphism of $\Q$-algebras\,: $$\ch_{orb}:
		K_{orb}\left([M/G]\right)_{\Q}\lra{\isom}
		\CH^{*}_{orb}\left([M/G]\right)_{\Q}.$$ The desired isomorphism of algebras is
		then obtained by the composition of $\ch_{orb}$ (tensored with $\C$), the Chern
		character isomorphism $\ch: K\left(Y\right)_{\Q}\lra{\isom}
		\CH^{*}\left(Y\right)_{\Q}$ tensored with $\C$, and the isomorphism
		$\CH^{*}(M)_{\C}\isom \CH_{orb}^{*}([M/G])_{\C}$ from our main results.\\
		Similarly, for topological K-theory one uses the orbifold topological Chern
		character, which is also constructed in \cite{MR2285746}, $$\ch_{orb}:
		K^{top}_{orb}\left([M/G]\right)_{\Q}\lra{\isom} H^{*}_{orb}\left([M/G],
		\C\right),$$ together with $\ch: K^{top}\left(Y\right)_{\Q}\lra{\isom}
		H^{*}\left(Y, \Q\right)$ and the Cohomological HyperK\"ahler Resolution
		Conjecture\,:
		$$H^{*}_{orb}\left([M/G], \C\right)\isom H^{*}\left(Y, \C\right),$$
		which is proved in Case (A) in \cite{MR1971293} and \cite{MR2154668} based on
		\cite{MR1974889} and in Case (B) in Theorem \ref{cor:CHRCKummer}.
	\end{proof}

	In the rest of this section, we explain the main steps of the proofs of Theorem
	\ref{thm:mainAb} and Theorem \ref{thm:mainKummer}.
	For both cases, the proof proceeds in three steps. For each step, Case (A) is
	quite straightforward and Case (B) requires more subtle and technical arguments.

	\noindent{\textbf{Step (i)}.}
	
	Recall the notation $\h(M, G):=\oplus_{g\in G}\h(M^{g})(-\age(g))$. Denote by
	$$\iota:\h\left(M, G\right)^{G}\inj \h\left(M, G\right)\quad  \text{and}\quad p:
	\h\left(M, G\right)\surj \h\left(M, G\right)^{G}$$ the inclusion of and the
	projection onto the $G$-invariant part $\h\left(M, G\right)^{G}$, which is a
	direct factor of $\h\left(M, G\right)$ inside $\CHM$. We will first construct an
	\apriori just additive $G$-equivariant morphism of Chow motives $\h(Y)\to
	\h(M,G)$, given by some correspondences $\left\{(-1)^{\age(g)}U^{g}\in
	\CH(Y\times M^{g})\right\}_{g\in G}$ inducing an (additive) isomorphism
	$$\phi=p\circ\sum_{g}(-1)^{\age(g)}U^{g}: \h(Y)\lra{\isom}\h_{orb}([M/G])=\h(M,
	G)^{G}.$$ 
	The isomorphism $\phi$ will have the property that its inverse is
	$\psi:=(\frac{1}{|G|}\sum_{g}{}^{t}U^{g})\circ\iota$ (see Proposition
	\ref{prop:addisom} and Proposition \ref{prop:addisom-B} for Case (A) and (B)
	respectively). Note that since $\sum_{g}(-1)^{\age(g)}U^{g}$ is $G$-equivariant,
	we have $\iota \circ \phi = \sum_{g}(-1)^{\age(g)}U^{g}$ and likewise $\psi
	\circ p  = \frac{1}{|G|}\sum_{g}{}^{t}U^{g}$. Our goal is then to prove that
	these morphisms are moreover multiplicative (after the sign change by discrete
	torsion), \ie the following diagram is commutative:
	\begin{equation}\label{diag:mult1}
	\xymatrix{
		\h(Y)^{\otimes 2} \ar[d]_{\phi^{\otimes 2}}\ar[r]^{\delta_{Y}} & \h(Y)
		\ar[d]^{\phi}\\
		\h_{orb}([M/G])^{\otimes 2} \ar[r]_{\star_{orb,dt}} & \h_{orb}([M/G]),
	}
	\end{equation}
	where the algebra structure $\star_{orb,dt}$ on the Chow motive
	$\h_{orb}([M/G])$ is the symmetrization of the algebra structure
	$\star_{orb,dt}$ on $\h(M,G)$ defined in Definition \ref{def:dt} (in the same
	way that the algebra structure $\star_{orb}$ on the Chow motive
	$\h_{orb}([M/G])$ is the symmetrization of the algebra structure $\star_{orb}$
	on $\h(M,G)$\,; see Remark \ref{rmk:additiveOrbMot}).
	
	The main theorem will then be deduced from the following
	\begin{prop}\label{prop:Multiplicativity}
		Notation being as before, the following two algebraic cycles have the same
		symmetrization in $\CH\left(\left(\coprod_{g\in G}M^{g}\right)^{3}\right)$\,:
		\begin{itemize}
			\item
			$W:=\left(\frac{1}{|G|}\sum_{g}U^{g}\times\frac{1}{|G|}\sum_{g}U^{g}\times\sum_{g}(-1)^{\age(g)}U^{g}\right)_{*}\left(\delta_{Y}\right)$\,;
			\item The algebraic cycle $Z$ determining the orbifold product (Definition
			\ref{def:OrbMot}(v)) with the sign change by discrete torsion (Definition
			\ref{def:dt})\,: 
			$$Z|_{M^{g_{1}}\times M^{g_{2}}\times M^{g_{3}}}=\begin{cases}
			0 &\text{     if     } g_{3}\neq g_{1}g_{2}\\
			(-1)^{\epsilon(g_{1},g_{2})}\cdot\delta_{*}c_{top}(F_{g_{1}, g_{2}}) &\text{    
				if    } g_{3}=g_{1}g_{2}.
			\end{cases}
			$$
		\end{itemize}
		Here the \emph{symmetrization} of a cycle in $\left(\coprod_{g\in
			G}M^{g}\right)^{3}$ is the operation $$\gamma\mapsto (p\otimes p \otimes
		p)_*\gamma =  \frac{1}{|G|^{3}}\sum_{g_{1}, g_{2}, g_{3}\in G}(g_{1}, g_{2},
		g_{3})_{.}\gamma.$$
	\end{prop}
	
	\begin{proof}[Proposition \ref{prop:Multiplicativity} implies Theorems
		\ref{thm:mainAb} and \ref{thm:mainKummer}]
		The only thing to show is the commutativity of (\ref{diag:mult1}), 
		which is of course equivalent to the commutativity of the diagram
		\begin{equation*}
		\xymatrix{
			\h(Y)^{\otimes 2} \ar[r]^{\delta_{Y}} & \h(Y) \ar[d]^{\phi}\\
			\h_{orb}([M/G])^{\otimes 2}\ar[u]^{\psi^{\otimes 2}} \ar[r]_{\star_{orb,dt}} &
			\h_{orb}([M/G])
		}
		\end{equation*}
		By the definition of $\phi$ and $\psi$, we need to show that the following
		diagram is commutative\,: 
		\begin{equation}\label{diag:mult3}
		\xymatrix{
			\h(Y)^{\otimes 2} \ar[r]^{\delta_{Y}} & \h(Y)
			\ar[d]^{\sum_{g}(-1)^{\age(g)}U^{g}}\\
			\h(M, G)^{\otimes 2}\ar[u]^{(\frac{1}{|G|}\sum_{g}{}^{t}U^{g})^{\otimes 2}}  &
			\h(M, G)\ar[d]^{p}\\
			\h_{orb}([M/G])^{\otimes 2}\ar[u]^{\iota^{\otimes 2}} \ar[r]_{\star_{orb,dt}} &
			\h_{orb}([M/G])
		}
		\end{equation}
		It is elementary to see that the composition
		$\sum_{g}(-1)^{\age(g)}U^{g}\circ\delta_{Y}\circ
		(\frac{1}{|G|}\sum_{g}{}^{t}U^{g})^{\otimes 2}$ is the morphism (or
		correspondence) induced by the cycle $W$ in Proposition
		\ref{prop:Multiplicativity}\,; see \emph{e.g.} \cite[Lemma~3.3]{Vialabelian}. On
		the other hand,  $\star_{orb, dt}$ for $\h_{orb}([M/G])$ is by definition
		$p\circ Z\circ\iota^{\otimes 2}$. Therefore, the desired commutativity, hence
		also the main results, amounts to the equality $p\circ W\circ\iota^{\otimes
			2}=p\circ Z\circ\iota^{\otimes 2}$, which says exactly that the symmetrizations
		of $W$ and of $Z$ are equal in $\CH\left(\left(\coprod_{g\in
			G}M^{g}\right)^{3}\right)$.
	\end{proof}
	One is therefore reduced to show Proposition \ref{prop:Multiplicativity} in both
	cases (A) and (B).  
	
	\noindent\textbf{Step (ii).}
	
	We prove that $W$ on the one hand and $Z$ on the other hand, as well as their
	symmetrizations, are both symmetrically distinguished in the sense of O'Sullivan
	\cite{MR2795752} (see Definition \ref{def:SD}). To avoid confusion, let us point
	out that the cycle $W$ is already symmetrized. In Case (B) concerning the
	generalized Kummer varieties, we have to generalize the category of
	abelian varieties and the corresponding notion of symmetrically distinguished
	cycles, in order to deal with algebraic cycles on `non-connected abelian
	varieties' in a canonical way. By the result of O'Sullivan \cite{MR2795752} (see
	Theorem \ref{thm:SD} and Theorem \ref{thm:SD2}), it suffices for us to check
	that the symmetrizations of $W$ and $Z$ are numerically equivalent. 
	
	\noindent\textbf{Step (iii).}
	
	Finally, in Case (A), explicit computations of the cohomological realization of
	$\phi$ show that the induced (iso-)morphism $\phi: H^{*}(Y) \to
	H^{*}_{orb}([M/G])$ is the same as the one constructed in \cite{MR1974889}.
	While in Case (B),  based on the result of \cite{MR2578804}, one can prove that
	the cohomological realization of $\phi$ satisfies Ruan's original Cohomological
	HyperK\"ahler Resolution Conjecture. Therefore the symmetrizations of $W$ and
	$Z$ are homologically equivalent, which finishes the proof by Step (ii).
	
	\section{Case (A)\,: Hilbert schemes of abelian surfaces}\label{sec:proof-A}
	We prove Theorem \ref{thm:mainAb} in this section. Notations are as before\,:
	$M:=A^{n}$ with the action of $G:=\gS_{n}$ and the quotient $A^{(n)}:=M/G$. Then
	the Hilbert--Chow morphism $$\rho: A^{[n]}=:Y\to A^{(n)}$$ gives a symplectic
	resolution. 
	
	\subsection{A recap of $\gS_{n}$-equivariant geometry}
	To fix the convention and terminology, let us collect here a few basic facts
	concerning $\gS_n$-equivariant geometry\,:
	\subsubsection{} The conjugacy classes of the group $\gS_n$ consist of
	permutations of the same cycle type\,; hence the conjugacy classes are in
	bijection to partitions of $n$. The number of disjoint cycles whose composition
	is $g\in \gS_n$ is exactly the number $|O(g)|$ of orbits in $\{1,\ldots,n\}$
	under the permutation action of $g\in \gS_n$. We will say that $g\in \gS_{n}$ is
	\emph{of partition type} $\lambda$, denoted by  $g\in \lambda$, if the partition
	determined by $g$ is $\lambda$. 
	\subsubsection{}  Let  $X$ be a variety of pure dimension $d$. Given a
	permutation $g\in \gS_n$, the fixed locus $(X^n)^g:= \mathrm{Fix}_g(X^n)$ can be
	described explicitly as the following partial diagonal
	$$(X^n)^g = \left\{ (x_1,\ldots, x_n) \in X^n \ |\ x_i =x_j \text{ if $i$ and $j$
		belong to the same orbit under the action of $g$}\right\}.$$
	As in \cite{MR1971293},  we therefore have the natural identification 
	$$(X^{n})^{g}= X^{O(g)}.$$
	In particular, the codimension of $(X^n)^g$ in $X^n$ is $d(n-|O(g)|)$.
	\subsubsection{} \label{sec:agesym} Since $g$ and $g^{-1}$ belong to the same
	conjugacy class, it follows from $\age (g) + \age (g^{-1}) = \codim ((X^n)^g\subset X^{n})$
	that 
	\begin{equation*}\label{eq:agesym}
	\age (g) = \frac{d}{2}(n-|O(g)|),
	\end{equation*} as was stated in  Example \ref{expl:age}.
	\subsubsection{}  \label{sec:stab} Let $\sP(n)$ be the set of partitions of $n$.
	Given such a partition $\lambda=(\lambda_{1}\geq \cdots \geq
	\lambda_{l})=(1^{a_{1}}\cdots n^{a_{n}})$, where $l:=|\lambda|$ is the
	\emph{length} of $\lambda$ and $a_{i}=|\{j~|~ 1\leq j\leq n\,;\lambda_{j}=i\}|$,
	we define $\gS_{\lambda}:=\gS_{a_{1}}\times\cdots\times\gS_{a_{n}}$. For $g\in
	\gS_n$ a permutation of partition type $\lambda$, its centralizer $C(g)$, \ie
	the stabilizer under the action of $\gS_n$ on itself by conjugation, is
	isomorphic to the semi-direct product:
	$$C(g)\isom\left(\Z/\lambda_{1}\times\cdots\times\Z/\lambda_{l}\right)\rtimes
	\gS_{\lambda}.$$
	Note that the action of $C(g)$ on $X^n$ restricts to an action on $(X^n)^g =
	X^{O(g)} \simeq X^l$ and the action of the normal subgroup
	$\Z/\lambda_{1}\times\cdots\times\Z/\lambda_{l} \subseteq C(g)$ is trivial.
	We denote the quotient $X^{(\lambda)}:=(X^n)^g/C(g)=(X^n)^g/\gS_\lambda$, and
	we regard the motive $\h\left(X^{(\lambda)}\right)$ as the  direct summand
	$\h\left((X^n)^g\right)^{\gS_\lambda}$ inside $\h\left((X^n)^g\right)$ \emph{via} the pull-back along
	the projection $(X^n)^g \to (X^n)^g/\gS_\lambda$\,; see Remark
	\ref{rmks:MotiveGeneral}(iii).

\subsection{Step (i) -- Additive isomorphisms}\label{subsec:AddIso-A}
	
	In this subsection, we establish an isomorphism between $\h(Y)$ and
	$\h_{orb}([M/G])$ by using results of \cite{MR1919155}, and more specifically by
	constructing correspondences similar to the ones used therein.
	
	Let 
	\begin{equation}\label{eqn:U}
	U^{g}:=(A^{[n]}\times_{A^{(n)}}(A^{n})^{g})_{red}=\left\{(z,
	x_{1},\cdots,x_{n})\in A^{[n]}\times(A^{n})^{g} ~\middle\vert~
	\rho(z)=[x_{1}]+\cdots+[x_{n}]\right\}
	\end{equation}
	be the incidence variety, where $\rho: A^{[n]}\to A^{(n)}$ is the Hilbert--Chow
	morphism. As the notation suggests, $U^{g}$ is the fixed locus of the induced
	automorphism $g$ on the \emph{isospectral Hilbert scheme}
	$$U:=U^{\id}=A^{[n]}\times_{A^{(n)}}A^{n}=\left\{(z, x_{1},\cdots,x_{n})\in
	A^{[n]}\times A^{n} ~\middle\vert~ \rho(z)=[x_{1}]+\cdots+[x_{n}]\right\}.$$ 
	Note that $\dim U^{g}=n+|O(g)|=2n-\age(g)$ (\cite{MR0457432}) and
	$\dim\left(A^{[n]}\times (A^{n})^{g}\right)=2\dim U^{g}$. We consider the
	following correspondence for each $g\in G$,
	\begin{equation}\label{eqn:Gammag}
	\Gamma_{g}:=(-1)^{\age(g)}U^{g}\in \CH^{2n-\age(g)}\left(A^{[n]}\times
	(A^{n})^{g}\right),
	\end{equation}
	which defines a morphism of Chow motives\,:
	\begin{equation}\label{eqn:Gamma1}
	\Gamma:=\sum_{g\in G}\Gamma_{g}: \h(A^{[n]})\to \bigoplus_{g\in
		G}\h\left((A^{n})^{g}\right)(-\age(g))=:\h(A^{n}, \gS_{n}),
	\end{equation}
	where we used the notation from Definition \ref{def:OrbMot}.
	\begin{lemma}\label{lemma:GInv}
		The algebraic cycle $\Gamma$ in (\ref{eqn:Gamma1}) defines an
		$\gS_{n}$-equivariant morphism with respect to the trivial action on $A^{[n]}$
		and the action on $\h(A^{n}, \gS_{n})$ of Definition \ref{def:OrbMot}.
		
	\end{lemma}
	\begin{proof}
		For each $g, h\in G$, as the age function is invariant under conjugation, it
		suffices to show that the following composition is equal to
		$\Gamma_{hgh^{-1}}$\,: 
		$$\h(A^{[n]})\lra{\Gamma_{g}}
		\h\left((A^{n})^{g}\right)(-\age(g))\lra{h}\h\left((A^{n})^{hgh^{-1}}\right)(-\age(g)).$$
		This follows from the the fact that the following diagram
		\begin{displaymath}
		\xymatrix{
			A^{[n]} & U^{g}\ar[l] \ar[d]\ar[r]^{h}_{\isom} & U^{hgh^{-1}}\ar[d]\\
			& (A^{n})^{g} \ar[r]^{h}_{\isom} & (A^{n})^{hgh^{-1}}
		}
		\end{displaymath}
		is commutative.
	\end{proof}
	As before, $\iota:\h\left(A^{n}, G\right)^{G}\inj \h\left(A^{n}, G\right) $ and
	$p: \h\left(A^{n}, G\right)\surj \h\left(A^{n}, G\right)^{G}$ are the inclusion
	of and the projection onto the $G$-invariant part. Thanks to Lemma
	\ref{lemma:GInv}, we obtain the desired morphism 
	\begin{equation}\label{eqn:phi}
	\phi:=p\circ\Gamma: \h(A^{[n]})\to \h_{orb}([A^{n}/G])=\h\left(A^{n},
	G\right)^{G},
	\end{equation} 
	which satisfies $\Gamma=\iota\circ\phi$.

	Now one can reformulate the result of de Cataldo--Migliorini \cite{MR1919155},
	which actually works for all surfaces, as follows\,:
	\begin{prop}\label{prop:addisom}
		The morphism $\phi$ is an isomorphism, whose inverse is given by
		$\psi:=\frac{1}{n!}\left(\sum_{g\in G}{}^{t}U^{g}\right)\circ\iota$, where
		$${}^{t}U^{g}: \h\left((A^{n})^{g}\right)(-\age(g)) \to \h(A^{[n]})$$ is the
		transposed correspondence of $U^{g}$.
	\end{prop}
	\begin{proof}
		Let $\lambda=(\lambda_{1}\geq \cdots \geq \lambda_{l}) \in \sP(n)$ be a
		partition of $n$ of length $l$
		and let $A^{\lambda}$ be $A^{l}$, equipped with the natural action of
		$\gS_{\lambda}$ and with the natural morphism to $A^{(n)}$ by sending $(x_{1},
		\cdots, x_{l})$ to $\sum_{j=1}^{l}\lambda_{j}[x_{j}]$. Define the incidence
		subvariety $U^{\lambda}:=(A^{[n]}\times_{A^{(n)}}A^{\lambda})_{red}$. Denote the
		quotient $A^{(\lambda)}:=A^{\lambda}/\gS_{\lambda}$
		and $U^{(\lambda)}:=U^{\lambda}/\gS_{\lambda}$, where the latter is also
		regarded as a correspondence between $A^{[n]}$ and $A^{(\lambda)}$. See
		Remark~\ref{rmks:MotiveGeneral}(iii) for the use of Chow motives of global
		quotients, and \eqref{sec:stab} for our case at hand (\ie, $A^{(\lambda)}$).\\
		The main theorem in \cite{MR1919155} asserts that the following correspondence
		is an isomorphism\,:
		$$\phi':=\sum_{\lambda\in \sP(n)} U^{(\lambda)}: \h(A^{[n]})\lra{\isom}
		\bigoplus_{\lambda\in \sP(n)}\h(A^{(\lambda)})(|\lambda|-n)\,;$$
		moreover, the inverse of $\phi'$ is given by 
		$$\psi':=\sum_{\lambda\in \sP(n)} \frac{1}{m_{\lambda}}\cdot{}^{t}U^{(\lambda)}:
		\bigoplus_{\lambda\in
			\sP(n)}\h(A^{(\lambda)})(|\lambda|-n)\lra{\isom}\h(A^{[n]}),$$ where
		$m_{\lambda}=(-1)^{n-|\lambda|}\prod_{j=1}^{|\lambda|}\lambda_{j}$ is a non-zero
		constant.
		To relate our morphism $\phi$ to the above isomorphism $\phi'$ as well as their
		inverses, one uses the following elementary
		\begin{lemma}\label{lemma:relate}
			One has a natural isomorphism\,:
			\begin{equation}\label{eq:canisoHilb}
			\left(\bigoplus_{g\in
				\gS_{n}}\h\left((A^{n})^{g}\right)(-\age(g))\right)^{\gS_{n}}\lra{\isom}\bigoplus_{\lambda\in
				\sP(n)}\h\left(A^{(\lambda)}\right)\left(|\lambda|-n\right).
			\end{equation}
		\end{lemma}
		\begin{proof} 
			By regrouping permutations by their partition types, we clearly have 
			$$\left(\bigoplus_{g\in
				\gS_{n}}\h\left((A^{n})^{g}\right)(-\age(g))\right)^{\gS_{n}}\cong
			\bigoplus_{\lambda\in \sP(n)}\left(\bigoplus_{g\in
				\lambda}\h\left((A^{n})^{g}\right)\right)^{\gS_{n}}\left(|\lambda|-n\right).$$
			So it suffices to give a natural isomorphism, for any fixed partition
			$\lambda\in \sP(n)$, between $\left(\bigoplus_{g\in
				\lambda}\h\left((A^{n})^{g}\right)\right)^{\gS_{n}}$ and
			$\h\left(A^{(\lambda)}\right)$. However, such an isomorphism of motives follows
			from the following isomorphism of quotient varieties\,:
			\begin{equation}\label{eqn:QuotientIsom}
			\left.\left(\coprod_{g\in \lambda} (A^{n})^{g}\right)
			~\middle/~\gS_{n}\right.\cong A^{\lambda}/\gS_{\lambda}=A^{(\lambda)},
			\end{equation}
			where the first isomorphism can be obtained by choosing a permutation $g_0\in
			\lambda$ and observing that the centralizer of $g_{0}$ is isomorphic to the
			semi-direct product
			$\left(\Z/\lambda_{1}\times\cdots\times\Z/\lambda_{l}\right)\rtimes
			\gS_{\lambda}$, where the normal subgroup
			$\Z/\lambda_{1}\times\cdots\times\Z/\lambda_{l}$ acts trivially. We remark that
			there are some other natural choices for the isomorphism in
			\eqref{eq:canisoHilb}, due to different points of view and convention\,; but
			they only differ from ours by a non-zero constant.
		\end{proof}
		
		Now it is easy to conclude the proof of Proposition \ref{prop:addisom}. The idea
		is to relate de Cataldo--Migliorini's isomorphisms $\phi'$, $\psi'$ recalled
		above to our morphisms $\phi$ and $\psi$. Given a partition $\lambda\in \sP(n)$,
		for any $g\in \lambda$, the isomorphism between $(A^{n})^{g}$ and $A^{\lambda}$
		will identify $U^{g}$ to $U^{\lambda}$. We have the following commutative
		diagram
		\begin{equation*}
		\xymatrix{
			\coprod_{g\in \lambda} U^{g} \ar[r]\ar@{->>}[d]^{q}& \coprod_{g\in \lambda}
			(A^{n})^{g}\ar@{->>}[d]^{q}\\
			U^{(\lambda)} \ar[r]\ar[d]& A^{(\lambda)}\\
			A^{[n]} &
		}
		\end{equation*}
		where the degree of the two quotient-by-$\gS_{n}$ morphisms $q$ are easily
		computed\,:
		$\deg(q)=\frac{n!}{\prod_{j=1}^{|\lambda|}\lambda_{j}}$.
		The natural isomorphism \eqref{eq:canisoHilb}  of Lemma \ref{lemma:relate} is
		simply given by 
		$$\frac{1}{\deg q}q_{*}\circ \iota\colon \h\left(\coprod_{g\in
			\lambda}(A^{n})^{g}\right)^{\gS_{n}}\lra{\isom} \h\left(A^{(\lambda)}\right),$$
		with inverse given by $p\circ q^{*}$ (in fact the image of $q^{*}$ is already
		$\gS_{n}$-invariant).
		Therefore the composition of $\phi$ with the natural isomorphism
		\eqref{eq:canisoHilb} is equal to  
		$$\sum_{\lambda\in \sP(n)}\left(\frac{1}{\deg q}q_{*}\circ \sum_{g\in \lambda}
		(-1)^{\age(g)}U^{g}\right)=\sum_{\lambda\in \sP(n)}\frac{1}{\deg q}q_{*}\circ
		q^{*}\circ (-1)^{|\lambda|-n}U^{(\lambda)}=\sum_{\lambda\in
			\sP(n)}(-1)^{|\lambda|-n}U^{(\lambda)},$$
		where we used the commutative diagram above for the first equality. As a
		consequence, $\phi$ is an isomorphism as $\phi'=\sum_{\lambda\in
			\sP(n)}U^{(\lambda)}$ is one.
		
		Similarly, the composition of the inverse of \eqref{eq:canisoHilb} with $\psi$
		is equal to 
		$$\sum_{\lambda\in \sP(n)}\left(\frac{1}{n!}\sum_{g\in \lambda} {}^{t}U^{g}\circ
		q^{*}\right)=\sum_{\lambda\in \sP(n)}\frac{1}{n!}\cdot {}^{t}U^{(\lambda)}\circ
		q_{*}\circ q^{*}=\sum_{\lambda\in \sP(n)}\frac{\deg
			q}{n!}\cdot{}^{t}U^{(\lambda)}=\sum_{\lambda\in
			\sP(n)}\frac{(-1)^{|\lambda|-n}}{m_{\lambda}}\cdot{}^{t}U^{(\lambda)}.$$
		Since $\psi'=\sum_{\lambda\in
			\sP(n)}\frac{1}{m_{\lambda}}\cdot{}^{t}U^{(\lambda)}$ is the inverse of $\phi'$
		by \cite{MR1919155} recalled above, $\psi$ is the inverse of $\phi$.
	\end{proof}
	
	Then to show Theorem \ref{thm:mainAb}, it suffices to prove Proposition
	\ref{prop:Multiplicativity} in this situation, which will be done in the next
	two steps.
	
	\subsection{Step (ii) -- Symmetrically distinguished cycles on abelian
		varieties}\label{subsec:SD-A}
	The following definition is due to O'Sullivan \cite{MR2795752}. Recall that all
	Chow groups are with rational coefficients. As in \loccit we denote in this
	section by $\bar \CH$ the $\Q$-vector space of algebraic cycles  modulo the
	numerical equivalence relation. 
	\begin{defi}[Symmetrically distinguished cycles \cite{MR2795752}]\label{def:SD}
		Let $A$ be an abelian variety and $\alpha\in \CH^{i}(A)$. For each integer
		$m\geq 0$, denote by $V_{m}(\alpha)$ the $\Q$-vector subspace of $\CH(A^{m})$
		generated by elements of the form
		$$p_{*}(\alpha^{r_{1}}\times \alpha^{r_{2}}\cdots\times \alpha^{r_{n}}),$$
		where $n\leq m$, $r_{j}\geq 0 $ are integers, and $p : A^{n}\to A^{m}$ is a
		closed immersion with each component $A^{n}\to A$ being either a projection or
		the composite of a projection with $[-1]: A\to A$. Then $\alpha$ is called
		\emph{symmetrically distinguished} if for every $m$ the restriction of the
		projection $\CH(A^{m})\to \bar\CH(A^{m})$ to $V_{m}(\alpha)$ is injective.
	\end{defi}
	Despite their seemingly complicated definition, symmetrically distinguished
	cycles behave very well. More precisely, we have
	
	\begin{thm}[O'Sullivan \cite{MR2795752}]\label{thm:SD}
		Let $A$ be an abelian variety.
		\begin{enumerate}
			\item[(i)] The symmetric distinguished cycles in $\CH^{i}(A)$ form a
			sub-$\Q$-vector space.
			\item[(ii)] The fundamental class of $A$ is symmetrically distinguished and the
			intersection product of two symmetrically distinguished cycles is symmetrically
			distinguished. They form therefore a graded sub-$\Q$-algebra of $\CH^{*}(A)$.
			\item[(iii)] Let $f:A \to B$ be a morphism of abelian varieties, then
			$f_{*}:\CH(A)\to \CH(B)$ and $f^{*}:\CH(B)\to \CH(A)$ preserve symmetrically
			distinguished cycles.
		\end{enumerate}
	\end{thm}
	
	The reason why this notion is very useful in practice is that it allows us to
	conclude an equality of algebraic cycles modulo rational equivalence from an
	equality modulo numerical equivalence (or, \emph{a fortiori}, modulo homological
	equivalence)\,:
	\begin{thm}[O'Sullivan \cite{MR2795752}]\label{thm:SD2}
		The composition $\CH(A)_{sd}\inj \CH(A)\surj \bar\CH(A)$ is an isomorphism of
		$\Q$-algebras, where $\CH(A)_{sd}$ is the sub-algebra of symmetrically
		distinguished cycles. In other words, in each numerical class of algebraic cycle
		on $A$, there exists a unique symmetrically distinguished algebraic cycle modulo
		rational equivalence. In particular, a (polynomial of) symmetrically
		distinguished cycles is trivial in $\CH(A)$ if and only if it is numerically
		trivial. 
	\end{thm}
	
	Returning to the proof of Theorem \ref{thm:mainAb}, it remains to prove
	Proposition \ref{prop:Multiplicativity}. Keep the same notation as in Step (i),
	we first prove that in our situation the two cycles in Proposition
	\ref{prop:Multiplicativity} are symmetrically distinguished.
	\begin{prop}\label{prop:ZWsd}
		The following two algebraic cycles, as well as their symmetrizations, 
		\begin{itemize}
			\item
			$W:=\left(\frac{1}{|G|}\sum_{g}U^{g}\times\frac{1}{|G|}\sum_{g}U^{g}\times\sum_{g}(-1)^{\age(g)}U^{g}\right)_{*}\left(\delta_{A^{[n]}}\right)$\,;
			\item The algebraic cycle $Z$ determining the orbifold product (Definition
			\ref{def:OrbMot}(v)) with the sign change by discrete torsion (Definition
			\ref{def:dt})\,: 
			$$Z|_{M^{g_{1}}\times M^{g_{2}}\times M^{g_{3}}}=\begin{cases}
			0 &\text{     if     } g_{3}\neq g_{1}g_{2}\\
			(-1)^{\epsilon(g_{1},g_{2})}\cdot\delta_{*}c_{top}(F_{g_{1}, g_{2}}) &\text{    
				if    } g_{3}=g_{1}g_{2}.
			\end{cases}
			$$
		\end{itemize}
		are symmetrically distinguished in $\CH\left(\left(\coprod_{g\in
			G}(A^{n})^{g}\right)^{3}\right)$.
	\end{prop}
	\begin{proof}
		For $W$, it amounts to show that for any $g_{1}, g_{2}, g_{3}\in G$, we have
		that $\left(U^{g_{1}}\times U^{g_{2}}\times
		U^{g_{3}}\right)_{*}(\delta_{A^{[n]}})$ are symmetrically distinguished in
		$\CH\left((A^{n})^{g_{1}}\times(A^{n})^{g_{2}}\times(A^{n})^{g_{3}}\right)$.
		Indeed, by \cite[Proposition 5.6]{MR3447105}, $\left(U^{g_{1}}\times
		U^{g_{2}}\times U^{g_{3}}\right)_{*}(\delta_{A^{[n]}})$ is a polynomial of big
		diagonals of $(A^{n})^{g_{1}}\times(A^{n})^{g_{2}}\times(A^{n})^{g_{3}}=:A^{N}$.
		However, all big diagonals of $A^{N}$ are clearly symmetrically distinguished
		since $\Delta_{A}\in \CH(A\times A)$ is. By Theorem \ref{thm:SD}, $W$ is
		symmetrically distinguished.\\
		As for $Z$, for any fixed $g_{1}, g_{2}\in G$, $F_{g_{1},g_{2}}$ is easily seen
		to always be a trivial vector bundle, at least virtually, hence its top Chern
		class is either 0 or $1$ (the fundamental class), which is of course
		symmetrically distinguished. Also recall that (Definition \ref{def:OrbMot})
		$$\delta: (A^{n})^{<g_{1}, g_{2}>}\inj (A^{n})^{g_{1}}\times
		(A^{n})^{g_{2}}\times (A^{n})^{g_{1}g_{2}},$$
		which is a (partial) diagonal inclusion, in particular a morphism of abelian
		varieties. Therefore $\delta_{*}(c_{top}(F_{g_{1}, g_{2}}))$ is symmetrically
		distinguished by Theorem \ref{thm:SD}, hence so is $Z$.\\
		Finally, since any automorphism in $G\times G\times G$ preserves symmetrically
		distinguished cycles, symmetrizations of $Z$ and $W$ remain symmetrically
		distinguished.
	\end{proof}
	
	By Theorem \ref{thm:SD2}, in order to show Proposition
	\ref{prop:Multiplicativity}, it suffices to show  on the one hand that the
	symmetrizations of $Z$ and $W$ are both symmetrically distinguished, and on the
	other hand that they are numerically equivalent. The first part is exactly the
	previous Proposition \ref{prop:ZWsd} and we now turn to an \apriori stronger
	version of the second part in the following final step.

	\subsection{Step (iii) -- Cohomological
		realizations}\label{subsec:Realization-A}
	We will show in this subsection that the symmetrizations of the algebraic cycles
	$W$ and $Z$ have the same (rational) cohomology class. To this end,  it is
	enough to show the following
	\begin{prop}\label{prop:realisation-A}
		The cohomology realization of the (additive) isomorphism 
		$$\phi: \h(A^{[n]})\lra{\isom} \left(\oplus_{g\in
			G}\h((A^{n})^{g})(-\age(g))\right)^{\gS_{n}}$$ 
		is an isomorphism of $\Q$-algebras $$\bar\phi: H^{*}(A^{[n]})\lra{\isom}
		H^{*}_{orb,dt}\left([A^{n}/\gS_{n}]\right)=\left(\bigoplus_{g\in
			G}H^{*-2\age(g)}\left((A^{n})^{g}\right), \star_{orb, dt}\right)^{\gS_{n}}.$$
		In other words, $\Sym(W)$ and $\Sym(Z)$ are homologically equivalent.
	\end{prop}
	
	Before we proceed to the proof of Proposition \ref{prop:realisation-A},
	we need to do some preparation on the Nakajima operators (\cf \cite{MR1711344}).
	Let $S$ be a smooth projective surface.
	Recall that given a cohomology class $\alpha\in H^{*}(S)$, the Nakajima operator
	$\p_{k}(\alpha): H^{*}(S^{[r]})\to H^{*}(S^{[r+k]})$, for any $r\in \N$, is by
	definition $\beta\mapsto
	{I_{r;k}}^{*}(\alpha\times\beta):=q_{*}\left(p^{*}(\alpha\times\beta)\cdot
	[I_{r; k}]\right)$, where  $p\colon S^{[r+k]}\times S\times S^{[r]} \to S\times
	S^{[r]}$,  $q\colon S^{[r+k]}\times S\times S^{[r]} \to S^{[r+k]}$ are the
	natural projections and the cohomological correspondence $I_{r;k}$ is defined as
	the unique irreducible component of maximal dimension of the incidence subscheme
	$$\left\{(\xi', x, \xi)\in S^{[r+k]}\times S\times S^{[r]} ~\middle\vert~
	\xi\subset \xi'~;~ \rho(\xi')=\rho(\xi)+k[x]\right\}.$$
	Here and in the sequel, $\rho$ is always the Hilbert--Chow morphism. To the best
	of our knowledge, it is still not known whether the above incidence subscheme is
	irreducible but we do know that there is only one irreducible component with
	maximal dimension ($=2r+k+1$), \cf \cite[\S 8.3]{MR1711344}, \cite[Lemma
	1.1]{MR1681097}.
	
	For our purpose, we need to consider the following generalized version of such
	correspondences in a similar fashion as in \cite{MR1681097}. Following \loccit,
	the short hand $S^{[n_{1}],\cdots,[n_{r}]}$ means the product $S^{[n_{1}]}\times
	\cdots\times S^{[n_{r}]}$. A sequence of $[1]$'s of length $n$ is denoted by
	$[1]^{n}$. For any $r, n, k_{1}, \cdots, k_{n}\in \N$, we consider the closed
	subscheme of $S^{[r+\sum k_{i}], [1]^{n}, [r]}$ whose closed points are given by
	(see \cite{MR1681097} for the natural scheme structure)\,:
	$$J_{r; k_{1}, \cdots, k_{n}}:=\left\{(\xi', x_{1}, \cdots, x_{n}, \xi)
	~\middle\vert~ \xi\subset \xi'~;~ 
	\rho(\xi')=\rho(\xi)+\sum_{i=1}^{n}k_{i}[x_{i}]\right\}.$$
	As far as we know, the irreducibility of $J_{r; k_{1}, \cdots, k_{n}}$ is
	unknown in general, but we will only need its component of maximal dimension. To
	this end, we consider the following locally closed subscheme of $S^{[r+\sum
		k_{i}], [1]^{n}, [r]}$ by adding an open condition\,:
	$$I_{r; k_{1}, \cdots, k_{n}}^{0}:=\left\{(\xi', x_{1}, \cdots, x_{n}, \xi)
	~\middle\vert~ \xi\subset \xi'~;~ x_{i}\text{'s are distinct and disjoint from }
	\xi~;~ \rho(\xi')=\rho(\xi)+\sum_{i}k_{i}[x_{i}]\right\}.$$
	Let $I_{r; k_{1}, \cdots, k_{n}}$ be its Zariski closure. By Brian\c{c}on
	\cite{MR0457432} (\cf also \cite[Lemma 1.1]{MR1681097}), $I_{r; k_{1}, \cdots,
		k_{n}}$ is irreducible of dimension $2r+n+\sum k_{i}$ and it is the unique
	irreducible component of maximal dimension of $J_{r; k_{1}, \cdots, k_{n}}$.
	In particular, the correspondence $I_{r;k}$ used by Nakajima mentioned above is
	the special case when $n=1$. Let us also mention that when $r=0$, we actually
	have that $J_{0; k_{1},\cdots, k_{n}}$ is irreducible (\cite[Remark
	2.0.1]{MR1919155}), and hence is equal to $I_{0; k_{1},\cdots, k_{n}}$.
	
	For any $r, n, m, k_{1}, \cdots, k_{n}, l_{1}, \cdots, l_{m}\in \N$, consider
	the following diagram analogous to the one found on \cite[p.~181]{MR1681097}.
	\begin{equation*}
	\xymatrix{
		S^{[r+\sum k_{i}+\sum l_{j}], [1]^{m},[r+\sum k_{i}]} & S^{[r+\sum k_{i}+\sum
			l_{j}], [1]^{m},[r+\sum k_{i}], [1]^{n}, [r]} \ar[l]_-{p_{123}}\ar[r]^-{p_{345}}
		\ar[d]^-{p_{1245}} & S^{[r+\sum k_{i}], [1]^{n}, [r]}\\
		& S^{[r+\sum k_{i}+\sum l_{j}], [1]^{m+n}, [r]}&
	}
	\end{equation*}
	By a similar argument as in \cite[p.~181]{MR1681097} (actually easier since we
	only need weaker dimension estimates), we see that
	\begin{itemize}
		\item  $p_{1245}$ induces an isomorphism from $p_{123}^{-1}\left(I^{0}_{r+\sum
			k_{i};l_{1}, \cdots, l_{m}}\right) \cap p_{345}^{-1}\left(I^{0}_{r; k_{1},
			\cdots, k_{n}}\right)$ to $I^{0}_{r;k_{1}, \cdots, k_{n}, l_{1}, \cdots,
			l_{m}}$\,;
		\item the complement of $I^{0}_{r;k_{1}, \cdots, k_{n}, l_{1}, \cdots, l_{m}}$
		in $p_{1245}\left(p_{123}^{-1}\left(I^{0}_{r+\sum k_{i};l_{1}, \cdots,
			l_{m}}\right) \cap p_{345}^{-1}\left(I^{0}_{r; k_{1}, \cdots,
			k_{n}}\right)\right)$ is of dimension $<2r+n+m+\sum k_{i}+\sum l_{j}=\dim
		I^{0}_{r;k_{1}, \cdots, k_{n}, l_{1}, \cdots, l_{m}}$\,;
		\item the intersection of $p_{123}^{-1}\left(I^{0}_{r+\sum k_{i};l_{1}, \cdots,
			l_{m}}\right)$ and $p_{345}^{-1}\left(I^{0}_{r; k_{1}, \cdots, k_{n}}\right)$ is
		transversal at the generic point of $p_{123}^{-1}\left(I_{r+\sum k_{i};l_{1},
			\cdots, l_{m}}\right)\cap p_{345}^{-1}\left(I_{r; k_{1}, \cdots, k_{n}}\right)$.
	\end{itemize} 
	Combining these, we have in particular that 
	\begin{equation}\label{eqn:recurrenceI}
	p_{1245,*}\left(p_{123}^{*}\left[I_{r+\sum k_{i};l_{1}, \cdots,
		l_{m}}\right]\cdot p_{345}^{*}\left[I_{r; k_{1}, \cdots,
		k_{n}}\right]\right)=\left[I_{r;k_{1}, \cdots, k_{n}, l_{1}, \cdots,
		l_{m}}\right].
	\end{equation}
	We will only need the case $r=0$ and $m=1$ in the proof of Proposition
	\ref{prop:realisation-A}.

	\begin{proof}[Proof of Proposition \ref{prop:realisation-A}]
		The existence of an isomorphism of $\Q$-algebras between the two cohomology
		rings $H^{*}(A^{[n]})$ and $H^{*}_{orb,dt}([A^{n}/\gS_{n}])$ is established by
		Fantechi and G\"ottsche \cite[Theorem 3.10]{MR1971293} based on the work of Lehn
		and Sorger \cite{MR1974889}. Therefore by the definition of $\phi$ in Step (i),
		it suffices to show that the cohomological correspondence 
		$$\Gamma_{*}:=\sum_{g\in \gS_{n}}(-1)^{\age(g)}{U^{g}}_{*}: H^{*}(A^{[n]})\to
		\bigoplus_{g\in \gS_{n}}H^{*-2\age(g)}\left((A^{n})^{g}\right)$$ coincides with
		the following inverse of the isomorphism $\Psi$ used in Fantechi--G\"ottsche
		\cite[Theorem 3.10]{MR1971293}
		\begin{eqnarray*}
			\Phi: H^{*}(A^{[n]}) &\to& \bigoplus_{g\in
				\gS_{n}}H^{*-2\age(g)}\left((A^{n})^{g}\right)\\
			\p_{\lambda_{1}}(\alpha_{1})\cdots\p_{\lambda_{l}}(\alpha_{l})\1 &\mapsto&
			n!\cdot\Sym(\alpha_{1}\times\cdots\times\alpha_{l}),
		\end{eqnarray*}
		Let us explain the notations from \cite{MR1971293} in the above formula\,:
		$\alpha_{1}, \ldots, \alpha_{l}\in H^{*}(A)$, $\times$ stands for the exterior
		product $\prod\pr_{i}^{*}(-)$, $\p$ is the Nakajima operator, $\1\in
		H^{0}(A^{[0]})\isom \Q$ is the fundamental class of the point,
		$\lambda=(\lambda_{1}, \ldots, \lambda_{l})$ is a partition of $n$, $g\in
		\gS_{n}$ is a permutation of type $\lambda$ with a numbering of orbits of $g$
		(as a permutation) chosen\,: $\{1, \ldots, l\}\lra{\sim}O(g)$, such that
		$\lambda_{j}$ is the length of the $j$-th orbit, then the class
		$\alpha_{1}\times\cdots\times\alpha_{l}$ is placed in the direct summand indexed
		by $g$ and $\Sym$ means the symmetrization operation $\frac{1}{n!}\sum_{h\in
			\gS_{n}}h$. Note that $\Sym(\alpha_{1}\times\cdots\times\alpha_{l})$ is
		independent of the choice of $g$, numbering \etc
		
		A repeated use of \eqref{eqn:recurrenceI} with $r=0$ and $m=1$, combined with
		the projection formula, yields that 
		$$\p_{\lambda_{1}}(\alpha_{1})\cdots\p_{\lambda_{l}}(\alpha_{l})\1=I^{*}_{0;\lambda_{1},
			\cdots,\lambda_{l}}(\alpha_{1}\times\cdots\times\alpha_{l})={U^{\lambda}}^{*}(\alpha_{1}\times\cdots\times\alpha_{l}),$$
		where the second equality comes from the definition and the irreducibility of
		$U^{\lambda}$ (\cf \cite[Remark 2.0.1]{MR1919155}). As a result, one only has to
		show that 
		\begin{equation}\label{eqn:composition}
		\sum_{g\in \gS_{n}}(-1)^{\age(g)}{U^{g}}_{*}\circ
		{U^{\lambda}}^{*}(\alpha_{1}\times\cdots\times\alpha_{l})=n!\cdot\Sym(\alpha_{1}\times\cdots\times\alpha_{l}).
		\end{equation}
		Indeed, for a given $g\in G$, if $g$ in not of type $\lambda$, then by
		\cite[Proposition 5.1.3]{MR1919155}, we know that
		${U^{g}}_{*}\circ{U^{\lambda}}^{*}=0$.  For any $g\in G$ of type $\lambda$, fix
		a numbering $\varphi: \{1, \dots, l\}\lra{\sim}O(g)$ such that
		$|\varphi(j)|=\lambda_{j}$ and let $\tilde\varphi:A^{\lambda}=A^{l}\to A^{O(g)}$
		be the induced isomorphism. Then denoting by $q: A^{\lambda}\surj A^{(\lambda)}$
		the quotient map by $\gS_{\lambda}$, the computation \cite[Proposition
		5.1.4]{MR1919155} implies that for such $g\in\lambda$,
		\begin{eqnarray*}
			&&{U^{g}}_{*}\circ{U^{\lambda}}^{*}(\alpha_{1}\times\cdots\times\alpha_{l})\\
			&=&\tilde\varphi_{*}\circ {U^{\lambda}}_{*}\circ
			{U^{\lambda}}^{*}(\alpha_{1}\times\cdots\times\alpha_{l})\\
			&=& m_\lambda \cdot \tilde\varphi_{*}\circ q^{*}\circ
			q_{*}(\alpha_{1}\times\cdots\times\alpha_{l})\\
			&=& m_\lambda\cdot\deg(q)\cdot \Sym(\alpha_{1}\times\cdots\times\alpha_{l})\\
			&=& m_\lambda \cdot |\gS_{\lambda}|\cdot
			\Sym(\alpha_{1}\times\cdots\times\alpha_{l}),
		\end{eqnarray*}
		where $m_\lambda = (-1)^{n-|\lambda|}\prod_{i=1}^{|\lambda|}\lambda_{i}$ as
		before.
		Putting those together, we have
		\begin{eqnarray*}
			&&\sum_{g\in
				\gS_{n}}(-1)^{\age(g)}{U^{g}}_{*}{U^{\lambda}}^{*}(\alpha_{1}\times\cdots\times\alpha_{l})\\
			&=&\sum_{g\in
				\lambda}(-1)^{n-|\lambda|}{U^{g}}_{*}{U^{\lambda}}^{*}(\alpha_{1}\times\cdots\times\alpha_{l})\\
			&=& \sum_{g\in \lambda}\left(\prod_{i=1}^{|\lambda|}\lambda_{i}\right)\cdot
			|\gS_{\lambda}|\cdot \Sym(\alpha_{1}\times\cdots\times\alpha_{l})\\
			&=& n!\cdot\Sym(\alpha_{1}\times\cdots\times\alpha_{l}),
		\end{eqnarray*}
		where the last equality is the orbit-stabilizer formula for the action of
		$\gS_n$ on itself by conjugation.
		The desired equality (\ref{eqn:composition}), hence also the Proposition, is
		proved. 
	\end{proof}
	As explained in \S \ref{sec:skeleton}, the proof of Theorem \ref{thm:mainAb} is
	now complete\,: Proposition \ref{prop:ZWsd} and Proposition
	\ref{prop:realisation-A} together imply that $\Sym(W)$ and $\Sym(Z)$ are
	rationally equivalent using Theorem \ref{thm:SD2}. Therefore Proposition
	\ref{prop:Multiplicativity} holds in our situation Case (A), which means exactly
	that the isomorphism $\phi$ in Proposition \ref{prop:addisom} (defined in
	(\ref{eqn:phi})) is also multiplicative with respect to the product structure on
	$\h\left(A^{[n]}\right)$ given by the small diagonal and the orbifold product with sign
	change by discrete torsion on $\h(A^{n}, \gS_{n})^{\gS_{n}}$.

	\section{Case (B)\,: Generalized Kummer varieties} \label{sec:proof-B}
	We prove Theorem \ref{thm:mainKummer} in this section. Notation is as in the
	beginning of \S\ref{sec:skeleton}\,: 
	$$M=A^{n+1}_{0}:=\Ker\left(A^{n+1}\lra{+} A\right)$$ which is non-canonically
	isomorphic to $A^{n}$, with the action of $G=\gS_{n+1}$ and the quotient
	$X:=A^{(n+1)}_{0}:=M/G$. Then the restriction of the Hilbert--Chow morphism to
	the \emph{generalized Kummer variety} $$K_{n}(A)=:Y\lra{f} A^{(n+1)}_{0}$$ is a
	symplectic resolution.

	\subsection{Step (i) -- Additive isomorphisms}\label{subsec:AddIso-B}
	We use the result in \cite{MR2067464} to establish an additive isomorphism
	$\h(Y)\lra{\isom}\h_{orb}([M/G])$.
	
	Recall that a morphism $f: Y\to X$ is called \emph{semi-small} if for all
	integer $k\geq 0$, the codimension of the locus $\left\{x\in X~|~\dim
	f^{-1}(x)\geq k\right\}$ is at least $2k$. In particular, $f$ is generically
	finite. Consider a (finite) Whitney stratification $X=\coprod_{a}X_{a}$ by
	connected strata, such that for any $a$, the restriction $f|_{f^{-1}(X_{a})}:
	f^{-1}(X_{a})\to X_{a}$ is a topological fiber bundle of fiber dimension
	$d_{a}$. Then the semismallness condition says that $\codim X_{a}\geq 2 d_{a}$
	for any $a$. In that case, a stratum $X_{a}$ is said to be \emph{relevant} if
	the equality holds\,: $\codim X_{a}=2d_{a}$.
	
	The result we need for Step (i) is de Cataldo--Migliorini \cite[Theorem
	1.0.1]{MR2067464}. Let us only state their theorem in the special case where all
	fibers over relevant strata are irreducible, which is enough for our purpose\,:
	
	\begin{thm}[\cite{MR2067464}]\label{thm:dCM}
		Let $f:Y\to X$ be a semi-small morphism of complex projective varieties with $Y$
		being smooth. 
		Suppose that all fibers over relevant strata are irreducible and that for each
		connected relevant stratum $X_{a}$ of codimension $2d_{a}$ (and fiber dimension
		$d_{a}$), the normalization $\bar Z_{a}$ of the closure $\bar X_{a}$ is
		projective and admits a stratification with strata being finite group quotients
		of smooth varieties.
		Then (the closure of) the incidence subvarieties between $X_{a}$ and $Y$ induce
		an isomorphism of Chow motives\,:
		$$\bigoplus_{a}\h\left(\bar Z_{a}\right)(-d_{a})\isom \h(Y).$$
		Moreover, the inverse isomorphism is again given by the incidence subvarieties
		but with different non-zero coefficients.
	\end{thm}
	
	\begin{rmks}
		\begin{itemize}
			\item The normalizations $\bar Z_{a}$ are singular, but they are $\Q$-varieties,
			for which the usual intersection theory works with rational coefficients (see
			Remark \ref{rmks:MotiveGeneral}).
			\item The statement about the correspondence inducing isomorphisms as well as
			the (non-zero) coefficients of the inverse correspondence is contained in
			\cite[\S 2.5]{MR2067464}.
			\item Since any symplectic resolution of a (singular) symplectic variety is
			semi-small, the previous theorem applies to the situation of Conjectures
			\ref{conj:MHRCbis} and \ref{conj:MHRCdt}.
			\item  Note that the correspondence in \cite{MR1919155} which is used in
			\S\ref{sec:proof-A} for Case (A) is a special case of Theorem \ref{thm:dCM}. 
			\item Theorem \ref{thm:dCM} is used in \cite{XuZeKummer} to deduce a motivic
			decomposition of generalized Kummer varieties equivalent to the Corollary
			\ref{cor:dCM-B} below.
		\end{itemize}
		
	\end{rmks}
	
	Let us start by making precise a Whitney stratification for the (semi-small)
	symplectic resolution $Y=K_{n}(A)\to X=A^{(n+1)}_{0}$. Notations are as in the
	proof of Proposition \ref{prop:addisom}. Let $\sP(n+1)$ be the set of partitions
	of $n+1$, then $$X=\coprod_{\lambda\in\sP(n+1)}X_{\lambda},$$ where the locally
	closed strata are defined by 
	$$X_{\lambda}:=\left\{\sum_{i=1}^{|\lambda|}\lambda_{i}[x_{i}]\in
	A^{(n+1)}~\middle\vert~
	\substack{\sum_{i=1}^{|\lambda|}\lambda_{i}x_{i}=0\\x_{i} \text{  distinct
	}}\right\},$$ with normalization of closure being $${\bar Z_{\lambda}}={\bar
		X_{\lambda}}^{norm}=A^{(\lambda)}_{0}:=A^{\lambda}_{0}/\gS_{\lambda}, $$ where 
	\begin{equation}\label{eqn:Alambda0}
	A^{\lambda}_{0}=\left\{(x_{1}, \ldots, x_{|\lambda|})\in A^{\lambda}
	~\middle\vert~ \sum_{i=1}^{|\lambda|}\lambda_{i}x_{i}=0\right\}.
	\end{equation}
	It is easy to see that $\dim X_{\lambda}=\dim A_{0}^{\lambda}=2(|\lambda|-1)$
	while the fibers over $X_{\lambda}$ are isomorphic to a product of Brian{\c c}on
	varieties (\cite{MR0457432}) $\prod_{i=1}^{|\lambda|}\mathfrak B_{\lambda_{i}}$,
	which is irreducible of dimension
	$\sum_{i=1}^{|\lambda|}(\lambda_{i}-1)=n+1-|\lambda|=\frac{1}{2}\codim
	X_{\lambda}$. 
	
	In conclusion, $f: K_{n}(A)\to A^{(n+1)}_{0}$ is a semi-small morphism with all
	strata being relevant and all fibers over strata being irreducible. One can
	therefore apply Theorem \ref{thm:dCM} to get the following 
	\begin{cor}\label{cor:dCM-B}
		For each $\lambda\in \sP(n+1)$, let $$V^{\lambda}:=\left\{(\xi, x_{1}, \ldots,
		x_{|\lambda|}) ~\middle\vert~
		\substack{\rho(\xi)=\sum_{i=1}^{|\lambda|}\lambda_{i}[x_{i}]\,;\\
			\sum_{i=1}^{|\lambda|}\lambda_{i}x_{i}=0}\right\}\subset K_{n}(A)\times
		A^{\lambda}_{0}$$ be the incidence subvariety, whose dimension is
		$n-1+|\lambda|$. Then the quotients
		$V^{(\lambda)}:=V^{\lambda}/\gS_{\lambda}\subset K_{n}(A)\times
		A^{(\lambda)}_{0}$ induce an isomorphism of rational Chow motives\,:
		$$\phi': \h\left(K_{n}(A)\right)\lra{\isom}
		\bigoplus_{\lambda\in\sP(n+1)}\h\left(A^{(\lambda)}_{0}\right)(|\lambda|-n-1).$$
		Moreover, the inverse $\psi':={\phi'}^{-1}$ is induced by $\sum_{\lambda\in
			\sP(n+1)}\frac{1}{m_{\lambda}}V^{(\lambda)}$, where
		$m_{\lambda}=(-1)^{n+1-|\lambda|}\prod_{i=1}^{|\lambda|}\lambda_{i}$ is a
		non-zero constant.
	\end{cor}
	
	Similarly to Proposition \ref{prop:addisom} for Case (A) , the previous
	Corollary \ref{cor:dCM-B} allows us to establish an additive isomorphism between
	$\h(K_{n}(A))$ and $\h_{orb}\left(\left[A_{0}^{n+1}/\gS_{n+1}\right]\right)$\,:
	\begin{prop}\label{prop:addisom-B}
		Let $M=A^{n+1}_{0}$ with the action of $G=\gS_{n+1}$. Let $p$ and $\iota$ denote
		the projection onto and the inclusion of the $G$-invariant part of $\h(M, G)$.
		For each $g\in G$, let 
		\begin{equation}\label{eqn:V}
		V^{g}:=(K_{n}(A)\times_{A^{(n+1)}_{0}}M^{g})_{red}\subset K_{n}(A)\times M^{g}
		\end{equation}
		be the incidence subvariety. Then they induce an isomorphism of rational Chow
		motives\,:
		$$\phi:=p\circ\sum_{g\in G}(-1)^{\age(g)}V^{g}:
		\h\left(K_{n}(A)\right)\lra{\isom} \left(\bigoplus_{g\in
			G}\h\left(M^{g}\right)(-\age(g))\right)^{G}.$$ Moreover, its inverse $\psi$ is
		given by $\frac{1}{(n+1)!}\cdot\sum_{g\in G}{}^{t}V^{g}\circ \iota$. 
	\end{prop}
	\begin{proof}
		The proof goes exactly as for Proposition \ref{prop:addisom}, with Lemma
		\ref{lemma:relate} replaced by the following canonical isomorphism\,:
		\begin{equation}\label{eq:canisoKum}
		\left(\bigoplus_{g\in
			\gS_{n+1}}\h\left((A^{n+1}_{0})^{g}\right)(-\age(g))\right)^{\gS_{n+1}}\lra{\isom}\bigoplus_{\lambda\in
			\sP(n+1)}\h\left(A_{0}^{(\lambda)}\right)(|\lambda|-n-1).
		\end{equation}
		Indeed, let $\lambda$ be the partition determined by $g$, then it is easy to
		compute $\age(g)=n+1-|O(g)|=n+1-|\lambda|$ and moreover the quotient of
		$(A^{n+1}_{0})^{g}$ by the centralizer of $g$, which is
		$\prod_{i=1}^{|\lambda|}\Z/\lambda_{i}\Z\rtimes \gS_{\lambda}$ with $\prod_{i=1}^{|\lambda|}\Z/\lambda_{i}\Z$ acting trivially, is exactly
		$A_{0}^{(\lambda)}$.
	\end{proof}
	
	To show Theorem \ref{thm:mainKummer}, it remains to show Proposition
	\ref{prop:Multiplicativity} in this situation (where all cycles $U$ are actually
	$V$ of Proposition \ref{prop:addisom-B}).  
	\subsection{Step (ii) -- Symmetrically distinguished cycles on abelian torsors
		with torsion structures}\label{subsec:SD-B}
	
	Observe that we have the extra technical difficulty that $(A^{n+1}_{0})^{g}$ is
	in general an extension of a finite abelian group by an abelian variety, thus
	non-connected. To deal with algebraic cycles on not necessarily connected
	`abelian varieties' in a canonical way as well as the property of being
	symmetrically distinguished, we introduce the following category. Roughly
	speaking, this is the category of \emph{abelian varieties with origin fixed only
		up to torsion}. It lies between the category of abelian varieties (with origin
	fixed) and the category of abelian torsors (\ie varieties isomorphic to an
	abelian variety, thus without a chosen origin).
	
	\begin{defi}[Abelian torsors with torsion structure]\label{def:atts}
		One defines the following category $\sA$.\\
		An object of $\sA$, called an \emph{abelian torsor with torsion structure}, or
		an \emph{\atts}, is a pair $(X, Q_{X})$ where $X$ is a connected smooth
		projective variety and $Q_{X}$ is a subset of $X$ such that there exists an
		isomorphism, as complex algebraic varieties, $f: X\to A$ from $X$ to an abelian
		variety $A$ which induces a bijection between $Q_{X}$ and $\Tor(A)$, the set of
		all torsion points of $A$. The point here is that the isomorphism $f$, called a
		\emph{marking}, usually being non-canonical in practice, is not part of the
		data of an \atts\\
		A morphism between two objects $(X, Q_{X})$ and $(Y, Q_{Y})$ is a morphism of
		complex algebraic varieties $\phi: X\to Y$ such that $\phi(Q_{X})\subset Q_{Y}$.
		Compositions of morphisms are defined in the natural way. Note that by choosing
		markings, a morphism between two objects in $\sA$ is essentially the composition
		of a morphism between two abelian varieties followed by a torsion translation.\\
		Denote by $\sAV$ the category of abelian varieties. Then there is a natural
		functor $\sAV\to \sA$ sending an abelian variety $A$ to $(A, \Tor(A))$.
	\end{defi}
	
	The following elementary lemma provides the kind of examples that we will be
	considering\,:
	\begin{lemma}[Constructing \atts and
		compatibility]\label{lemma:ConstructingATTS}
		Let $A$ be an abelian variety. Let $f: \Lambda\to \Lambda'$ be a morphism of
		lattices\footnote{A \emph{lattice} is a free abelian group of finite rank.} and
		$f_{A}: A\otimes_{\Z}\Lambda\to A\otimes_{\Z}\Lambda'$ be the induced morphism
		of abelian varieties. 
		\begin{enumerate}
			\item[(i)] Then $\Ker(f_{A})$ is canonically a disjoint union of \atts such that
			$Q_{\Ker(f_{A})}=\Ker(f_{A})\cap \Tor(A\otimes_{\Z}\Lambda)$.
			\item[(ii)] If one has another morphism of lattices $g: \Lambda'\to \Lambda''$
			inducing morphism of abelian varieties $g_{A}:A\otimes_{\Z}\Lambda'\to
			A\otimes_{\Z}\Lambda''$. Then the natural inclusion
			$\Ker(f_{A})\inj\Ker(g_{A}\circ f_{A})$ is a morphism of \atts (on each
			component).
		\end{enumerate}
	\end{lemma}
	\begin{proof}
		For (i), we have the following two short exact sequences of abelian groups\,:
		$$0\to \Ker(f)\to \Lambda \lra{\pi} \im(f)\to 0\,;$$
		$$0\to \im(f)\to \Lambda'\to \Coker(f)\to 0,$$
		with $\Ker(f)$ and $\im(f)$ being lattices.
		Tensoring them with $A$, one has exact sequences
		$$0\to A\otimes_{\Z}\Ker(f)\to A\otimes_{\Z}\Lambda\lra{\pi_{A}}
		A\otimes_{\Z}\im(f)\to 0\,;$$
		$$0\to \Tor^{\Z}\left(A, \Coker(f)\right)=:T\to A\otimes_{\Z}\im(f)\to
		A\otimes_{\Z}\Lambda',$$ where $T=\Tor^{\Z}\left(A, \Coker(f)\right)$ is a
		finite abelian group consisting of some torsion points of $A\otimes_{\Z}\im(f)$.
		Then $$\Ker(f_{A})=\pi_{A}^{-1}\left(T\right)$$
		is an extension of the finite abelian group $T$ by the abelian variety
		$A\otimes_{\Z}\Ker(f)$. 
		Choosing a section of $\pi$ makes $A\otimes_{\Z}\Lambda$ the product of
		$A\otimes_{\Z}\Ker(f)$ and $A\otimes_{\Z}\im(f)$, inside of which $\Ker(f_{A})$
		is the product of $A\otimes_{\Z}\Ker(f)$ and the finite subgroup $T$ of
		$A\otimes_{\Z}\im(f)$. This shows that $Q_{\Ker(f_{A})}:=\Ker(f_{A})\cap
		\Tor(A\otimes_{\Z}\Lambda)$, which is independent of the choice of the section,
		makes the connected components of $\Ker(f_{A})$ , the fibers over $T$,
		\atts's.\\
		With (i) being proved, (ii) is trivial\,: the torsion structures on
		$\Ker(f_{A})$ and on $\Ker(g_{A}\circ f_{A})$ are both defined by claiming that
		a point is torsion if it is a torsion point in $A\otimes_{\Z}\Lambda$.
	\end{proof}

	Before generalizing the notion of symmetrically distinguished cycles to the new
	category $\sA$, we have to first prove the following well-known fact. 
	\begin{lemma}\label{lemma:TorsionTranslation}
		Let $A$ be an abelian variety, $x\in \Tor(A)$ be a torsion point. Then the
		corresponding torsion translation 
		\begin{align*}
		t_{x}: \ A&\to A\\
		\ y&\mapsto x+y
		\end{align*}
		acts trivially on $\CH(A)$.
	\end{lemma}
	\begin{proof}
		The following proof, which we reproduce for the sake of completeness, is taken
		from  \cite[Lemma 2.1]{MR3650385}. Let $m$ be the order of $x$. Let
		$\Gamma_{t_{x}}$ be the graph of $t_{x}$, then one has
		$\mathbf{m}^{*}(\Gamma_{t_{x}})=\mathbf{m}^{*}(\Delta_{A})$ in $\CH(A\times A)$,
		where $\mathbf{m}$ is the multiplication by $m$ map of $A\times A$. However,
		$\mathbf{m}^{*}$ is an isomorphism of $\CH(A\times A)$ by Beauville's
		decomposition \cite{MR826463}. We conclude that $\Gamma_{t_{x}}=\Delta_{A}$,
		hence the induced correspondences are the same, which are $t_{x}^{*}$ and the
		identity respectively. 
	\end{proof}

	\begin{defi}[Symmetrically distinguished cycles in $\sA$]\label{def:SD2}
		Given an \atts $(X,Q_{X})\in \sA$ (see Definition \ref{def:atts}), an algebraic
		cycle $\gamma\in \CH(X)$ is called \emph{symmetrically distinguished}, if for a
		marking $f: X\to A$, the cycle $f_{*}(\gamma)\in CH(A)$ is symmetrically
		distinguished in the sense of O'Sullivan (Definition \ref{def:SD}). By Lemma
		\ref{lemma:TorsionTranslation}, this definition is independent of the choice of
		marking. An algebraic cycle on a disjoint union of \atts is symmetrically
		distinguished if it is so on each component. We denote $\CH(X)_{sd}$ the
		subspace consisting of symmetrically distinguished cycles.
	\end{defi}
	
	The following proposition is clear from Theorem \ref{thm:SD} and Theorem
	\ref{thm:SD2}.
	\begin{prop}\label{prop:SDATTS}
		Let $(X,Q_{X})\in \Obj(\sA)$ be an \atts.
		\begin{enumerate}
			\item[(i)] The space of symmetric distinguished cycles $\CH^{*}(X)_{sd}$ is a
			graded sub-$\Q$-algebra of $\CH^{*}(X)$.
			\item[(ii)] Let $f:(X,Q_{X}) \to (Y, Q_{Y})$ be a morphism in $\sA$, then
			$f_{*}:\CH(X)\to \CH(Y)$ and $f^{*}:\CH(Y)\to \CH(X)$ preserve symmetrically
			distinguished cycles.
			\item[(iii)] The composition $\CH(X)_{sd}\inj \CH(X)\surj \bar\CH(X)$ is an
			isomorphism. In particular, a (polynomial of) symmetrically distinguished cycles
			is trivial in $\CH(X)$ if and only if it is numerically trivial. 
		\end{enumerate}
	\end{prop}
	
	We will need the following easy fact to prove  that some cycles on an \atts are
	symmetrically distinguished by checking it in an ambient abelian variety.
	\begin{lemma}\label{lemma:SDonATTS}
		Let $i: B\inj A$ be a morphism of \atts which is a closed immersion. Let
		$\gamma\in \CH(B)$ be an algebraic cycle. Then $\gamma$ is symmetrically
		distinguished in $B$ if and only if $i_{*}(\gamma)$ is so in $A$.
	\end{lemma}
	\begin{proof}
		One implication is clear from Proposition \ref{prop:SDATTS}~(ii). For the other
		one, assuming $i_{*}(\gamma)$ is symmetrically distinguished in $A$.
		By choosing markings, one can suppose that $A$ is an abelian variety and $B$ is
		a torsion translation by $\tau\in \Tor(A)$ of a sub-abelian variety of $A$.
		Thanks to Lemma \ref{lemma:TorsionTranslation}, changing the origin of $A$ to
		$\tau$ does not change the cycle class $i_{*}(\gamma)\in \CH(A)$, hence one can
		further assume that $B$ is a sub-abelian variety of $A$. By Poincar\'e
		reducibility, there is a sub-abelian variety $C\subset A$, such that the natural
		morphism $\pi: B\times C\to A$ is an isogeny. We have the following diagram\,:
		\begin{displaymath}
		\xymatrix{
			& B\times C \ar@/_/[dl]_{\pr_{1}}\ar[d]^{\pi}\\
			B\ar@/_/[ur]_{j} \ar[r]_{i} & A
		}
		\end{displaymath}
		As $\pi^{*}:\CH(A)\to \CH(B\times C)$ is an isomorphism with inverse
		$\frac{1}{\deg(\pi)}\pi_{*}$, we have 
		$$\gamma={\pr_{1}}_{*}\circ
		j_{*}(\gamma)={\pr_{1}}_{*}\circ\pi^{*}\circ\frac{1}{\deg(\pi)}\pi_{*} \circ
		j_{*}(\gamma)=\frac{1}{\deg(\pi)}{\pr_{1}}_{*}\circ\pi^{*}\circ i_{*}(\gamma).$$
		Since $\pi$ and $\pr_{1}$ are morphisms of abelian varieties, the hypothesis
		that $i_{*}(\gamma)$ is symmetrically distinguished implies that $\gamma$ is
		also symmetrically distinguished by Proposition \ref{prop:SDATTS}~(ii).
	\end{proof}
	
	We now turn to the proof of Proposition \ref{prop:Multiplicativity} in Case (B),
	which takes the following form. As is explained in \S\ref{sec:skeleton}, with
	Step (i) being done (Proposition \ref{prop:addisom-B}),  this would finish the
	proof of Theorem \ref{thm:mainKummer}.
	
	\begin{prop}[=Proposition \ref{prop:Multiplicativity} in Case
		(B)]\label{prop:multiplicativity-B}
		In $\CH\left(\left(\coprod_{g\in G}M^{g}\right)^{3}\right)$, the symmetrizations
		of the following two algebraic cycles are rationally equivalent\,: 
		\begin{itemize}
			\item
			$W:=\left(\frac{1}{|G|}\sum_{g}V^{g}\times\frac{1}{|G|}\sum_{g}V^{g}\times\sum_{g}(-1)^{\age(g)}V^{g}\right)_{*}\left(\delta_{K_{n}(A)}\right)$\,;
			\item $Z$ is the cycle determining the orbifold product (Definition
			\ref{def:OrbMot}(v)) with the sign change by discrete torsion (Definition
			\ref{def:dt})\,: 
			$$Z|_{M^{g_{1}}\times M^{g_{2}}\times M^{g_{3}}}=\begin{cases}
			0 &\text{     if     } g_{3}\neq g_{1}g_{2}\\
			(-1)^{\epsilon(g_{1},g_{2})}\cdot\delta_{*}c_{top}(F_{g_{1}, g_{2}}) &\text{    
				if    } g_{3}=g_{1}g_{2}.
			\end{cases}
			$$
		\end{itemize}
	\end{prop}

	\noindent To this end, we apply Proposition \ref{prop:SDATTS}~(iii) by proving in
	this subsection that they are both symmetrically distinguished (Proposition
	\ref{prop:ZWsd-B}) and then verifying in the next one
	\S\ref{subsec:Realization-B} that they are homologically equivalent (Proposition
	\ref{prop:realisation-B}). 
	
	Let $M$ be the abelian variety $A^{n+1}_{0}=\left\{(x_{1}, \cdots, x_{n+1})\in
	A^{n+1}~\middle\vert~ \sum_{i}x_{i}=0\right\}$ as before. For any $g\in G$, the
	fixed locus $$M^{g}=\left\{(x_{1}, \cdots, x_{n+1})\in A^{n+1}~\middle\vert~
	\sum_{i}x_{i}=0\,;~ x_{i}=x_{g.i}~ \forall i\right\}$$ has the following 
	decomposition into connected components\,: 
	\begin{equation}\label{eqn:Mgcomponents}
	M^{g}=\coprod_{\tau\in A[d]}M^{g}_{\tau},
	\end{equation}
	where $d:=\gcd(g)$ is the greatest common divisor of the lengths of orbits of
	the permutation $g$, $A[d]$ is the set of $d$-torsion points and the connected
	component $M^{g}_{\tau}$ is described as follows.\\
	Let $\lambda\in\sP(n+1)$ be the partition determined by $g$ and $l:=|\lambda|$
	be its length. Choose a numbering $\varphi:\{1, \cdots, l\} \lra{\isom} O(g)$ of
	orbits such that $|\varphi(i)|=\lambda_{i}$. Then $d=\gcd(\lambda_{1},\cdots,
	\lambda_{l})$ and $\varphi$ induces an isomorphism
	\begin{equation}\label{eqn:isom}
	\tilde\varphi: A_{0}^{\lambda}\lra{\isom} M^{g},
	\end{equation}
	sending $(x_{1}, \cdots, x_{l})$ to $(y_{1}, \cdots, y_{n+1})$ with
	$y_{j}=x_{i}$ if $j\in \varphi(i)$. Here $A_{0}^{\lambda}$ is defined in
	(\ref{eqn:Alambda0}), which has obviously the following decomposition into
	connected components\,:
	\begin{equation}\label{eqn:Alambda0components}
	A^{\lambda}_{0}=\coprod_{\tau\in A[d]}A^{\lambda/d}_{\tau},
	\end{equation}
	where $$A^{\lambda/d}_{\tau}=\left\{(x_{1}, \cdots, x_{l})\in
	A^{\lambda}~\middle\vert~
	\sum_{i=1}^{l}\frac{\lambda_{i}}{d}x_{i}=\tau\right\}$$ is connected
	(non-canonically isomorphic to $A^{l-1}$ as varieties) and is equipped with a
	canonical \atts (Definition \ref{def:atts}) structure, namely, a point of
	$A^{\lambda/d}_{\tau}$ is defined to be of torsion (\ie in
	$Q_{A^{\lambda/d}_{\tau}}$) if and only if it is a torsion point (in the usual
	sense) in the abelian variety $A^{\lambda}$. The decomposition
	(\ref{eqn:Mgcomponents}) of $M^{g}$ is the transportation of the decomposition
	(\ref{eqn:Alambda0components}) of $A_{0}^{\lambda}$ via the isomorphism
	(\ref{eqn:isom})\,:
	$A^{\lambda/d}_{\tau}\lra[\isom]{\tilde\varphi}M^{g}_{\tau}$. The component
	$M^{g}_{\tau}$ hence acquires a canonical structure of \atts It is clear that
	the decomposition (\ref{eqn:Mgcomponents})  and the \atts structure on
	components are both independent of the choice of $\varphi$. One can also define
	the \atts structure on $M^{g}$ by using Lemma \ref{lemma:ConstructingATTS}.
	
	Similar to Proposition \ref{prop:ZWsd}, here is the main result of this
	subsection\,:
	
	\begin{prop}\label{prop:ZWsd-B}
		Notation is as in Proposition \ref{prop:multiplicativity-B}.
		$W$ and $Z$, as well as their symmetrizations, are symmetrically distinguished
		in $\CH\left(\left(\coprod_{g\in G}M^{g}\right)^{3}\right)$, where $M^{g}$ is
		viewed as a disjoint union of \atts as in (\ref{eqn:Mgcomponents}) and
		symmetrical distinguishedness is in the sense of Definition \ref{def:SD2}.
	\end{prop}
	\begin{proof}

		For $W$, it is enough to show that for any $g_{1}, g_{2}, g_{3}\in G$, 
		$q_{*}\circ p^{*}\circ\delta_{*}(\1_{K_{n}(A)})$ is symmetrically distinguished,
		where the notation is explained in the following commutative diagram, whose
		squares are all cartesian and without excess intersections.
		\begin{equation}\label{eqn:diagram}
		\xymatrix{
			& (A^{[n+1]})^{3} \cart& U^{g_{1}}\times U^{g_{2}}\times U^{g_{3}}
			\cart\ar[l]_-{p''}\ar[r]^-{q''} &
			(A^{n+1})^{g_{1}}\times(A^{n+1})^{g_{2}}\times(A^{n+1})^{g_{3}}\\
			A^{[n+1]}\cart \ar@{^{(}->}[ur]^{\delta''} \ar@{^{(}->}[r]^{\delta'} &
			(A^{[n+1]})^{3/A}\cart \ar@{^{(}->}[u]& U^{g_{1}}\times_{A} U^{g_{2}}\times_{A}
			U^{g_{3}} \cart\ar[l]_-{p'}\ar[r]^-{q'} \ar@{^{(}->}[u] &
			(A^{n+1})^{g_{1}}\times_{A}(A^{n+1})^{g_{2}}\times_{A}(A^{n+1})^{g_{3}}\ar@{^{(}->}[u]_{j}\\
			K_{n}(A)\ar@{^{(}->}[u] \ar@{^{(}->}[r]^{\delta} & K_{n}(A)^{3} \ar@{^{(}->}[u]
			& V^{g_{1}}\times V^{g_{2}}\times V^{g_{3}} \ar@{^{(}->}[u]
			\ar[l]_-{p}\ar[r]^-{q} & M^{g_{1}}\times M^{g_{2}}\times M^{g_{3}}
			\ar@{^{(}->}[u]_{i}
		}
		\end{equation}
		where the incidence subvarieties $U^{g}$'s are defined in
		\S\ref{subsec:AddIso-A} (\ref{eqn:U}) (with $n$ replaced by $n+1$)\,; all fiber
		products in the second row are over $A$\,; the second row is the base change by
		the inclusion of small diagonal $A\inj A^{3}$ of the first row\,; the third row
		is the base change by $O_{A}\inj A$ of the second the row\,; finally, $\delta,
		\delta', \delta''$ are various (absolute or relative) small diagonals.
		
		Observe that the two inclusions $i$ and $j$ are in the situation of Lemma
		\ref{lemma:ConstructingATTS}\,: let
		$$\Lambda:=\Z^{O(g_{1})}\oplus\Z^{O(g_{2})}\oplus\Z^{O(g_{3})},$$ which admits a
		natural morphism $u$ to $\Lambda':=\Z\oplus\Z\oplus\Z$ by weighted sum on each
		factor (with weights being the lengths of orbits). Let $v: \Lambda' \to
		\Lambda'':=\Z\oplus\Z$ be $(m_{1}, m_{2}, m_{3})\mapsto (m_{1}-m_{2},
		m_{1}-m_{3})$. Then it is clear that $i$ and $j$ are identified with the
		following inclusions $$\Ker(u_{A})\xhookrightarrow{i} \Ker(v_{A}\circ
		u_{A})\xhookrightarrow{j} A\otimes_{\Z}\Lambda.$$
		By Lemma \ref{lemma:ConstructingATTS},
		$(A^{n+1})^{g_{1}}\times_{A}(A^{n+1})^{g_{2}}\times_{A}(A^{n+1})^{g_{3}}$ and
		$M^{g_{1}}\times M^{g_{2}}\times M^{g_{3}}$ are naturally disjoint unions of
		\atts and the inclusions $i$ and $j$ are morphisms of \atts on each component.

		Now by functorialities and the base change formula (\cf \cite[Theorem
		6.2]{MR1644323}), we have $$j_{*}\circ q'_{*}\circ
		p'^{*}\circ\delta'_{*}(\1_{A^{[n+1]}})=q''_{*}\circ
		p''^{*}\circ\delta''_{*}(\1_{A^{[n+1]}}),$$
		which is a polynomial of big diagonals of $A^{|O(g_{1})|+|O(g_{2})|+|O(g_{3})|}$
		by Voisin's result \cite[Proposition 5.6]{MR3447105}, thus symmetrically
		distinguished in particular. By Lemma \ref{lemma:SDonATTS}, $q'_{*}\circ
		p'^{*}\circ\delta'_{*}(\1_{A^{[n+1]}})$ is symmetrically distinguished on each
		component of
		$(A^{n+1})^{g_{1}}\times_{A}(A^{n+1})^{g_{2}}\times_{A}(A^{n+1})^{g_{3}}$.\\
		Again by functorialities and the base change formula,  we have $$q_{*}\circ
		p^{*}\circ\delta_{*}(\1_{K_{n}(A)})=i^{*}\circ q'_{*}\circ
		p'^{*}\circ\delta'_{*}(\1_{A^{[n+1]}}).$$ Since $i$ is a morphism of \atts on
		each component (Lemma \ref{lemma:ConstructingATTS}), one concludes that
		$q_{*}\circ p^{*}\circ\delta_{*}(\1_{K_{n}(A)})$ is symmetrically distinguished
		on each component. Hence $W$, being a linear combination of such cycles, is also
		symmetrically distinguished.\\
		For $Z$, as in the Case (A), it is easy to see that all the obstruction bundles
		$F_{g_{1}, g_{2}}$ are (at least virtually) trivial vector bundles because
		according to Definition \ref{def:OrbMot}, there are only tangent/normal bundles
		of/between abelian varieties involved. Therefore the only non-zero case is the
		push-forward of the fundamental class of $M^{<g_{1}, g_{2}>}$ by the inclusion
		into $M^{g_{1}}\times M^{g_{2}}\times M^{g_{1}g_{2}}$, which is obviously
		symmetrically distinguished.
	\end{proof}

	\subsection{Step (iii) -- Cohomological
		realizations}\label{subsec:Realization-B}
	We keep the notation as before.
	To finish the proof of Proposition \ref{prop:multiplicativity-B}, hence Theorem
	\ref{thm:mainKummer}, it remains to show that the cohomology classes of the
	symmetrizations of $W$ and $Z$ are the same. In other words, they have the same
	realization for Betti cohomology.
	
	\begin{prop}\label{prop:realisation-B}
		The cohomology realization of the (\apriori additive) isomorphism in Proposition
		\ref{prop:addisom-B}
		$$\phi: \h(K_{n}(A))\lra{\isom} \left(\oplus_{g\in
			G}\h((A^{n+1}_{0})^{g})(-\age(g))\right)^{\gS_{n+1}}$$ 
		is an isomorphism of $\Q$-algebras $$\bar\phi: H^{*}(K_{n}(A))\lra{\isom}
		H^{*}_{orb,dt}([A^{n+1}_{0}/\gS_{n+1}])=\left(\bigoplus_{g\in
			\gS_{n+1}}H^{*-2\age(g)}((A^{n+1}_{0})^{g}), \star_{orb, dt}\right)^{\gS_{n}}.$$
		In other words, $\Sym(W)$ and $\Sym(Z)$ are homologically equivalent.
	\end{prop}
	\begin{proof}
		We use Nieper--Wi\ss kirchen's following description \cite{MR2578804} of the
		cohomology ring $H^{*}(K_{n}(A),\C)$. Let $s: A^{[n+1]}\to A$ be the composition
		of the Hilbert--Chow morphism followed by the summation map. Recall that $s$ is
		an isotrivial fibration. In the sequel, if not specified, all cohomology groups
		are with complex coefficients. We have a commutative diagram\,:
		\begin{equation*}
		\xymatrix{
			H^{*}(A) \ar[r]^{s^{*}} \ar[d]_{\epsilon} & H^{*}(A^{[n]})\ar[d]^{restr.}\\
			\C \ar[r] & H^{*}(K_{n}(A))
		}
		\end{equation*}
		where the upper arrow $s^{*}$ is the pull-back by $s$, the lower arrow is the
		unit map sending 1 to the fundamental class $\1_{K_{n}(A)}$, $\epsilon$ is the
		quotient by the ideal consisting of elements of strictly positive degree and the
		right arrow is the restriction map. The commutativity comes from the fact that
		$K_{n}(A)=s^{-1}(O_{A})$ is a fiber. Thus one has a ring homomorphism
		$$R: H^{*}(A^{[n]})\otimes_{H^{*}(A)} \C\to H^{*}(K_{n}(A)).$$
		Then \cite[Theorem 1.7]{MR2578804} asserts that this is an isomorphism of
		$\C$-algebras.\\
		Now consider the following diagram\,:
		\begin{equation}\label{diag:CaseAtoB}
		\xymatrix{
			H^{*}(A^{[n+1]})\otimes_{H^{*}(A)}\C \ar[r]^{R}_{\isom} \ar[d]_{\Phi}^{\isom} &
			H^{*}(K_{n}(A)) \ar[d]^{\bar\phi}_{\isom} \\
			\left(\oplus_{g\in
				\gS_{n+1}}H^{*-2\age(g)}((A^{n+1})^{g})\right)^{\gS_{n+1}}\otimes_{H^{*}(A)}\C
			\ar[r]_-{r} & \left(\oplus_{g\in
				\gS_{n+1}}H^{*-2\age(g)}((A_{0}^{n+1})^{g})\right)^{\gS_{n+1}},
		}
		\end{equation}
		\begin{itemize}
			\item As just stated, the upper arrow is an isomorphism of $\C$-algebras, by
			Nieper--Wi\ss kirchen \cite[Theorem 1.7]{MR2578804}.
			\item The left arrow $\Phi$ comes from the ring isomorphism (which is exactly
			CHRC \ref{conj:CHRC} for Case (A), see \S \ref{subsec:Realization-A})\,:
			$$H^{*}(A^{[n+1]})\lra{\isom} \left(\oplus_{g\in
				\gS_{n+1}}H^{*-2\age(g)}((A^{n+1})^{g})\right)^{\gS_{n+1}},$$
			established in \cite{MR1971293} based on \cite{MR1974889}. By (the proof of)
			Proposition \ref{prop:realisation-A}, this isomorphism is actually induced by
			$\sum_{g}(-1)^{\age(g)}\cdot U^{g}_{*}: H(A^{[n+1]})\to
			\oplus_{g}H((A^{n+1})^{g})$ with $U^{g}$ the incidence subvariety defined in
			(\ref{eqn:U}).
			Note that on the lower-left term of the diagram, the ring homomorphism
			$H^{*}(A)\to \left(\oplus_{g\in
				\gS_{n+1}}H^{*-2\age(g)}((A^{n+1})^{g})\right)^{\gS_{n+1}}$ lands in the summand
			indexed by $g=\id$, and the map $H^{*}(A)\to H^{*}(A^{n+1})^{\gS_{n+1}}$ is
			simply the pull-back by the summation map $A^{(n+1)}\to A$.
			\item The right arrow is the morphism $\bar\phi$ in question. It is already
			shown in Step (i) Proposition \ref{prop:addisom-B} to be an isomorphism of
			vector spaces. The goal is to show that it is also multiplicative.
			\item The lower arrow $r$ is defined as follows. On the one hand, let the image
			of the unit $1\in \C$ be the fundamental class of $A_{0}^{(n+1)}$ in the summand
			indexed by $g=\id$. On the other hand, for any $g\in \gS_{n+1}$, we have a
			natural restriction map $H^{*-2\age(g)}((A^{n+1})^{g})\to
			H^{*-2\age(g)}((A^{n+1}_{0})^{g})$. They will induce a ring homomorphism
			$H^{*}(A^{n+1}, \gS_{n+1})_{\C}\to H^{*}(A_{0}^{n+1}, \gS_{n+1})_{\C}$ by Lemma
			\ref{lemma:restriction} below, which is easily seen to be compatible with the
			$\gS_{n+1}$-action and the ring homomorphisms from $H^{*}(A)$, hence $r$ is a
			well-defined homomorphism of $\C$-algebras.
			\item To show the commutativity of the diagram (\ref{diag:CaseAtoB}), the case
			for the unit $1\in \C$ is easy to check. For the case of $H^{*}(A^{[n+1]})$, it
			suffices to remark that for any $g$ the following diagram is commutative
			\begin{equation*}
			\xymatrix{
				H^{*}(A^{[n+1]}) \ar[r]^{restr.}  \ar[d]_{U^{g}_{*}}  &
				H^{*}(K_{n}(A))\ar[d]_{V^{g}_{*}}\\
				H((A^{n+1})^{g}) \ar[r]_{restr} & H((A_{0}^{n+1})^{g}) 
			}
			\end{equation*}
			where $V^{g}$ is the incidence subvariety defined in (\ref{eqn:V}).
		\end{itemize}
		In conclusion, since in the commutative diagram (\ref{diag:CaseAtoB}), $\Phi, R$
		are isomorphisms of $\C$-algebras, $r$ is a homomorphism of $\C$-algebra and $
		\bar\phi$ is an isomorphism of vector spaces, we know that they are all
		isomorphisms of algebras. Thus Proposition \ref{prop:realisation-B} is proved
		assuming the following\,:
		
		\begin{lemma}\label{lemma:restriction}
			The natural restriction maps $H^{*-2\age(g)}((A^{n+1})^{g})\to
			H^{*-2\age(g)}((A^{n+1}_{0})^{g})$ for all $g\in \gS_{n+1}$ induce a ring
			homomorphism $H^{*}(A^{n+1}, \gS_{n+1})\to H^{*}(A_{0}^{n+1}, \gS_{n+1})$, where
			their product structures are given by the orbifold product (see Definition
			\ref{def:OrbMot} or \ref{def:OrbChow}).
		\end{lemma}
		
		\begin{proof}
			This is straightforward by definition. Indeed, for any $g_{1}, g_{2}\in
			\gS_{n+1}$ together with $\alpha\in H((A^{n+1})^{g_{1}})$ and $\beta\in
			H((A^{n+1})^{g_{2}})$, since the obstruction bundle $F_{g_{1}, g_{2}}$ is a
			trivial vector bundle, we have
			$$\alpha\star_{orb}\beta=\begin{cases}i_{*}\left(\alpha|_{(A^{n+1})^{<g_{1},
					g_{2}>}} \cup \beta|_{(A^{n+1})^{<g_{1}, g_{2}>}}\right)  &\text{     if     } 
			\rk F_{g_{1}, g_{2}}=0\\
			0 &\text{      if    } \rk F_{g_{1}, g_{2}}\neq 0
			\end{cases}
			$$
			where $i: (A^{n+1})^{<g_{1}, g_{2}>}\inj (A^{n+1})^{g_{1}g_{2}}$ is the natural
			inclusion. 
			Therefore by the base change for the cartesian diagram without excess
			intersection\,:
			\begin{equation*}
			\xymatrix{
				(A_{0}^{n+1})^{<g_{1}, g_{2}>} \ar@{^{(}->}[r]^{i_{0}} \ar@{^{(}->}[d] &
				(A_{0}^{n+1})^{g_{1}g_{2}}\ar@{^{(}->}[d]\\
				(A^{n+1})^{<g_{1}, g_{2}>} \ar@{^{(}->}[r]_{i}  & (A^{n+1})^{g_{1}g_{2}}
			}
			\end{equation*}
			we have\,: 
			\begin{eqnarray*}
				&&\alpha\star_{orb}\beta|_{(A^{n+1}_{0})^{g_{1}g_{2}}}\\
				&=& \begin{cases} {i_{0}}_{*}\left(\left(\alpha|_{(A^{n+1})^{<g_{1}, g_{2}>}}
					\cup \beta|_{(A^{n+1})^{<g_{1},
							g_{2}>}}\right)~\middle\vert~_{(A_{0}^{n+1})^{<g_{1}, g_{2}>}}\right)=
					{i_{0}}_{*}\left(\alpha|_{(A_{0}^{n+1})^{<g_{1}, g_{2}>}} \cup
					\beta|_{(A_{0}^{n+1})^{<g_{1}, g_{2}>}}\right)&\text{     if     }  \rk
					F_{g_{1}, g_{2}}=0\\
					0 &\text{      if    } \rk F_{g_{1}, g_{2}}\neq 0
				\end{cases}\\
				&=&   \alpha|_{(A_{0}^{n+1})^{g_{1}}} \star_{orb} \beta|_{(A_{0}^{n+1})^{g_{2}}}
					\end{eqnarray*}
			which means that the restriction map is a ring homomorphism.
		\end{proof}
		The proof of Proposition \ref{prop:realisation-B} is finished.
	\end{proof}
	
	Now the proof of Theorem \ref{thm:mainKummer} is complete\,: by Proposition
	\ref{prop:ZWsd-B} and Proposition \ref{prop:realisation-B}, we know that, thanks
	to Proposition \ref{prop:SDATTS}(iii), the symmetrizations of $Z$ and $W$ in
	Proposition \ref{prop:multiplicativity-B} are rationally equivalent, which
	proves Proposition \ref{prop:Multiplicativity} in Case (B). Hence the
	isomorphism $\phi$ in Proposition \ref{prop:addisom-B} is an isomorphism of
	algebra objects between the motive of the generalized Kummer variety
	$\h(K_{n}(A))$ and the orbifold Chow motive $\h_{orb}\left(\left[A^{n+1}_{0}/\gS_{n+1}\right]\right)$.
	\qed\\

	We would like to note the following corollary obtained by applying the
	cohomological realization functor to Theorem \ref{thm:mainKummer}.
	
	\begin{cor}[CHRC\,: Kummer case]
		The Cohomological HyperK\"ahler Resolution Conjecture is true for Case (B),
		namely, one has an isomorphism of $\Q$-algebras\,:
		$$H^{*}(K_{n}(A),\Q)\isom H^{*}_{orb,dt}\left([A^{n+1}_{0}/\gS_{n+1}]\right).$$
	\end{cor}
	\begin{rmk}\label{rmk:CHRCKummer}
		This result has never appeared in
		the literature. It is presumably not hard to check CHRC in the case of
		generalized Kummer varieties directly based on the cohomology result of
		Nieper--Wi\ss kirchen \cite{MR2578804}, which is of course one of the key
		ingredients used in our proof. It is also generally believed that the main
		result of Britze's Ph.D. thesis \cite{Britze} should also imply this result.
		However, the proof of its main result \cite[Theorem 40]{Britze} seems to be
		flawed\,: the linear map $\Theta$ constructed in the last line of Page 60, which
		is claimed to be the desired ring isomorphism, is actually the zero map.
		Nevertheless, the authors believe that it is feasible to check CHRC in this case
		with the very explicit description of the ring structure of
		$H^{*}(K_{n}(A)\times A)$ obtained in \cite{Britze}. 
	\end{rmk}

	\section{Application 1\,: Towards Beauville's splitting
		property}\label{sec:splitting}
	In this section, a \emph{holomorphic symplectic} variety is always assumed to be
	smooth projective unless stated otherwise and we require neither the simple
	connectedness nor the uniqueness up to scalar of the holomorphic symplectic
	2-form. Hence examples of holomorphic symplectic varieties include projective
	deformations of Hilbert schemes of K3 or abelian surfaces, generalized Kummer
	varieties \etc. 
	\subsection{Beauville's Splitting Property}
	
	Based on \cite{MR826463} and \cite{MR2047674}, Beauville envisages in
	\cite{MR2187148} the following \emph{Splitting Property} for all holomorphic
	symplectic varieties.
	\begin{conj}[Splitting Property\,: Chow rings]\label{conj:SplittingPrincipleCh}
		Let $X$ be a holomorphic symplectic variety of dimension $2n$. Then one has a
		canonical bigrading of the rational Chow ring $\CH^{*}(X)$, called
		\emph{multiplicative splitting of $\CH^{*}(X)$ of Bloch--Beilinson type}\,: for
		any $0\leq i\leq 4n$,
		\begin{equation}\label{eqn:decompCH}
		\CH^{i}(X)=\bigoplus_{s=0}^{i} \CH^{i}(X)_{s},
		\end{equation}
		which satisfies\,:
		\begin{itemize}
			\item (Multiplicativity) $\CH^{i}(X)_{s}\bullet \CH^{i'}(X)_{s'}\subset
			\CH^{i+i'}(X)_{s+s'}$\,;
			\item (Bloch--Beilinson) The associated ring filtration
			$F^{j}\CH^{i}(X):=\bigoplus_{s\geq j}\CH^{i}(X)_{s}$ satisfies the
			Bloch--Beilinson conjecture (\cf \cite[Conjecture 11.21]{MR1997577} for
			example). In particular\,:
			\begin{itemize}
				\item ($F^{1}=\CH_{hom}$) The restriction of the cycle class map $\cl:
				\bigoplus_{s>0}\CH^{i}(X)_{s}\to H^{2i}(X,\Q)$ is zero\,;
				\item (Injectivity) The restriction of the cycle class map $\cl:
				\CH^{i}(X)_{0}\to H^{2i}(X,\Q)$ is injective.
			\end{itemize}
		\end{itemize}
	\end{conj}
	
	We would like to reformulate (and slightly strengthen) Conjecture
	\ref{conj:SplittingPrincipleCh} by using the language of Chow motives as
	follows, which is, we believe, more fundamental. Let us first of all introduce
	the following  notion, which was introduced in \cite{MR3460114} and which avoids
	any mentions to the Bloch--Beilinson conjecture.
	\begin{defi}[Multiplicative Chow--K\"unneth decomposition]\label{def:MCK}
		Given a smooth projective variety $X$ of dimension $n$, a \emph{self-dual
			multiplicative Chow--K\"unneth decomposition} is a direct sum decomposition in
		the category $\CHM$ of Chow motives with rational coefficients\,:
		\begin{equation}\label{eqn:Decomp}
		\h(X)=\bigoplus_{i=0}^{2n}\h^{i}(X)
		\end{equation}
		satisfying the following two properties\,:
		\begin{itemize}
			\item (Chow--K\"unneth) The cohomology realization of the decomposition gives
			the K\"unneth decomposition\,: for each $0\leq i\leq 2n$,
			$H^{*}(\h^{i}(X))=H^{i}(X)$.
			\item (Self-duality) The dual motive $\h^i(X)^\vee$ identifies with
			$\h^{2n-i}(X)(n)$.
			\item (Multiplicativity) The product $\mu: \h(X)\otimes \h(X)\to \h(X)$ given by
			the small diagonal $\delta_{X}\subset X\times X\times X$ respects the
			decomposition\,: the restriction of $\mu$ on the summand
			$\h^{i}(X)\otimes\h^{j}(X)$ factorizes through $\h^{i+j}(X)$.
		\end{itemize}
		Such a decomposition induces a (multiplicative) bigrading of the rational Chow
		ring $\CH^{*}(X)=\oplus_{i,s}\CH^{i}(X)_{s}$ by setting\,:
		\begin{equation}\label{eqn:bigradingCH}
		\CH^{i}(X)_{s}:=\CH^{i}(\h^{2i-s}(X)):= \Hom_{\CHM}\left(\1(-i),
		\h^{2i-s}(X)\right).
		\end{equation}
		Conjecturally (\emph{cf.} \cite{MR1265533}), the associated ring filtration
		$F^{j}\CH^{i}(X):=\bigoplus_{s\geq j}\CH^{i}(X)_{s}$ satisfies the
		Bloch--Beilinson conjecture.
		
		By the definition of motives (\cf \ref{def:Mot}), a multiplicative
		Chow--K\"unneth decomposition is equivalent to a collection of
		self-correspondences $\left\{\pi^{0}, \cdots, \pi^{2\dim X}\right\}$, where
		$\pi^{i}\in \CH^{\dim X}(X\times X)$, satisfying
		\begin{itemize}
			\item $\pi^{i}\circ\pi^{i}=\pi^{i}, \forall i$\,;
			\item $\pi^{i}\circ\pi^{j}=0, \forall i\neq j$\,;
			\item $\pi^{0}+\cdots+\pi^{2\dim X}=\Delta_{X}$\,;
			\item $\im(\pi^{i}_{*}: H^{*}(X)\to H^{*}(X))=H^{i}(X)$\,;
			\item $\pi^{k}\circ\delta_{X}\circ (\pi^{i}\otimes\pi^{j})=0, \forall k\neq
			i+j$.
		\end{itemize}
		The induced multiplicative bigrading on the rational Chow ring $\CH^{*}(X)$ is
		given by $$\CH^{i}(X)_{s}:=\im\left(\pi^{2i-s}_{*}: \CH^{i}(X)\to
		\CH^{i}(X)\right).$$
		The above Chow--K\"unneth decomposition is self-dual if the transpose of $\pi^i$
		is equal to $\pi^{2\dim X-i}$.
	\end{defi}
	
	For later use, we need to generalize the previous notion for Chow motive
	algebras\,:
	\begin{defi}\label{def:MCKmotive}
		Let $\h$ be an (associative but not-necessarily commutative) algebra object in
		the category $\CHM$ of rational Chow motives. Denote by $\mu: \h\otimes\h\to \h$
		its multiplication structure. A \emph{multiplicative Chow--K\"unneth
			decomposition} of $\h$ is a direct sum decomposition 
		$$\h=\bigoplus_{i\in \Z} \h^{i},$$ such that 
		\begin{itemize}
			\item (Chow--K\"unneth) the cohomology realization gives the K\"unneth
			decomposition\,:  $H^{i}(\h)=H^{*}(\h^{i})$ for all $i\in \Z$\,;
			\item (Multiplicativity) the restriction of $\mu$ to $\h^{i}\otimes \h^{j}$
			factorizes through $\h^{i+j}$ for all $i, j\in \Z$.
		\end{itemize}
	\end{defi}

	Now one can enhance Conjecture \ref{conj:SplittingPrincipleCh} to the
	following\,:
	\begin{conj}[Motivic Splitting Property = Conjecture
		\ref{conj:MSPIntro}]\label{conj:SplittingPrincipleMot}
		Let $X$ be a holomorphic symplectic variety of dimension $2n$. Then we have a
		canonical (self-dual) multiplicative Chow--K\"unneth decomposition of $\h(X)$\,:
		\begin{equation*}
		\h(X)=\bigoplus_{i=0}^{4n}\h^{i}(X)
		\end{equation*}
		which is moreover \emph{of Bloch--Beilinson--Murre type}, that is, for any $i,
		j\in \N$,
		\begin{enumerate}
			\item[(i)] $\CH^{i}(\h^{j}(X))=0$ if $j<i$\,;
			\item[(ii)] $\CH^{i}(\h^{j}(X))=0$ if $j>2i$\,;
			\item[(iii)] the realization induces an injective map $\Hom_{\CHM}\left(\1(-i),
			\h^{2i}(X)\right)\to \Hom_{\Q-HS}\left(\Q(-i), H^{2i}(X)\right)$. 
		\end{enumerate}
	\end{conj}
	
	One can deduce Conjecture \ref{conj:SplittingPrincipleCh} from Conjecture
	\ref{conj:SplittingPrincipleMot} via (\ref{eqn:bigradingCH}). Note that the
	range of $s$ in (\ref{eqn:decompCH}) follows from the first two
	Bloch--Beilinson--Murre conditions in Conjecture
	\ref{conj:SplittingPrincipleMot}.
	
	\subsection{Splitting Property for abelian varieties}\label{subsec:AV}
	Recall that for an abelian variety $B$ of dimension $g$, using Fourier transform
	\cite{MR726428}, Beauville \cite{MR826463} constructs a multiplicative bigrading
	on $\CH^{*}(B)$\,:
	\begin{equation}
	\CH^{i}(B)=\bigoplus_{s=i-g}^{i}\CH^{i}(B)_{s}, ~\text{for any } 0\leq i\leq g
	\end{equation}
	where
	\begin{equation}\label{eqn:BeauvilleDecompAV}
	\CH^{i}(B)_{s}:=\left\{\alpha\in \CH^{i}(B) ~\middle\vert~
	\mathbf{m}^{*}\alpha=m^{2i-s}\alpha\,;~\forall m\in \Z\right\},
	\end{equation}
	is the simultaneous eigenspace for all $\mathbf{m}: B\to B$, the multiplication
	by $m\in \Z$ map.
	
	Using similar idea as in \loccit, Deninger and Murre \cite{MR1133323}
	constructed a multiplicative Chow--K\"unneth decomposition (Definition
	\ref{def:MCK})
	\begin{equation}\label{eqn:DeningerMurre}
	\h(B)=\bigoplus_{i=0}^{2g}\h^{i}(B),
	\end{equation}
	with (by \cite{MR1609325}) 
	\begin{equation}\label{eqn:Kings}
	\h^{i}(B)\isom\Sym^{i}(\h^{1}(B)).
	\end{equation}
	Moreover, one may choose such a multiplicative Chow--K\"unneth decomposition to
	be symmetrically distinguished\,; see \cite[Chapter~7]{MR3460114}.
	This Chow--K\"unneth decomposition induces, via (\ref{eqn:bigradingCH}),
	Beauville's bigrading (\ref{eqn:BeauvilleDecompAV}). That such a decomposition
	satisfies the Bloch--Beilinson condition is the following conjecture of
	Beauville \cite{MR726428} on $\CH^{*}(B)$ , which is still largely open.
	\begin{conj}[Beauville's conjecture on abelian
		varieties]\label{conj:BeauvilleAV}
		Notation is as above. Then
		\begin{itemize}
			\item $\CH^{i}(B)_{s}=0$ for $s<0$\,;
			\item The restriction of the cycle class map $\cl: \CH^{i}(B)_{0}\to
			H^{2i}(B,\Q)$ is injective.
		\end{itemize}
	\end{conj}
	
	\begin{rmk}\label{rmk:DecompATTS}
		As torsion translations  act trivially on the Chow rings of abelian varieties
		(Lemma \ref{lemma:TorsionTranslation}), the Beauville--Deninger--Murre
		decompositions (\ref{eqn:BeauvilleDecompAV}) and (\ref{eqn:DeningerMurre})
		naturally extend to the slightly broader context of \emph{abelian torsors with
			torsion structure} (see Definition \ref{def:atts}). 
	\end{rmk}

	We collect some facts about the Beauville--Deninger--Murre decomposition
	(\ref{eqn:DeningerMurre}) for the proof of Theorem \ref{thm:MCK} in the next
	subsection. By choosing markings for \atts's, thanks to Lemma
	\ref{lemma:TorsionTranslation}, we see that \atts's can be endowed with
	multiplicative Chow--K\"unneth decompositions consisting of Chow--K\"unneth
	projectors that are symmetrically distinguished, and  enjoying the properties
	embodied in the two following lemmas. Their proofs are reduced immediately to
	the case of abelian varieties, which are certainly well-known.
	
	\begin{lemma}[K\"unneth]\label{lemma:prodAV}
		Let $B$ and $B'$ be two abelian varieties (or more generally \atts's), then the
		natural isomorphism $\h(B)\otimes \h(B')\isom \h(B\times B')$ identifies the
		summand $\h^{i}(B)\otimes \h^{j}(B)$ as a direct summand of $\h^{i+j}(B\times
		B')$ for any $i, j\in\N$.\qed
	\end{lemma}

	\begin{lemma}\label{lemma:morphAV}
		Let $f: B\to B'$ be a morphism of abelian varieties (or more generally \atts's)
		of dimension $g$, $g'$ respectively. 
		\begin{itemize}
			\item The pull back $f^{*}:={}^{t}\Gamma_{f}: \h(B')\to \h(B)$ sends
			$\h^{i}(B')$ to $\h^{i}(B)$\,;
			\item The push forward $f_{*}:=\Gamma_{f}: \h(B)\to \h(B')$ sends $\h^{i}(B)$ to
			$\h^{i+2g'-2g}(B')$.\qed
		\end{itemize} 
	\end{lemma}

	\subsection{Candidate decompositions in Case (A) and (B)}
	In the sequel, let $A$ be an abelian surface and we consider the holomorphic
	symplectic variety $X$ which is either $A^{[n]}$ or $K_{n}(A)$. We construct a
	canonical Chow--K\"unneth decomposition of $X$  and show that it is self-dual
	and multiplicative. In Remark \ref{rmk:MurreKummer}, we observe that this
	decomposition can be expressed in terms of the Beauville--Deninger--Murre
	decomposition of the Chow motive of $A$, and as a consequence we note that
	Beauville's Conjecture \ref{conj:BeauvilleAV} for powers of $A$ implies the
	Bloch--Beilinson conjecture for $X$.
	
	Let us start with the existence of a self-dual multiplicative Chow--K\"unneth
	decomposition\,:
	\begin{thm}\label{thm:MCK}
		Given an abelian surface $A$, let $X$ be\\ Case (A)\,: the $2n$-dimensional
		Hilbert scheme $A^{[n]}$\,; or
		Case (B)\,: the $n$-th generalized Kummer variety $K_{n}(A)$.\\
		Then $X$ has a canonical self-dual multiplicative Chow--K\"unneth decomposition.
	\end{thm}
	\begin{rmk}\label{rmk:related}
		The existence of a self-dual multiplicative Chow--K\"unneth decomposition of
		$A^{[n]}$ is not new\,: it was previously obtained by Vial in \cite{VialHK}. 
		As for the generalized Kummer varieties, if one ignores the multiplicativity of
		the Chow--K\"unneth decomposition, which is of course the key point here, then
		it follows rather directly from De Cataldo and Migliorini's result
		\cite{MR2067464} as explained in \S\ref{subsec:AddIso-B} (see Corollary
		\ref{cor:dCM-B}) and is explicitly written down by Z. Xu \cite{XuZeKummer}. 
	\end{rmk}
	
	\begin{proof}[Proof of Theorem \ref{thm:MCK}]
		The following proof works for both cases. Let $M:=A^{n}$, $G:=\gS_{n}$,
		$X:=A^{[n]}$ in Case (A) and $M:=A^{n+1}_{0}$, $G:=\gS_{n+1}$, $X:=K_{n}(A)$ in
		Case (B). Thanks to Theorem \ref{thm:mainAb} and Theorem \ref{thm:mainKummer},
		we have an isomorphism of motive algebras\,:
		$$\h(X) \stackrel{\isom}{\longrightarrow} \left(\bigoplus_{g\in G}
		\h\left(M^{g}\right)(-\age(g)) , \star_{orb, dt}\right)^{G}$$ 
		whose inverse on each direct summand  $ \h\left(M^{g}\right)(-\age(g))$ is given
		by a rational multiple of the  transpose of the induced morphism $\h(X) \to
		\h\left(M^{g}\right)(-\age(g))$.
		It thus suffices to prove that each direct summand has a self-dual
		Chow--K\"unneth decomposition in the sense of Definition \ref{def:MCK}, and that
		the induced Chow--K\"unneth decomposition on the motive algebra
		$$\h:=\bigoplus_{g\in G} \h\left(M^{g}\right)(-\age(g)), \text{with }
		\star_{orb, dt} \text{as the product},$$ is multiplicative in the sense of
		Definition \ref{def:MCKmotive}. To this end, for each $g\in G$, an application
		of Deninger--Murre's decomposition (\ref{eqn:DeningerMurre}) to $M^{g}$, which
		is an abelian variety in Case (A) and a disjoint union of \atts in Case (B),
		gives us a self-dual multiplicative Chow--K\"unneth decomposition
		$$\h(M^{g})=\bigoplus_{i=0}^{2\dim M^{g}} \h^{i}(M^{g}).$$
		Now we define for each $i\in \N$, 
		\begin{equation}\label{eqn:hi}
		\h^{i}:=\bigoplus_{g\in G} \h^{i-2\age(g)}(M^{g})(-\age(g)).
		\end{equation}
		Here by convention, $\h^{j}(M^{g})=0$ for $j<0$, hence in (\ref{eqn:hi}),
		$\h^{i}=0$ if $i-2\age(g)>2\dim(M^{g})$ for any $g\in G$, that is, when
		$i>\max_{g\in G}\{4n-2\age(g)\}=4n$.
		
		Then obviously, as a direct sum of  Chow--K\"unneth decompositions,
		$$\h=\bigoplus_{i=0}^{4n}\h^{i}$$ is a  Chow--K\"unneth decomposition. It is
		self-dual because each $M^g$ has dimension $2n-2\age (g)$. It remains to show
		the multiplicativity condition that $\mu: \h^{i}\otimes \h^{j}\to \h$ factorizes
		through $\h^{i+j}$, which is equivalent to say that for any $i, j\in\N$ and $g,
		h\in G$, the orbifold product $\star_{orb}$ (discrete torsion only changes a
		sign thus irrelevant here)  restricted to the summand
		$\h^{i-2\age(g)}(M^{g})(-\age(g))\otimes\h^{j-2\age(h)}(M^{h})(-\age(h))$
		factorizes through $h^{i+j-2\age(gh)}(M^{gh})(-\age(gh))$. Thanks to the fact
		that the obstruction bundle $F_{g,h}$ is always a trivial vector bundle in both
		of our cases, we know that (see Definition \ref{def:OrbMot}) $\star_{orb}$ is
		either zero when $\rk(F_{g,h})\neq 0$\,; or when $\rk(F_{g,h})=0$, is defined 
		as the correspondence from $M^{g}\times M^{h}$ to $M^{gh}$ given by the
		following composition
		\begin{equation}\label{eqn:degree}
		\h(M^{g})\otimes\h(M^{h})\lra{\isom} \h(M^{g}\times
		M^{h})\lra{\iota_{1}^{*}}\h(M^{<g,h>})\lra{{\iota_{2}}_{*}}\h(M^{gh})(\codim(\iota_{2})),
		\end{equation}
		where 
		\begin{displaymath}
		\xymatrix{
			M^{gh} & M^{<g, h>} \ar@{_{(}->}[l]_{\iota_{2}} \ar@{^{(}->}[r]^{\iota_{1}} &
			M^{g}\times M^{h}
		}
		\end{displaymath}
		are morphisms of abelian varieties in Case (A) and morphisms of \atts's in Case
		(B). Therefore, one can suppose further that $\rk(F_{g,h})=0$, which implies by
		using (\ref{eqn:VirRkF}) that the Tate twists match\,:
		$$\codim(\iota_{2})-\age(g)-\age(h)=-\age(gh).$$
		Now Lemma \ref{lemma:prodAV} applied to the first isomorphism in
		(\ref{eqn:degree}) and Lemma \ref{lemma:morphAV} applied to the last two
		morphisms in (\ref{eqn:degree}) show that, omitting the Tate twists, the summand
		$\h^{i-2\age(g)}(M^{g})\otimes\h^{j-2\age(h)}(M^{h})$ is sent by $\mu$ inside
		the summand $h^{k}(M^{gh})$, with the index
		$$k=i+j-2\age(g)-2\age(h)+2\dim(M^{gh})-2\dim(M^{<g,h>})=i+j-2\age(gh),$$ where
		the last equality is by Equation (\ref{eqn:VirRkF}) together with the assumption
		$\rk(F_{g,h})=0$. \\
		In conclusion, we get a multiplicative Chow--K\"unneth decomposition
		$\h=\bigoplus_{i=0}^{4n}\h^{i}$ with $\h^{i}$ given in (\ref{eqn:hi})\,; hence a
		multiplicative Chow--K\"unneth decomposition for its $G$-invariant part of the
		sub-motive algebra $\h(X)$.
	\end{proof}

	The decomposition in Theorem \ref{thm:MCK} is supposed to be Beauville's
	splitting of the Bloch--Beilinson--Murre filtration on the rational Chow ring of
	$X$. In particular,
	\begin{conj}(Bloch--Beilinson for $X$)\label{conj:BB}
		Notation is as in Theorem \ref{thm:MCK}. Then for all $i\in \N$,
		\begin{itemize}
			\item $\CH^{i}(X)_{s}=0$ for $s<0$\,;
			\item The restriction of the cycle class map $\cl: \CH^{i}(X)_{0}\to
			H^{2i}(X,\Q)$ is injective.
		\end{itemize}
	\end{conj}
	As a first step towards this conjecture, let us make the following
	
	\begin{rmk}\label{rmk:MurreKummer}
		\emph{Beauville's conjecture \ref{conj:BeauvilleAV} on abelian varieties implies
			Conjecture \ref{conj:BB}. } 
		Indeed, keep the same notation as before. From \eqref{eqn:hi} (together with the
		canonical isomorphisms \eqref{eq:canisoHilb} and \eqref{eq:canisoKum}), we
		obtain 
		$$\CH^{i}(A^{[n]})_{s}=\CH^{i}(\h^{2i-s}(A^{[n]}))=\left(\bigoplus_{g\in\gS_{n}}\CH^{i-\age(g)}(\h^{2i-s-2\age(g)}(A^{O(g)}))\right)^{\gS_{n}}=\bigoplus_{\lambda\in
			\sP(n)}\CH^{i+|\lambda|-n}(A^{\lambda})^{\gS_{\lambda}}_{s}\,;$$
		$$\CH^{i}(K_{n}A)_{s}=\CH^{i}(\h^{2i-s}(K_{n}A))=\left(\bigoplus_{g\in\gS_{n+1}}\CH^{i-\age(g)}(\h^{2i-s-2\age(g)}(A_{0}^{O(g)}))\right)^{\gS_{n+1}}=\bigoplus_{\lambda\in
			\sP(n+1)}\CH^{i+|\lambda|-n-1}(A^{\lambda}_{0})^{\gS_{\lambda}}_{s},$$
		in two cases respectively, whose vanishing ($s<0$) and injectivity into
		cohomology by cycle class map ($s=0$) follow directly from those of
		$A^{\lambda}$ or $A^{\lambda}_{0}$.\\
		In fact, \cite[Theorem 3]{Vialabelian} proves more generally that the second
		point of Conjecture \ref{conj:BeauvilleAV} (the injectivity of the cycle class
		map $\cl: \CH^{i}(B)_{0}\to H^{2i}(B,\Q)$ for all complex abelian varieties)
		implies Conjecture \ref{conj:BB} for all smooth projective complex varieties $X$
		whose Chow motive is of abelian type, which is the case for a generalized Kummer
		variety by  Proposition \ref{prop:addisom-B}. Of course, one has to check that
		our definition of $\CH^{i}(X)_{0}$ here coincides with the one in
		\cite{Vialabelian}, which is quite straightforward.
	\end{rmk}

	The Chern classes of a (smooth) holomorphic symplectic variety $X$ are also
	supposed to be in $\CH^{i}(X)_{0}$ with respect to Beauville's conjectural
	splitting. We can indeed check this in both cases considered here\,:
	
	\begin{prop}\label{prop:ChernClass}
		Set-up as in Theorem \ref{thm:MCK}. The Chern class $c_i(X)$ belongs to
		$\CH^i(X)_0$ for all $i$.
	\end{prop}
	\begin{proof}
		In Case (A), that is, in the case where $X$ is the Hilbert scheme $A^{[n]}$,
		this is proved in \cite{VialHK}. Let us now focus on Case (B), that is, on the
		case where $X$ is the generalized Kummer variety $K_n(A)$. 
		Let $\{\pi^i : 0\leq i \leq 2n\}$ be the Chow--K\"unneth decomposition of
		$K_n(A)$ given by~\eqref{eqn:hi}. We have to show that $c_i(K_n(A)) =
		(\pi^{2i})_* c_i(K_n(A))$, or equivalently that $(\pi^{j})_* c_i(K_n(A)) = 0$ as
		soon as $(\pi^{j})_* c_i(K_n(A))$ is homologically trivial.
		By Proposition \ref{prop:addisom-B}, it suffices to show that for any $g\in G$
		$(\pi^{j}_{M^g})_*((V^g)_*c_i(K_n(A))) = 0$ as soon as
		$(\pi^j_{M^g})_*((V^g)_*c_i(K_n(A))) $ is homologically trivial. Here, recall
		that \eqref{eqn:Mgcomponents} makes $M^g$  a disjoint union of \atts and that
		$\pi^j_{M^g}$ is a Chow--K\"unneth projector on $M^g$ which is symmetrically
		distinguished on each component of $M^g$.
		By Proposition \ref{prop:SDATTS}, it is enough to show that $(V^g)_*(
		c_i(K_n(A))$ is symmetrically distinguished on each component of $M^g$.
		As in the proof of Proposition \ref{prop:ZWsd-B}, we have  for any $g\in G$ the
		following 
		commutative diagram, whose squares are cartesian and without excess
		intersections\,:
		\begin{equation}\label{eqn:diagram2}
		\xymatrix{ 
			A^{[n+1]}\cart & U^{g} \cart\ar[l]_-{p'}\ar[r]^-{q'} & (A^{n+1})^{g}\\
			K_{n}(A)\ar@{^{(}->}[u] & V^{g} \ar@{^{(}->}[u] \ar[l]_-{p}\ar[r]^-{q} & M^{g}
			\ar@{^{(}->}[u]_{i}
		}
		\end{equation}
		where the incidence subvariety $U^{g}$ is defined in \S\ref{subsec:AddIso-A}
		(\ref{eqn:U}) (with $n$ replaced by $n+1$) and  the bottom row is the base
		change by $O_{A}\inj A$ of the top row.
		Note that $c_i(K_n(A)) = c_i(A^{[n+1]})|_{K_n(A)}$, since the tangent bundle of
		$A$ is trivial. Therefore, by functorialities and the base change formula (\cf
		\cite[Theorem 6.2]{MR1644323}), we have $$ (V^g)_*( c_i(K_n(A)) := q_{*} \circ
		p^{*}( c_i(K_n(A)) = i^*\circ q'_{*} \circ p'^{*}(c_i({A^{[n+1]}})).$$
		By Voisin's result \cite[Theorem 5.12]{MR3447105}, $q'_{*} \circ
		p'^{*}(c_i({A^{[n+1]}}))$ is a polynomial of big diagonals of $A^{|O(g)|}$ ,
		thus symmetrically distinguished in particular. It follows from Proposition
		\ref{prop:SDATTS} that $(V^g)_*( c_i(K_n(A)) $ is symmetrically distinguished on
		each component of $M^g$. This concludes the proof of the proposition.
	\end{proof}

	\section{Application 2\,: Multiplicative decomposition theorem of rational
		cohomology}
	
	Deligne's decomposition theorem states the following\,:
	\begin{thm}[Deligne \cite{MR0244265}]
		Let $\pi : \mathcal{X} \rightarrow B$ be a smooth projective
		morphism. In the derived category of sheaves of $\Q$-vector spaces on $B$,
		there is a decomposition (which is non-canonical in general)
		\begin{equation} \label{eq deligne}
		R\pi_*\Q \cong \bigoplus_i R^i\pi_*\Q[-i].
		\end{equation}
	\end{thm}
	Both sides of \eqref{eq deligne} carry a cup-product\,: on the
	right-hand side the cup-product is the direct sum of the usual
	cup-products $R^i\pi_*\Q \otimes R^j\pi_*\Q \rightarrow R^{i+j}\pi_*\Q
	$ defined on local systems, while on the left-hand side the derived
	cup-product $R\pi_*\Q \otimes R\pi_*\Q \rightarrow R\pi_*\Q $ is induced by the
	(derived) action of the relative small diagonal $\delta \subset \mathcal{X}
	\times_B\mathcal{X} \times_B \mathcal{X}$ seen as a relative correspondence from
	$\mathcal{X} \times_B\mathcal{X} $ to $\mathcal{X}$. As explained in
	\cite{MR2916291}, the isomorphism \eqref{eq deligne} does not respect
	the cup-product in general. Given a family of smooth projective
	varieties $\pi : \mathcal{X} \rightarrow B$, Voisin \cite[Question
	0.2]{MR2916291} asked if there exists a decomposition as in \eqref{eq
		deligne} which is multiplicative, \emph{i.e.}, which is compatible
	with cup-product, maybe over a nonempty Zariski open subset of $B$. By
	Deninger--Murre \cite{MR1133323}, there does exist such
	a decomposition for an abelian scheme $\pi : \mathcal{A} \rightarrow
	B$. The main result of \cite{MR2916291} is\,:
	
	\begin{thm}[Voisin \cite{MR2916291}] \label{thm dec voisin} For any
		smooth projective family $\pi : \mathcal{X} \rightarrow B$ of K3
		surfaces, there exist a decomposition isomorphism as in \eqref{eq
			deligne} and a nonempty Zariski open subset $U$ of $B$, such that
		this decomposition becomes multiplicative for the restricted family
		$\pi|_U : \mathcal{X}|_U \rightarrow U$.
	\end{thm}
	
	As implicitly noted in \cite[Section 4]{VialHK}, Voisin's Theorem \ref{thm dec
		voisin} holds more generally for any smooth projective family $\pi : \mathcal{X}
	\rightarrow B$ whose generic fiber admits a multiplicative Chow--K\"unneth
	decomposition (K3 surfaces do have a multiplicative Chow--K\"unneth
	decomposition\,; this follows by suitably
	reinterpreting, as in \cite[Proposition 8.14]{MR3460114}, the vanishing of the
	modified
	diagonal cycle of Beauville--Voisin \cite{MR2047674} as the multiplicativity of
	the Beauville--Voisin Chow--K\"unneth decomposition.)\,:
	
	\begin{thm}\label{thm:dec} Let  $\pi : \mathcal{X} \rightarrow B$ be 
		a smooth projective family, and assume that the generic fiber $X$ of $\pi$
		admits a multiplicative Chow--K\"unneth decomposition. Then there exist a
		decomposition isomorphism as in \eqref{eq
			deligne} and a nonempty Zariski open subset $U$ of $B$, such that
		this decomposition becomes multiplicative for the restricted family
		$\pi|_U : \mathcal{X}|_U \rightarrow U$.
	\end{thm}
	\begin{proof} By
		spreading out a multiplicative Chow--K\"unneth decomposition of $X$, there
		exist a sufficiently small but
		nonempty Zariski open subset $U$ of $B$ and relative correspondences $\Pi^i \in
		\CH^{\dim_B\mathcal{X}}(\mathcal{X}|_U \times_U \mathcal{X}|_U)$,
		$0\leq i \leq 2\dim_B\mathcal{X}$, forming a relative Chow--K\"unneth
		decomposition, meaning that 
		$\Delta_{\mathcal{X}|_U/U} = \sum_i \Pi^i$, $\Pi^i \circ \Pi^i =
		\Pi^i$, $\Pi^i \circ \Pi^j = 0$ for $i\neq j$, and $\Pi^i$ acts as
		the identity on $R^i(\pi^{[n]}|_U)_*\Q$ and as zero on
		$R^j(\pi^{[n]}|_U)_*\Q$ for $j\neq i$. 
		By \cite[Lemma 2.1]{MR2916291}, the relative idempotents $\Pi^i$
		induce a decomposition in the derived category $$R(\pi|_U)_*\Q \cong
		\bigoplus_{i=0}^{4n} H^i(R(\pi|_U)_*\Q)[-i] = \bigoplus_{i=0}^{4n}
		R^i(\pi|_U)_*\Q[-i]$$ with the property that $\Pi^i$ acts as the identity
		on the summand $H^i(R(\pi|_U)_*\Q)[-i]$ and acts as zero on the
		summands $H^j(R(\pi|_U)_*\Q)[-j]$ for $j \neq i$.  In order to
		establish the existence of a decomposition as in \eqref{eq deligne} that
		is multiplicative and hence to conclude the proof of the theorem, we thus have
		to show that $\Pi^k \circ \delta \circ (\Pi^i \times \Pi^j)$ acts as zero on
		$R(\pi|_U)_*\Q \otimes R(\pi|_U)_*\Q$, after possibly further shrinking $U$,
		whenever $k\neq i+j$. But more is true\,: being generically multiplicative, the
		relative Chow--K\"unneth decomposition $\{\Pi^i\}$ is multiplicative, that is,
		$\Pi^k \circ \delta \circ (\Pi^i \times \Pi^j) = 0$  whenever $k\neq i+j$, after
		further shrinking $U$ if necessary. The theorem is now proved.
	\end{proof}

	As a corollary, we can extend Theorem
	\ref{thm dec voisin} to families of generalized Kummer varieties\,:
	\begin{cor}\label{cor:decompKummer}
		Let $\pi : \mathcal{A} \rightarrow B$ be
		an abelian surface over $B$. Consider Case (A)\,: $ \mathcal{A}^{[n]}
		\rightarrow B$
		the relative Hilbert scheme of length-$n$ subschemes on $\mathcal{A}
		\rightarrow B$\,; or Case (B)\,: $K_n(\mathcal{A}) \rightarrow B$ the relative
		generalized Kummer variety. Then, in both cases, there exist a decomposition
		isomorphism  as in \eqref{eq
			deligne} and a nonempty Zariski open subset $U$ of $B$, such that
		this decomposition becomes multiplicative for the restricted family
		over $U$.
	\end{cor}
	\begin{proof}
		The generic fiber of $\mathcal{A}^{[n]} \rightarrow B$
		(\emph{resp.}~$K_n(\mathcal{A}) \rightarrow B$) is the $2n$-dimensional Hilbert
		scheme  (\emph{resp.}~generalized Kummer variety) attached to the abelian
		surface that is the generic fiber of $\pi$. By Theorem \ref{thm:MCK}, it admits
		a multiplicative Chow--K\"unneth decomposition. (Strictly speaking, we only
		established Theorem \ref{thm:MCK} for Hilbert schemes of abelian surfaces and
		generalized Kummer varieties over the complex numbers\,; however, the proof carries through over any base field of characteristic zero.) We conclude by invoking
		Theorem \ref{thm:dec}. 
	\end{proof}

	\bibliographystyle{amsplain}
	\bibliography{biblio_fulie}

\end{document}